\documentclass[11pt]{amsart}
\usepackage[UKenglish]{isodate}
\usepackage{amsmath}
\usepackage{amsthm}
\usepackage{amsbsy}
\usepackage{amssymb}
\usepackage{amsfonts}
\usepackage[dvips]{lscape}
\usepackage{xcolor}
\usepackage{amscd}
\usepackage[all,cmtip]{xy}
\usepackage{euscript}
\usepackage{parskip}
\usepackage{enumerate}
\usepackage{yfonts}
\usepackage{xcolor}
\usepackage{graphicx}
\usepackage{fancyhdr}

\usepackage[margin=1in]{geometry}
\usepackage[expansion=false]{microtype}
\usepackage[pdfusetitle,unicode,hidelinks]{hyperref}

\theoremstyle{plain}

\newtheorem{theorem}{Theorem}[section]

\newtheorem{proposition}[theorem]{Proposition}
\newtheorem{lemma}[theorem]{Lemma}
\newtheorem{corollary}[theorem]{Corollary}
\newtheorem{remark}[equation]{Remark}
\newtheorem{definition}[theorem]{Definition}

\newtheorem{example}[theorem]{Example}
\newtheorem{conjecture}[theorem]{Conjecture}
\newtheorem{question}[theorem]{Question}

\newcommand{\cal}{\EuScript}
\newcommand{\Spec}{\textup{Spec }}

\newcommand{\Sm}{\cal{S}\textup{m}}

\newcommand{\Map}{\textup{Map}}
\newcommand{\et}{\textup{\'et}}

\newcommand{\Hom}{\textup{Hom}}

\newcommand{\D}{\textup{D}}
\newcommand{\DM}{\textup{DM}}
\newcommand{\Cat}{\cal{C}\textup{at}_{\infty}}

\newcommand{\Sch}{\cal{S}\textup{ch}}

\newcommand{\Var}{\textup{Var}}
\let\lim=\relax
\DeclareMathOperator*{\colim}{colim}
\DeclareMathOperator*{\lim}{lim}
\newcommand{\ord}{\textup{ord}}
\newcommand{\im}{\textup{im}}
\newcommand{\coker}{\textup{coker}}

\newcommand{\vdim}{\textup{vdim}}

\newcommand{\dcat}{D^b\textup{Coh}}
\newcommand{\Num}{\textup{Num}}
\newcommand{\Vect}{\textup{Vec}}
\newcommand{\Chow}{\textup{Chow}}
\renewcommand{\contentsname}{}
\newcommand{\imK}{\textup{im}_K}
\newcommand{\imG}{\textup{im}_G}
\newcommand{\gm}{\textup{gm}}
\newcommand{\conv}{\textup{conv}}
\newcommand{\eff}{\textup{eff}}
\newcommand{\mot}{\textup{mot}}
\newcommand{\Hdg}{\textup{Hdg}}
\begin{document}
\title{Integration of Voevodsky Motives}
\author{Masoud Zargar}
\renewcommand{\contentsname}{}

\begin{abstract}
In this paper, we construct four different theories of integration, two that are for Voevodsky motives, one for mixed $\ell$-adic sheaves, and a fourth theory of integration for rational mixed Hodge structures. We then show that they circumvent some of the complications of classical motivic integration, leading to new arithmetic and geometric results concerning K-equivalent $k$-varieties. For example, in addition to recovering known results regarding K-equivalent smooth projective varieties, we show that K-equivalent smooth projective $\mathbb{F}_q$-varieties have isomorphic rational $\ell$-adic Galois representations (up to semisimplification), and so also the same zeta functions (the equality of zeta functions is true even without projectivity). This is an arithmetic result inaccessible to classical motivic integration. Furthermore, we formulate a natural injectivity conjecture, and show that in combination with our theories of integration, implies that up to direct summing a common Chow motive K-equivalent smooth projective $k$-varieties have the same $\mathbb{Z}[1/p]$-Chow motives ($p$ is the characteristic exponent of $k$). This gives more evidence for a conjecture of Chin-Lung Wang suggesting the equivalence of integral motives of K-equivalent smooth projective varieties. In particular, this is the case if $k$ is such that the slice filtration preserves geometricity, which is conjecturally true for $k$ a number field or finite field, at least if we consider rational coefficients. Furthermore, we connect our theory of integration of rational Voevodsky motives to the existence of motivic $t$-structures for geometric Voevodsky motives; we show that the existence of a motivic t-structure implies that K-equivalent smooth projective varieties have equivalent rational (Chow) motives. We also connect this to a conjecture of Orlov concerning bounded derived categories of coherent sheaves. This makes progress on showing that all cohomology theories should agree for K-equivalent smooth projective varieties (at least rationally and for suitable base fields).
\end{abstract}
\maketitle
\setcounter{tocdepth}{2}
\tableofcontents
\section{Introduction}\label{intro}
Let $k$ be a perfect field and $R$ a commutative ring. Suppose throughout that $k$ has exponential characteristic $p$ invertible in the coefficient ring $R$. For our purposes, $R$ will be $\mathbb{Z}[1/p]$ or $\mathbb{Q}$ most of the time. Furthermore, a variety is said to be Calabi-Yau if it has trivial canonical bundle. In this paper, we take inspiration from classical motivic integration and its wealth of applications to birational geometry to develop two theories of integration for Voevodsky's mixed motives, a theory of integration of mixed $\ell$-adic sheaves, and one for rational mixed Hodge structures. We construct them so that they circumvent some of the obstructions in the path of classical motivic integration to concrete applications in arithmetic and geometry. \\
\\
The primary applications in which we are interested are related to the following general question in birational geometry.
\begin{question}
If $X$ and $Y$ are birational smooth projective Calabi-Yau varieties, what properties do they have in common? How can we distinguish non-isomorphic birational Calabi-Yau varieties?
\end{question}
As we will soon discuss, there are tools in the literature such as classical motivic integration that give results in this direction. If we are concerned with the \textit{module} structure of cohomology theories, the ultimate (conjectural) answer is that birational smooth projective Calabi-Yau varieties have isomorphic integral Chow motives. Indeed, this is the conjecture of Chin-Lung Wang \cite{Wang} (see Conjecture~\ref{Wang2conjecture} below). Via realization functors, this would give us the equivalence of many cohomology theories for such varieties. We try to come as close as possible to answering this aspect of the above question by using our theories of integration. One of the theories of integration is for rational Voevodsky motives, and assumes the existence of a motivic $t$-structure. This theory for rational Voevodsky motives also allows us to relate the existence of motivic $t$-structures to conjectures in non-commutative geometry. Ultimately, we would like to have a fully developed theory of integration for integral Voevodsky motives with all the desired properties.\\
\\
In this paper, we introduce an extension of Voevodsky's geometric effective motives that we call \textit{slice} (effective) motives and denote by $\DM_{\textup{sl}}^{\textup{eff}}(k;R)$. This is the smallest stable $\infty$-category containing geometric effective motives $\DM_{\textup{gm}}^{\textup{eff}}(k;R)$ as well as all their slices (in Voevodsky's slice filtration). Conjecturally, over finite fields and number fields and at least with rational coefficients, the slice of geometric motives are again geometric and so $\DM_{\textup{sl}}^{\textup{eff}}(k;\mathbb{Q})=\DM_{\textup{gm}}^{\textup{eff}}(k;\mathbb{Q})$ for such fields. However, if $k$ is large enough, for example if $k=\mathbb{C}$, it is known that the slice filtration does not preserve geometricity \cite{Ayoubslice}. Nonetheless, we propose the following conjecture.
\begin{conjecture}[Injectivity conjecture $I(k;R)$]\label{conj:inj}
The inclusion $\DM_{\textup{gm}}^{\textup{eff}}(k;R)\hookrightarrow\DM_{\textup{sl}}^{\textup{eff}}(k;R)$ induces an \textup{injective} group homomorphism
\[K_0(\DM_{\textup{gm}}^{\textup{eff}}(k;R))\rightarrow K_0(\DM_{\textup{sl}}^{\textup{eff}}(k;R)).\]
\end{conjecture}
After setting up the foundations, we use our theories to obtain a number of consequences regarding classical questions on derived cateogies of coherent sheaves, arithmetic, and geometry. The first main application in this paper is the following.
\begin{theorem}\label{maintheoremintro}If $X$ and $Y$ are K-equivalent smooth $k$-varieties, then $[M(X)]=[M(Y)]$ in $K_0(\DM_{\textup{sl}}^{\textup{eff}}(k;\mathbb{Z}[1/p]))$. If the injectivity conjecture $I(k;R)$ is true, then equality holds in $K_0(\DM_{\textup{gm}}^{\textup{eff}}(k;R))$.\end{theorem}
Recall that two smooth $k$-varieties $X$ and $Y$ are said to be \textit{K-equivalent} if there is a smooth $k$-variety $Z$ along with proper birational morphisms $f:Z\rightarrow X$ and $g:Z\rightarrow Y$ such that $f^*\omega_X\simeq g^*\omega_Y$. Examples of K-equivalent complex varieties are birational smooth projective complex varieties with nef canonical divisors. If $k$ admits resolution of singularities, then birational smooth projective Calabi-Yau $k$-varieties are K-equivalent. \\
\\
As a consequence, we obtain the following stronger result using Bondarko's weight structures \cite{Bondarko2}, \cite{Bondarko}.
\begin{theorem}\label{maintheoremintro2}Suppose conjecture $I(k;R)$ is true. If $X$ and $Y$ are K-equivalent smooth projective $k$-varieties, then $M(X)\oplus P\simeq M(Y)\oplus P$ in $\DM_{\textup{gm}}^{\textup{eff}}(k;R)$ for some Chow motive $P$.
\end{theorem}
Note that (effective) Chow motives sit fully faithfully inside (effective) geometric Voevodsky motives (Corollary 4.2.6 of Voevodsky's~\cite{Voevodsky00}), and so the above is also a statement about effective Chow motives with coefficients in $R$. This gives more evidence for the following conjecture of Chin-Lung Wang.
\begin{conjecture}[Chin-Lung Wang \cite{Wang}]\label{Wang2conjecture}If $X$ and $Y$ are K-equivalent smooth projective complex varieties, then they have isomorphic \textup{integral} Chow motives.
\end{conjecture}
We will come back to this conjecture later in the introduction.\\
\\
As a corollary of Theorem~\ref{maintheoremintro2}, we obtain the following.
\begin{corollary}
Suppose $k$ is field of characteristic zero and the injectivity conjecture $I(k;\mathbb{Z})$ is true. Then if $X$ and $Y$ are K-equivalent smooth projective $k$-varieties, $H^*(X^{an};\mathbb{Z})\simeq H^*(Y^{an};\mathbb{Z})$.
\end{corollary}
Furthermore, note that Theorem~\ref{maintheoremintro2} has, as a particular consequence, the equality of Hodge numbers of K-equivalent complex varieties, which is a theorem due to Kontsevich \cite{Kontsevich}. Since the power of the integration of Voevodsky motives relies on the validity of the the injectivity conjecture, we cannot obtain this result unconditionally from our motivic construction. Instead, we specialize the steps in the construction of the theory to rational mixed Hodge structures and develop what we call \textit{mixed Hodge integration}. As we will see, this theory comes equipped with an injective homomorphism. From this, we unconditionally recover Kontsevich's result; we obtain that K-equivalent smooth projective $k$-varieties, $k$ of characteristic zero, have isomorphic rational Hodge structures. We prove the following.
\begin{theorem}If $k$ is a field of characteristic zero, and $X$ and $Y$ are K-equivalent smooth projective $k$-varieties, then they have isomorphic rational Hodge structures.
\end{theorem}
Of course, a special case of this corollary is when $X$ and $Y$ are birational Calabi-Yau complex varieties. Note that this is also a refinement of Batyrev's result on the equality of Betti numbers of birational smooth projective Calabi-Yau complex varieties \cite{Batyrev}.
\begin{remark}
Chenyang Xu and Qizheng Yin have an unpublished preprint in which they prove that K-equivalent smooth projective complex varieties have isomorphic integral singular cohomology groups. Furthermore, Mark McLean has recently claimed a proof that birational Calabi-Yau complex manifolds admit Hamiltonians with isomorphic Hamiltonian Floer cohomology algebras after a certain change of Novikov rings. From this, he deduces that birational Calabi-Yau complex projective manifolds have isomorphic integral singular cohomology groups. See Theorem 1.2 and Corollary 1.3 of~\cite{Mclean}.
\end{remark}
We also have the following arithmetic corollary.
\begin{corollary}
Suppose $k$ is a perfect field and the injectivity conjecture $I(k;\mathbb{Z}[1/p])$ is true. If $X$ and $Y$ are K-equivalent smooth projective $k$-varieties, then $H^*(X_{\overline{k}};\mathbb{Z}_{\ell})^{ss}\simeq H^*(Y_{\overline{k}};\mathbb{Z}_{\ell})^{ss}$ as graded Galois representations (up to semi-simplification).
\end{corollary}
As in mixed Hodge integration, in order to obtain unconditional arithmetic results, we develop a theory of integration for $\ell$-adic sheaves over finite fields, and call it \textit{mixed $\ell$-adic integration}. This will come with an adequate injectivity result that allows us to prove the following weaker, but unconditional, result about ratoinal $\ell$-adic Galois representations.
\begin{theorem}
If $X$ and $Y$ are K-equivalent smooth projective $\mathbb{F}_q$-varieties, then $H^*(X_{\overline{k}};\overline{\mathbb{Q}}_{\ell})^{ss}\simeq H^*(Y_{\overline{k}};\overline{\mathbb{Q}}_{\ell})^{ss}$ as graded Galois representations (up to semi-simplification).
\end{theorem}
As a particular consequence, two such $X$ and $Y$ have the same zeta functions. This result on zeta functions is also true without the projectivity condition as our proofs show that even without the projectivity condition, the classes of the $\ell$-adic Galois representations in the Grothendieck group of mixed $\ell$-adic sheaves are equal. If $\mathbb{F}_q$ admits resolution of singularities, then two birational Calabi-Yau $\mathbb{F}_q$-varieties have the same rational $\ell$-adic Galois representations (up to semi-simplification), and so also the same zeta functions.\\
\\
Recall the notion of a Krull-Schmidt category.
\begin{definition}[Krull-Schmidt category]\label{KScategory} An $R$-linear additive category $\cal{C}$ is said to be a \textup{Krull-Schmidt} category if every object is a finite direct sum of objects with local endomorphism rings.
\end{definition}
The Krull-Schmidt theorem says that an object in a Krull-Schmidt category has a local endomorphism ring if and only if it is indecomposable. Furthermore, it also says that any object is uniquely, up to permutation, a direct sum of indecomposable objects. Many examples of Krull-Schmidt categories come from abelian categories in which every object has finite length. A concrete example is the category of finitely generated modules over a finite $R$-algebra, where $R$ is a commutative Neotherian local complete ring (e.g. $\mathbb{Z}_{\ell}$). See \cite{AtiyahKS} for a discussion of the Krull-Schmidt theorem. It is not known if the category of effective Chow motives over a field $k$ with $R$-coefficients $\Chow^{\textup{eff}}(k;R)$ is a Krull-Schmidt category, even if $k$ is of characteristic zero and $R=\mathbb{Q}$. It is known, however, that integral Chow motives do not form a Krull-Schmidt category (see Example 32 of \cite{CherMerk} of Chernousov and Merkurejv).\\
\\
Using Theorem~\ref{maintheoremintro2}, we have the following corollary.
\begin{corollary}Suppose $\Chow^{\textup{eff}}(k;R)$ is a Krull-Schmidt category, and suppose the injectivity conjecture $I(k;R)$ is true. Then for $X$ and $Y$ K-equivalent smooth projective $k$-varieties, the Chow motives of $X$ and $Y$ in $\Chow^{\textup{eff}}(k;R)$ are equivalent.
\end{corollary}
As stated above, Example 32 of \cite{CherMerk} shows that $\Chow^{\textup{eff}}(k;\mathbb{Z})$ is not a Krull-Schmidt category, and so this line of argument will not prove (the integral version of) Conjecture~\ref{Wang2conjecture}.\\
\\
On the other hand, if there is a motivic $t$-structure on rational geometric Voevodsky motives (see subsection~\ref{mot} for a discussion of motivic $t$-structures), then $\Chow(k;\mathbb{Q})$ is a Krull-Schmidt category. In fact, we show that if there is a motivic $t$-structure, there is another theory of integration based on the motivic $t$-structure that allows us to remove the assumption on the validity of the injectivity conjecture. See Section~\ref{rationalintegration} for details. Using this theory, we prove the following theorem.
\begin{theorem}\label{firstmot}
Suppose the motivic $t$-structure conjecture is true for $\DM_{\textup{gm}}(k;\mathbb{Q})$. If $X$ and $Y$ are two K-equivalent smooth projective $k$-varieties, then $M(X)_{\mathbb{Q}}\simeq M(Y)_{\mathbb{Q}}$.
\end{theorem}
The existence of such a motivic $t$-structure for rational geometric Voevodsky motives is a central conjecture in Voevodsky's theory of motives, and if it is not true then his theory is inadequate for many envisioned purposes. As particular consequences, this theorem implies that the existence of a suitable motivic $t$-structure on $\DM_{\textup{gm}}(k;\mathbb{Q})$ implies that K-equivalent smooth projective $k$-varieties have isomorphic rational noncommutative motives, rational $\ell$-adic Galois representations, rational Hodge structures (when defined), rational Chow groups, and rational algebraic K-groups.\\
\\
Recall the following standard definition.
\begin{definition}[D-equivalence] Two $k$-varieties $X$ and $Y$ are said to be D-equivalent if their bounded derived categories of coherent sheaves are equivalent as $k$-linear triangulated categories.
\end{definition}
Another consequence of our first integration of Voevodsky motives via the slice filtration (see Section~\ref{construction} for the construction) is that we can prove a theorem pertaining to a conjecture of Orlov \cite{Orlov}. Orlov's conjecture states that D-equivalent smooth projective complex varieties have equivalent rational motives. Indeed, we have the following. 
\begin{theorem}
Suppose conjecture $I(\mathbb{C};R)$ is true. If $X$ and $Y$ are D-equivalent smooth projective complex varieties such that either $\kappa(X)=\dim X$ (general type) or $\kappa(X,-K_X)=\dim X$, then $M(X)\oplus P\simeq M(Y)\oplus P$ in $\DM_{\textup{gm}}^{\textup{eff}}(\mathbb{C};R)$ for some Chow motive $P\in\Chow^{\textup{eff}}(\mathbb{C};R)$. 
\end{theorem}
On the other hand, assuming $R=\mathbb{Q}$, we can obtain the same result by replacing $I(\mathbb{C};\mathbb{Q})$ with the existence of a motivic $t$-structure for rational Voevodsky motives; in fact, we could also assume that $P=0$ when there is a motivic $t$-structure, thus proving that the existence of a motivic $t$-structure implies Orlov's conjecture for Fano and general type complex varieties. Note that if Chow motives with $R$ coefficients form a Krull-Schmidt category, then we can assume that $P=0$. Again, this would follow rationally if there is a motivic $t$-structure.
\begin{remark}
It is known that if $X$ and $Y$ are D-equivalent $k$-varieties with ample or anti-ample canonical bundles, then $X$ and $Y$ are isomorphic \cite{BO2}. When they have ample canonical bundles for example, we are in the setting of varieties of general type. 
\end{remark}
The proof of the above is an easy consequence of our main theorem. Indeed, for smooth projective complex varieties of general type, Kawamata \cite{Kawamata} has proved that D-equivalence implies K-equivalence. See Theorem~\ref{motivicOrlov}. As far as the author knows, this is the first time the existence of motivic $t$-structures or Chow motives forming a Krull-Schmidt category has been connected to this conjecture of Orlov.\\
\\
Before we describe the history behind and applications of classical motivic integration, we summarize here the four different theories of integration that are constructed in this paper.
\begin{enumerate}[(A)]
	\item \textbf{Integration of Voevodsky motives using the slice filtration}: this is based on the slice filtration. The construction itself is unconditional, and comes equipped with a group homomorphism
	\[c_R:K_0(\DM_{\textup{sl}}^{\textup{eff}}(k;R))\rightarrow \cal{M}(k;R)\]
	that is injective on the image of
\[K_0(\DM_{\textup{gm}}^{\textup{eff}}(k;R))\rightarrow K_0(\DM_{\textup{sl}}^{\textup{eff}}(k;R)),\]
where $\cal{M}(k;R)$ is the abelian group in which our integrals will land. The main drawback of this construction is that it is not known if the latter morphism is injective. Though the author conjectures this to be true, not having this prevents us from deducing equality of classes of geometric Voevodsky motives in $K_0(\DM_{\textup{gm}}^{\textup{eff}}(k;R))$.
\item \textbf{Integration of rational Voevodsky motives using a motivic $t$-structure}: this is based on the hypothetical motivic $t$-structure on rational geometric Voevodsky motives. This theory comes equipped with an injective homomorphism
\[c_R^{mot}:K_0(\DM_{\textup{gm}}^{\textup{eff}}(k;\mathbb{Q}))\rightarrow \cal{M}^{\mot}(k;\mathbb{Q}),\]
where $\cal{M}^{\mot}(k;R)$ is the group in which the integrals will land. This theory is based on a filtration of $\DM_{\textup{gm}}^{\textup{eff}}(k;\mathbb{Q})$ that relies on the existence of a motivic $t$-structure, and so, contrary to the previous theory, the construction itself is conditional. Furthermore, it applies to rational motives as we cannot hope for the existence of a motivic $t$-structure on integral geometric Voevodsky motives in general (See Proposition 4.3.8 of \cite{Voevodsky00} for when the base field $k$ is such that there is a conic over $k$ without $k$-rational points). Though conditional, this allows us to demonstrate new connections between the existence of motivic $t$-structures and questions in birational geometry.
\item \textbf{Mixed $\ell$-adic integration}: this is a theory of integration using mixed $\ell$-adic sheaves. It comes equipped with an injective homomorphism
\[K_0(D^b_{m\geq 0})\rightarrow\cal{M}^{\ell}(k).\]
Two drawback are that it works over finite fields, and that it works rationally. On the positive side, this theory is unconditional, and has the advantage that it allows us to unconditionally deduce arithmetic results related to the birational geometry of K-equivalent varieties. Notably, it allows us to show that K-equivalent smooth projective varieties over finite fields have the $\ell$-adic Galois representations up to semisimplication, and consequently also the same zeta functions (even true without projectivity). This is a result inaccessible to classical motivic integration.
\item \textbf{Mixed Hodge integration}: this is a theory of integration for mixed Hodge modules. The construction is unconditional, and allows us to obtain results about mixed Hodge structures when we are given two K-equivalent smooth projective varieties over a field of characteristic zero. This theory also comes equipped with an injective homomorphism
\[K_0(D^b_{\Hdg,\geq 0})\rightarrow\cal{M}^{\Hdg}(k).\]
The advantage of this theory over the motivic theories is that it allows us to obtain unconditional results about mixed Hodge structures.
\end{enumerate}
Let us recall why the idea of motivic integration is important by recalling the history behind the theory. To mathematicians as well as physicists, the classification of Calabi-Yau varieties is important. Batyrev proved that two birationally equivalent smooth projective Calabi-Yau complex varieties have the same Betti numbers \cite{Batyrev}. This result dating to 1996 was used by Beauville to explain the Yau-Zaslow formula counting the number of rational curves on K3 surfaces \cite{Beauville}. Such results are also used to bound Hodge numbers of elliptic Calabi-Yau varieties \cite{DiCerbo}, used to prove the log canonical threshold formula \cite{Mustata1},\cite{Mustata2}. Motivic integration has also been used to prove transfer principles that allow the study of the Fundamental Lemma in the Langlands program \cite{LoeCluck2}, \cite{Hales}.\\
\\
Roughly, Batyrev's proof of the equality of Betti numbers of birational smooth projective Calabi-Yau complex varieties goes as follows. Choose a lift of the two varieties to the maximal compact subring $B$ of an appropriately chosen local field, and suppose its maximal ideal is $\mathfrak{q}$ and its residue field is isomorphic to the finite field $\mathbb{F}_{q}$ of characteristic $p$. Count the number of $\mathbb{F}_{q}$-points on the varieties using $p$-adic integration with respect to canonical measures induced by gauge forms on the two varieties (gauge forms exist because the varieties are assumed to be Calabi-Yau). By showing that the $p$-adic integrals in this setup can be computed on dense open subsets of the varieties, he showed via the transformation rule for Haar integrals that the reduction modulo $\mathfrak{q}$ of the models have the same number of $\mathbb{F}_{q}$-points. By doing the same process with cyclotomic extensions of $B$, he showed that they have the same number of $\mathbb{F}_{q^n}$-points. Concisely, the zeta functions of the reductions modulo $\mathfrak{q}$ are the same. By using the Weil conjectures proved by Deligne \cite{WeilI}, \cite{WeilII}, he concluded that they have the same Betti numbers. The arithmetic nature of this proof for a complex-geometric result suggests the possibly of the existence of a deeper underlying motivic reason for the validity of this theorem.\\
\\
On the other hand, it was conjectured that two such varieties in fact have the same Hodge numbers. This generalizes the result of Batyrev; indeed, given this result, the decomposition theorem in Hodge theory implies the result of Batyrev. On December 7 1995, Kontsevich gave a lecture at Orsay envisioning a theory of motivic integration to prove the stronger statement that two such complex varieties have the same Hodge numbers \cite{Kontsevich}. In fact, he proved the following more general result.
\begin{theorem}[Kontsevich \cite{Kontsevich}, Denef-Loeser \cite{DenLoe}]\label{Konttheorem}If $X$ and $Y$ are K-equivalent smooth projective complex varieties, then they have the same Hodge numbers.\end{theorem}
Kontsevich's idea of using motivic integration in the proof of his theorem is based on the following observation. Hodge numbers are encoded in the Deligne-Hodge polynomial. In the case of a smooth projective complex variety $X$, the polynomial is given by
\[E(X)=\sum_{i,j}(-1)^{i+j}h^i(X;\Omega_X^j)u^iv^j\in\mathbb{Z}[u,v].\]
For general smooth complex varieties, this is defined using mixed Hodge structures. The Deligne-Hodge polynomial has the property that it can be viewed as a function on the Grothendieck ring of complex varieties $K_0(\Var_{\mathbb{C}})$. In his lecture in Orsay, Kontsevich envisioned a theory of motivic integration that allows us to prove that the classes of two K-equivalent complex varieties in a completion $\widehat{\cal{M}}_k$ of $K_0(\Var_{\mathbb{C}})[\mathbb{L}^{-1}]$ ($\mathbb{L}:=[\mathbb{A}^1]$ the Lefschetz motive) are equal. The Deligne-Hodge polynomial extends to this completion, from which the result follows.\\
\\
In 1996, a preprint of Denef and Loeser containing such a construction of motivic integration was circulating in the mathematical community. This construction was published in 1999 \cite{DenLoe}. The advantage of this proof is that it not only circumvents the usage of the Weil conjectures proved by Deligne and used by Batyrev, but it gives a stronger result in the sense that two such complex varieties have the same value on any function on the aforementioned completion of the Grothendieck ring of varieties with the Lefschetz motive inverted. We remark that later on Chin-Lung Wang in 2002 \cite{Wang} and Tetsushi Ito in 2004 \cite{Tetsushi} independently proved this result on Hodge numbers by using $p$-adic Hodge theory to refine Batyrev's proof \cite{Tetsushi}.\\
\\
Though classical motivic integration has developed in various directions since its invention, we now discuss some of its defects. In the classical theory of motivic integration, motivic integrals take values in the \textit{completion} $\widehat{\cal{M}}_k$ for suitable $k$ because we want to be able to talk about infinite series and their convergence. Since not all motivic measures factor through this completion, we cannot deduce all the results we would like such a theory to imply. For example, the counting measure $C_q:K_0(\Var_{\mathbb{F}_q})[\mathbb{L}^{-1}]\rightarrow\mathbb{Q}$ does not factor through $\widehat{\cal{M}}_{\mathbb{F}_q}$. Indeed, $q^n/\mathbb{L}^n\rightarrow 0$ in the topology of the completion while $C_q(q^n/\mathbb{L}^n)=1$. This is one reason it is difficult to deduce arithmetic information using classical motivic integration. Therefore, often it is needed to know that the motivic integrals computed take values is a proper subring of $\widehat{\cal{M}}_k$.  Unfortunately, it is unknown if the natural map from $K_0(\Var_k)[\mathbb{L}^{-1}]$ to its completion $\widehat{\cal{M}}_k$ is injective. On the other hand, there have been improvements to classical motivic integration in characteristic zero in this direction; however, the integrals still take values in the localization of $K_0(\Var_k)$ with respect to the classes $\mathbb{L}$ and $\mathbb{L}^n-1$, $n\geq 1$ \cite{LoeCluck}. Such problems have been central to the theories of motivic integration developed using the Grothendieck ring of varieties.\\
\\
The elements of $K_0(\Var_k)$ are called \textit{virtual motives} because this ring encapsulates the idea of cutting and pasting $k$-varieties. On the other hand, it is by now clear that the formalism of mixed motives \`a la Voevodsky is a better framework for thinking about motives because they are $\infty$-\textit{categories} carrying a very rich structure. For example, they are stable $\infty$-categories that come with six functor formalisms. Moreover, there are realization functors to derived categories of $\ell$-adic sheaves, mixed Hodge structures, etc. The fact that Voevodsky motives are categorical as opposed to ring-theoretic roughly means that they are richer in information and flexibility. It is this flexibility that allows us to construct theories of integration for Voevodsky motives that have better properties in some respects. Though the two theories of integration for Voevodsky motives rely on conjectures regarding motives, the integration for mixed $\ell$-adic sheaves and mixed Hodge structures are unconditional. They could be viewed as new tools for birational geometry.\\
\\
It should be pointed out that it is not clear to the author how the integrations in this paper relate to the classical theory beyond both being the same in philosophy. One main goal of motivic integration theories based on $K_0(\Var_k)[\mathbb{L}^{-1}]$ is proving equality of classes in this ring; however, as previously mentioned, to a great extent this is not achieved. Contrary to classical motivic integration, we are able to show that our mixed $\ell$-adic and mixed Hodge integration theories come equipped with the natural analogous injective maps (see Propositions~\ref{ellinjectivity} and ~\ref{Hdginjectivity}).\\
\\
If we want to study concrete structures like $\ell$-adic Galois representations and mixed Hodge structures, having equality in either $K_0(\DM_{\textup{gm}}^{\textup{eff}}(k;R))$ (or $K_0(\DM_{\textup{gm}}(k;R))$) or $K_0(\Var_k)[\mathbb{L}^{-1}]$ is good enough. However, there are also some differences between the classical theory and our theories. Indeed, the natural morphism 
\[K_0(\Var_{\mathbb{C}})[\mathbb{L}^{-1}]\rightarrow K_0(\DM_{\textup{gm}}(\mathbb{C};\mathbb{Q}))\]
induced by sending $X$ to $\pi^X_!\mathbf{1}_X$, where $\pi^X:X\rightarrow\Spec\mathbb{C}$ is the structure morphism of $X$, has a nontrivial kernel. For example, for $g\geq 2$, there are abelian $g$-folds $A$ such that $\textup{End}(A)\simeq\mathbb{Z}$ and $A\not\simeq\widehat{A}:=\textup{Pic}^0(A)$. $[A]-[\widehat{A}]$ is a nonzero element of $K_0(\Var_{\mathbb{C}})[\mathbb{L}^{-1}]$ that maps to zero because $A$ and $\widehat{A}$ are isogenous and so have equivalent rational Chow motives. (For details, see \cite{injcounterexample}.) As a result, equality in $K_0(\Var_{\mathbb{C}})[\mathbb{L}^{-1}]$ is more refined than equality in $K_0(\DM_{\textup{gm}}(\mathbb{C};\mathbb{Q}))$. Therefore, the ultimate goal of proving equality of virtual motives is more ambitious than proving equality in the Grothendieck group of geometric Voevodsky motives. That being said, what we show in this paper is that using Voevodsky motives, mixed $\ell$-adic sheaves, or mixed Hodge structures for our theories of integration allows us to circumvent some of the deficiencies of classical motivic integration that prevent us from obtaining stronger concrete results in geometry and arithmetic. Additionally, we are able to relate previously unrelated conjectures to one another.\\
\\
This paper grew out of an attempt to construct a categorified motivic integration taking values in some refinement of $\DM_{\textup{gm}}(k;R)$ and not just in a Grothendieck-group-like construction associated to a triangulated (or stable $\infty$-)category. This categorified version, though, requires a better understanding of the geometric side of motivic integration, and is an ongoing project. If we manage to develop such a categorified motivic integration, then all the results in this paper would become unconditional. An integral version of such a theory would prove the aforementioned philosophically important Conjecture~\ref{Wang2conjecture} of Chin-Lung Wang. The author noticed a subtlety regarding compact objects in a previous version of this paper that forced the author to introduce slice motives, construct a theory of integration based on motivic $t$-structures, and develop mixed $\ell$-adic integration as well as mixed Hodge integration in order to preserve some of the concrete applications unconditionally. The gap was based on the fact that the slice filtration does not preserve geometric objects for large enough based fields. See Remark~\ref{badslice} and Ayoub's~\cite{Ayoubslice} for details.\\
\\
\textit{Acknowledgment.} I would like to thank Professor Denis-Charles Cisinski for discussions regarding previous versions of Lemma~\ref{fullfaithfulness} and Proposition~\ref{injectivity}, in particular for explaining to me the theorem of Thomason used in Lemma~\ref{fullfaithfulness}. Furthermore, I thank him for his hospitality and generosity while I was at the University of Regensburg in 2017. I would also like to thank Simon Pepin Lehalleur for pointing out a mistake in the applications section of a previous version of this paper, and to thank
Professor Charles Weibel for comments on previous versions of this paper. Furthermore, I would like to thank Professor Mikhail Bondarko for pointing out the connection to his work on weight structures. This project was supported by Princeton University and SFB1085:Higher invariants.
\section{Conventions and Preliminaries}\label{prelim}
Throughout this paper, $k$ is a perfect field with exponential characteristic $p$ and $R$ is a commutative ring in which $p$ is invertible. For us, a Calabi-Yau $k$-variety $X$ is one whose canonical bundle $\omega_X$ is trivial. Henceforth, we will use the language of $\infty$-categories as developed by Lurie in \cite{HTT} and \cite{HA}. In particular, homotopy (co)limits will be called (co)limits. Note that if $\cal{C}_0$ is an ordinary category, we abuse notation and still write $\cal{C}_0$ instead of its (classical) nerve $N(\cal{C}_0)$ which is an $\infty$-category. Classical (co)limits in $\cal{C}_0$ correspond to $\infty$-categorical/homotopy (co)limits in $N(\cal{C}_0)$, and so this convention should not cause confusion when dealing with ordinary categories.\\
\\
In the rest of this section, we recall the definition of Jet schemes and prove some of its properties. Furthermore, we discuss the essentials of the stable $\infty$-categories of mixed motives in the sense of Voevodsky, and discuss their properties that will be important in this paper. Everything in this section is known, and is discussed here only for the convenience of the reader and to fix the notation.\subsection{Jet Schemes} In this subsection, we give the definition and some of the properties of Jet schemes.
\begin{proposition}
Suppose $X$ is a $k$-scheme of finite type. For each $n\geq 0$, there is a $k$-scheme $\cal{J}_n(X)$ of finite type representing the functor
\[Z\mapsto\textup{Mor}_k(Z\times_k\Spec k[\![t]\!]/(t^{n+1}),X).\]
\end{proposition}
For a proof, see \cite{Green1} and \cite{Green2} of Greenberg. For $n\geq m$, $k[\![t]\!]/(t^{n+1})\rightarrow k[\![t]\!]/(t^{m+1})$ induces a morphism of $k$-schemes
\[\pi^n_m:\cal{J}_n(X)\rightarrow\cal{J}_m(X).\]
These morphisms are \textit{affine}, and so we have a $k$-scheme
\[\cal{J}_{\infty}(X):=\varprojlim_n\cal{J}_n(X).\]
For brevity, we often write $X_{\infty}$ instead of the cumbersome $\cal{J}_{\infty}(X)$. We shall call this the \textit{Jet scheme} of $X$, and we let
\[\pi^X_n:\cal{J}_{\infty}(X)\rightarrow\cal{J}_n(X)\]
be the natural projection. We know that $\cal{J}_{\infty}(X)$ represents the functor
\[Z\mapsto\textup{Mor}_k(Z\times_k\Spec k[\![t]\!],X),\]
and that $\pi^X_n$ is induced by the ring homomorphism $k[\![t]\!]\rightarrow k[\![t]\!]/(t^{n+1})$.\\
\\
Jet schemes are higher order versions of tangent bundles. In particular, if $X$ is a $k$-variety, then $\cal{J}_1(X)\simeq TX$ is the tangent bundle of $X$.\\
\\
We now recall some known properties of Jet schemes.
\begin{proposition}
Let $X\rightarrow Y$ be an \'etale morphism of $k$-schemes of finite type. The following natural square is a pullback square:
\[\xymatrix{\cal{J}_m(X) \ar@{->}[r] \ar@{->}[d] & \cal{J}_m(Y) \ar@{->}[d] \\ X \ar@{->}[r] & Y.}\]
In other words, $\cal{J}_m(X)\simeq \cal{J}_m(Y)\times_YX$.
\end{proposition}
\begin{proof}
We show the isomorphism on the level of the corresponding functor of points. Precisely, we want to show that the two functors
\[\textup{Mor}_k(-,\cal{J}_m(X))\left(\simeq\textup{Mor}_k(-\times_k\Spec k[\![t]\!]/(t^{m+1}),X)\right)\]
and
\begin{eqnarray*}\textup{Mor}_k(-,\cal{J}_m(Y)\times_YX)&\Big(\simeq&\textup{Mor}_k(-,\cal{J}_m(Y))\times_{\textup{Mor}_k(-,Y)}\textup{Mor}_k(-,X)\\ &\simeq& \textup{Mor}_k(-\times_k\Spec k[\![t]\!]/(t^{m+1}),Y)\times_{\textup{Mor}_k(-,Y)}\textup{Mor}_k(-,X)\Big)
\end{eqnarray*}
are equivalent. Let $Z$ be a $k$-scheme and consider the diagram
\[\xymatrix{X \ar@{->}[r] & Y\\ Z \ar@{->}[u]^{p} \ar@{->}[r] & Z\times_k\Spec k[\![t]\!]/(t^{m+1}). \ar@{->}[u]_{\theta} \ar@{.>}[ul]_{\gamma}}\]
Since $X\rightarrow Y$ is \'etale, and so formally \'etale, there is a one-to-one correspondence between $\gamma\in\textup{Mor}_k(Z\times_k\Spec k[\![t]\!]/(t^{m+1}),X)$ as in the diagram above, and morphisms $\theta$ and $p$ making the above diagram commute. The conclusion follows. 
\end{proof}
From this, we obtain the following important corollary.
\begin{corollary}
Let $X$ be a smooth $k$-scheme of finite type of pure dimension $d$. Then $\cal{J}_m(X)$ is an $\mathbb{A}^{dm}$-bundle over $X$. In particular, $\cal{J}_m(X)$ is smooth of pure dimension $d(m+1)$. Consequently, $\pi^{m+1}_m:\cal{J}_{m+1}(X)\rightarrow\cal{J}_m(X)$ is an $\mathbb{A}^d$-bundle.
\end{corollary}
\begin{proof}
It suffices to check this locally. By the \'etale invariance above and the fact that every smooth morphism is locally a composition of an \'etale morphism followed by an affine projection, it suffices to check this for $X=\mathbb{A}^d$, which is a simple computation.
\end{proof}
\subsection{Mixed Motives}\label{mixedmotives}
In the mid-eighties, Beilinson and Deligne conjectured the existence of a triangulated category with a $t$-structure whose heart is the conjectural abelian category of mixed motives. Voevodsky constructed triangulated categories with the hope that (with rational coefficients) their full subcategory of constructable objects is the derived category of the hypothetical abelian category of mixed motives. Though this property is, as of today, conjectural, these triangulated categories have had many applications. For example, the Bloch-Kato conjecture about the relation of Galois cohomology to Milnor K-theory was proved by Voevodsky \cite{Voe1},\cite{Voe2},\cite{Voe3},\cite{Voe4},\cite{Voe5}. We now have a theory of motivic cohomology, a spectral sequence from motivic cohomology to algebraic K-theory (analogue of the Atiyah-Hirzebruch spectral sequence for complex K-theory), and a six functor formalism among many other things. These tools give us a flexibility that allows us to study motivic phenomena within a very rich framework. In our case, the theory of Voevodsky motives is indispensable.\\
\\
In this subsection, we briefly recall the definition of the main stable $\infty$-category of Voevodsky motives. We will then discuss a part of the six functor formalism, localization sequences, and purity, all of which will be essential to our work. Our exposition is terse, and so we recommend the reader to look at \cite{CisDeg1} and \cite{CisDeg2} for the basic definitions and theorems.\\
\\
There are many variants of the construction of motives in the spirit of Voevodsky. One way the stable $\infty$-category of Voevodsky motives is constructed is using smooth correspondences. It is defined using the Nisnevich topology, partly because algebraic K-theory satisfies Nisnevich descent and certain motivic cohomology groups give us Chow groups. We say that a morphism $f:X\rightarrow Y$ is \textit{completely decomposed} at $y\in Y$ if there is an $x\in X$ above $y$ such that the residual field extension $k(y)\rightarrow k(x)$ is an isomorphism. A Nisnevich covering of an $S$-scheme $Y$ in this topology is a finite family $\{f_i:U_i\rightarrow Y\}$ of \'etale $S$-morphisms with the property that for every point $y\in Y$, there is a $j$ such that $f_j$ is completely decomposed at $y$. This gives us the site $(\Sm/S)_{Nis}$ of smooth $S$-schemes with the Nisnevich topology. In this setting, a Nisnevich sheaf of $R$-modules \textit{with transfers} is a contravariant functor on the category of correspondences $\textup{Cor}_{S}$ to the category of $R$-modules that is a Nisnevich sheaf once we restrict it to $\Sm/S$. Let us denote this category of Nisnevich sheaves of $R$-modules with transfers by $\textup{Sh}^{\textup{tr}}(S;R)$. Consider the derived $\infty$-category of chain complexes of $R$-modules $\textup{D}(\textup{Sh}^{\textup{tr}}(S;R))$ on the category of Nisnevich sheaves with transfers, and invert the chain complexes of the form
\[\hdots\rightarrow 0\rightarrow h_{X\times\mathbb{A}^1}\rightarrow h_X\rightarrow 0\rightarrow\hdots,\]
where $h_T$ denotes the representable Nisnevich sheaf with transfers associated to $T$. This gives us the stable $\infty$-category $\DM^{\textup{eff}}(S;R)$ of \textit{effective Voevodsky motives}. The object in $\DM^{\textup{eff}}(S;R)$ associated to $h_X$ will be denoted by $M_S(X)$, and will be called the motive of $X$. Voevodsky motives in general are obtained by stabilizing with respect to the Tate twist, which we now describe. The structure morphism $\mathbb{P}^1_S\rightarrow S$ induces the morphism of motives
\[M_S(\mathbb{P}^1_S)\rightarrow M_S(S).\]
Its fiber shifted by $[-2]$ is denoted by $\mathbf{1}_S(1)$ and called the \textit{Tate twist}. Stabilizing with respect to the Tate twist gives us
\[\DM(S;R):=\colim_{n}\left(\DM^{\textup{eff}}(S;R)\xrightarrow{-\otimes\mathbf{1}_S(1)}\DM^{\textup{eff}}(S;R)\xrightarrow{-\otimes\mathbf{1}_S(1)}\hdots\right),\]
where the colimit is taken in the $\infty$-category of presentable $\infty$-categories with morphisms left adjoint functors. This is the stable $\infty$-category of \textit{Voevodsky motives} (without the effectivity condition). The smallest stable subcategory of $\DM^{\textup{eff}}(S;R)$ containing $M_S(X)$ is denoted by $\DM^{\textup{eff}}_{\textup{gm}}(S;R)$ and called the $\infty$-category of \textit{geometric effective} Voevodsky motives. $\DM_{\textup{gm}}(S;R)$, called the $\infty$-category of \textit{geometric} Voevodsky motives, is obtained by inverting $\mathbf{1}_S(1)$ in the category of geometric effective Voevodsky motives $\DM^{\textup{eff}}(S;R)$.\\
\\
In our case, however, we will need to work with schemes of finite type, not just the smooth ones. The construction above does not allow us to do this because of its restriction to \textit{smooth} schemes. Therefore, we may consider another variant of the above construction. Instead of working with smooth correspondences, we may work with the category $\Sch^{cor}/S$ consisting of separated finite type $S$-schemes with morphisms finite $S$-correspondences. Given a Grothendieck topology $\tau$, we may consider the big site $(\Sch/S)_{\tau}$ of finite type $S$-schemes with the $\tau$-topology. Considering presheaves of $R$-modules on $\Sch^{cor}/S$ that restrict to $\tau$-sheaves on $(\Sch/S)_{\tau}$ gives us the category $\underline{\textup{Sh}}^{tr}_{\tau}(S;R)$ of $\tau$-sheaves of $R$-modules with transfers. Taking the derived category, $\mathbb{A}^1$-localizing, and inverting the Tate twist as before gives us the category $\underline{\DM}_{\tau}(S;R)$ of $\tau$-motives. In this general context, any separated finite type $S$-scheme $X$ defines an object $\underline{M}_S(X)\in\underline{\DM}_{\tau}(S;R)$. When we do not use the symbol $\tau$, we mean that $\tau=Nis$. Therefore, $\underline{\DM}(S;R):=\underline{\DM}_{Nis}(S;R)$. We denote by $\DM_{\textup{gm}}^{\textup{eff}}(S;R)$ the smallest stable $\infty$-category containing motives of the form $M_S(X)(n)$, $n\geq 0$ and $X$ smooth $S$-scheme. Its stabilization with respect to the Tate twist will be denoted by $\DM_{\textup{gm}}(S;R)$. Considering the largest localizaing full subcategory of $\underline{\DM}_{\tau}(S;R)$ generated by motives $\underline{M}_S(X)(n)$ for $X$ smooth $S$-schemes and $n\in\mathbb{Z}$ gives us the category $\DM_{\tau}(S;R)$. From now on, we let $S$ be Noetherian. Taking $\tau=cdh$ gives us the categories $\DM_{cdh}(S;R)$ and $\underline{\DM}_{cdh}(S;R)$ that turn out to have some nice properties proved in \cite{CisDeg2}.\\
\\
For all such categories, we have basic functors $f^*,f_*, \otimes,\underline{\textup{Hom}}(-,-)$. For smooth $f:X\rightarrow Y$, we also have a left adjoint $f_{\texttt{\#}}$ to $f^*$. If $X$ is a smooth $S$-scheme with structure morphism $\pi^X:X\rightarrow S$, then $M_S(X)=\pi^X_{\texttt{\#}}\mathbf{1}_X$.\\
\\
The functor that will be of greatest importance to us is $f_!$, which is, up to equivalence, given by $p_*j_{\texttt{\#}}$ for $f=p\circ j$ any factorization of $f$ into an open embedding $j$ followed by a proper morphism $p$ (such a factorization exists by Nagata compactification~\cite{Nagata}).\\
\\
Essential to our work will be localization sequences. By theorem 5.11 of Cisinski and D\'eglise in \cite{CisDeg2}, since by assumption $k$ has characteristic exponent invertible in $R$, for each closed embedding $i:Z\hookrightarrow X$ of $k$-varieties with open complement $j:U\hookrightarrow X$, there is a cofiber sequence
\[j_{\texttt{\#}}j^*\mathbf{1}_X\rightarrow\mathbf{1}_X\rightarrow i_*i^*\mathbf{1}_X\rightarrow j_{\texttt{\#}}j^*\mathbf{1}_X[1]\]
in $\DM_{cdh}(k;R)$. Applying the exact functor $\pi^X_!$ to it, we obtain the cofiber sequence
\[\pi^U_!\mathbf{1}_U\rightarrow\pi^X_!\mathbf{1}_X\rightarrow\pi^Z_!\mathbf{1}_Z\rightarrow \pi^U_!\mathbf{1}_U[1]\]
in $\DM_{cdh}(k;R)$. Note that we are not assuming smoothness, which is an advantage of working with the cdh topology. Note, however, that by Corollary 5.9 of Cisinski and D\'eglise \cite{CisDeg2}, the natural morphism $\DM(X;R)\rightarrow\DM_{cdh}(X;R)$ is an equivalence of symmetric monoidal stable $\infty$-categories if $X$ is a regular $k$-scheme. Furthermore, using Proposition 8.1(c) of Cisinski and D\'eglise \cite{CisDeg2} for the field $k$ we have that $\DM(k;R)\simeq\DM_{cdh}(k;R)=\underline{\DM}_{cdh}(k;R)$, and so we can work in the larger category $\underline{\DM}_{cdh}(k;R)$ and pass to $\DM(k;R)$.\\
\\
Furthermore, we also have purity, a part of which says that if $f:X\rightarrow Y$ is a smooth separated morphism of finite type over $k$ of relative dimension $d$, then $f_!\simeq f_{\texttt{\#}}(-d)[-2d]$ (see Theorem 11.4.5 of \cite{CisDeg1}). We will take advantage of these properties when proving the well-definedness of our motivic measure.\\
\\
Note that we could make our constructions with cdh-motives from the onset, and use the above equivalences between Nisnevich and cdh-motives to conclude our results about Nisnevich motives.
\section{Integration via the slice filtration}\label{construction}
In this section, we define integration in the setting of geometric Voevodsky motives using the slice filtration. In order to do so, we first define the category of \textit{completed} geometric effective motives $\DM^{\textup{eff},\wedge}_{\textup{gm}}(k;R)$. This will be defined in Subsection~\ref{bdefinitions}. We also define the notion of \textit{convergent} motives that will be used in the definition of the group $\cal{M}(k;R)$ in which our integrals will land. We will also define the category of \textit{slice} motives $\DM_{\textup{sl}}^{\textup{eff}}(k;R)$, and show that there is a natural map from $K_0(\DM_{\textup{sl}}^{\textup{eff}}(k;R))$ to $\cal{M}(k;R)$ whose kernel does not contain any nontrivial elements coming from geometric motives. In order to define our integrals, we need a notion of measure; we define a measure on the Jet scheme $X_{\infty}$ taking values in $\DM^{\textup{eff},\wedge}_{\textup{gm}}(k;R)$. We do so by first defining the measure on the so-called \textit{stable subschemes} of $X_{\infty}$, and then we extend this by defining the notions of \textit{good} and \textit{measurable} subsets along with their measures. Given this measure, we define the integration of Voevodsky motives, and specialize to a particular class of functions, called \textit{measurable}, that will be of greatest interest to us. In the final subsection, we prove the transformation rule, a formula describing how an integral changes with respect to proper birational morphisms of smooth schemes. The transformation rule is the device that allows us to deduce results in birational geometry; it is a birational-geometric variant of the change of variables formula in differential geometry.
\subsection{Completed, Convergent, and Slice Motives}\label{bdefinitions}
Integration of Voevodsky motives will deal with the category of effective geometric Voevodsky motives $\DM^{\textup{eff}}_{\textup{gm}}(k;R)$ and will take values in an abelian group $\cal{M}(k;R)$ close to the Grothendieck group of some categorical limit $\DM_{\textup{gm}}^{\textup{eff},\wedge}(k;R)$ of Verdier quotients of $\DM_{\textup{gm}}^{\textup{eff}}(k;R)$ by $\otimes$-ideals. In this subsection, we define this categorical completion, the notion of convergent motives, as well as slice motives. We also discuss their relations to the usual categories of Voevodsky's motives.\\
\\
Note that by Theorem 1.1.4.4 of Lurie's \cite{HA}, the $\infty$-category $\Cat^{Ex}$ of small stable $\infty$-categories with exact functors is closed under small limits. The $\infty$-category of commutative algebra objects in it, denoted by $\Cat^{Ex,\otimes}$, is also closed under small limits. In fact, limits in $\Cat^{Ex,\otimes}$ can be computed in $\Cat^{Ex}$. See Section 3.2.2. of \textit{loc.cit}.
\begin{definition}
Define the $\infty$-category of \textup{completed} effective (geometric) Voevodsky motives with $R$-coefficients as the following limit in the $\infty$-category $\Cat^{Ex,\otimes}$ of small stable symmetric monoidal $\infty$-categories with exact functors:
\[\DM_{(gm)}^{\textup{eff},\wedge}(k;R):=\lim_{n}\DM_{(gm)}^{\textup{eff}}(k;R)/\DM_{(gm)}^{\textup{eff}}(k;R)(n),\]
where $\DM_{(gm)}^{\textup{eff}}(k;R)(n)$ is the full sub-$\infty$-category of $\DM_{(gm)}^{\textup{eff}}(k;R)$ generated by effective (geometric) motives Tate-twisted $n\geq 0$ times. The transition functors $\DM_{(gm)}^{\textup{eff}}(k;R)/\DM_{(gm)}^{\textup{eff}}(k;R)(n+1)\rightarrow\DM_{(gm)}^{\textup{eff}}(k;R)/\DM_{(gm)}^{\textup{eff}}(k;R)(n)$ are the natural localization functors.
\end{definition}
Let us set some notation that we will use soon. Let
\[L_n:\DM^{\textup{eff}}(k;R)\rightarrow\DM^{\textup{eff}}(k;R)/\DM^{\textup{eff}}(k;R)(n)\]
be the natural localization functor. Since $\DM^{\textup{eff}}(k;R)$ is presentable, this localization functor has a fully faithful right adjoint that we denote by 
\[i_n:\DM^{\textup{eff}}(k;R)/\DM^{\textup{eff}}(k;R)(n)\hookrightarrow\DM^{\textup{eff}}(k;R).\] 
By the universal property of limits, there is a natural functor $L_{\infty}:\DM^{\textup{eff}}(k;R)\hookrightarrow\DM^{\textup{eff},\wedge}(k;R)$. We also denote the restriction of $L_{\infty}$ to $\DM_{\textup{gm}}^{\textup{eff}}(k;R)$ by $L_{\infty}$. This abuse of notation should not cause confusion.\\
\\
Let us prove a lemma that will be used in the proof of subsequent lemmas.
\begin{lemma}\label{sufflarge}
Suppose $N\in\DM_{\textup{gm}}^{\textup{eff}}(k;R)$, and $A\in\DM^{\textup{eff}}(k;R)$. Then for $n\gg 1$ (depending on $N$), $A(n)\rightarrow N$ is trivial.
\end{lemma}
\begin{proof}
$\DM^{\textup{eff}}(k;R)$ is compactly generated by motives of the form $M(X)(n)$ for $n\geq 0$ and $X$ smooth \textit{projective} $k$-varieties. Since $N$ is geometric it is generated by \textit{finitely} many such $M(X_i)(n_i)[m_i]$, $X_i$ smooth projective $k$-varieties of dimension $d_i$. Suppose $A$ is a colimit of objects $A_j=M(Y_j)(k_j)[t_j]$, where $Y_j$ are smooth projective $k$-varieties and $k_j\geq 0$. Consider 
\begin{eqnarray*}
&&\Map_{\DM^{\textup{eff}}(k;R)}(A_j(n),N)\\&\simeq&\colim_{i}\Map_{\DM^{\textup{eff}}(k;R)}(M(Y_j)(k_j+n)[t_j],M(X_i)(n_i)[m_i])\\&\simeq& \colim_{i}\Map_{\DM(k;R)}(M(Y_j)(k_j+n-n_i)[t_j-m_i],M(X_i))\\&\simeq^{(1)}& \colim_{i}\Map_{\DM(k;R)}(M(Y_j)(k_j+n-n_i)[t_j-m_i],M^c(X_i))\\&\simeq^{(2)}&\colim_{i}\Map_{\DM(k;R)}(M(Y_j)(k_j+n-n_i)[t_j-m_i],M(X_i)^{\vee}(d_i)[2d_i])\\&\simeq&\colim_{i}\Map_{\DM(k;R)}(M(Y_j\times_kX_i)(k_j+n-n_i)[t_j-m_i],\mathbf{1}_k(d_i)[2d_i]),
\end{eqnarray*}
where $(1)$ follows from the fact that the $X_i$ are smooth \textit{projective}, and $(2)$ follows from $M^c(X)\simeq M(X)^{\vee}(d)[2d]$, where $d$ is the dimension of smooth $k$-variety $X$ and the $\vee$ denotes dualization. The last mapping space is contractible if $k_j+n-n_i>d_i$, that is, when $n>n_i+d_i-k_j$ (this follows from Voevodsky's paper~\cite{VoevodskyEM}). We can take $n\geq\max_i\{n_i+d_i+1\}$, which depends only on $N$ and ranges over finitely many $i$ by the assumption that $N$ is geometric. Consequently, $\Map_{\DM^{\textup{eff}}(k;R)}(A(n),N)\simeq\lim_{j}\Map_{\DM^{\textup{eff}}(k;R)}(A_j(n),N)\simeq 0$ for $n\geq\max_i\{n_i+d_i+1\}$, that is, for $n\gg 1$. The conclusion follows.
\end{proof}
A lemma that will be of importance to us later is the following. It demonstrates the natural expectation that completed effective geometric Voevodsky motives contain usual effective geometric Voevodsky motives as a full subcategory.
\begin{lemma}\label{fullfaithfulness}
The natural functors
\[\DM_{\textup{gm}}^{\textup{eff}}(k;R)\rightarrow\DM_{\textup{gm}}^{\textup{eff},\wedge}(k;R)\]
and
\[\DM_{\textup{gm}}^{\textup{eff},\wedge}(k;R)\rightarrow\DM^{\textup{eff},\wedge}(k;R)\]
are fully faithful.
\end{lemma}
\begin{proof}
Consider the natural commutative diagram
\[\xymatrix{\DM_{\textup{gm}}^{\textup{eff}}(k;R) \ar@{->}[r] \ar@{->}[d] & \DM_{\textup{gm}}^{\textup{eff},\wedge}(k;R) \ar@{->}[d] \\ \DM^{\textup{eff}}(k;R) \ar@{->}[r] & \DM^{\textup{eff},\wedge}(k;R)}\]
of functors. In order to show that $\DM_{\textup{gm}}^{\textup{eff}}(k;R)\rightarrow\DM_{\textup{gm}}^{\textup{eff},\wedge}(k;R)$ is fully faithful, we show that $\DM_{\textup{gm}}^{\textup{eff},\wedge}(k;R)\rightarrow\DM^{\textup{eff},\wedge}(k;R)$ and $L_{\infty}:\DM_{\textup{gm}}^{\textup{eff}}(k;R)\rightarrow\DM^{\textup{eff}}(k;R)\rightarrow\DM^{\textup{eff},\wedge}(k;R)$ are fully faithful functors.\\
\\
First, let us show that $\DM_{\textup{gm}}^{\textup{eff},\wedge}(k;R)\rightarrow\DM^{\textup{eff},\wedge}(k;R)$ is fully faithful. Since the limit of fully faithful functors is fully faithful, it suffices to show that for each $n\in\mathbb{Z}$, 
\[\DM_{\textup{gm}}^{\textup{eff}}(k;R)/\DM_{\textup{gm}}^{\textup{eff}}(k;R)(n)\rightarrow \DM^{\textup{eff}}(k;R)/\DM^{\textup{eff}}(k;R)(n)\]
is fully faithful. However, this is a consequence of Theorem 2.1 of \cite{Neeman}.\\
\\
We now show that the composition $\DM_{\textup{gm}}^{\textup{eff}}(k;R)\hookrightarrow\DM^{\textup{eff}}(k;R)\xrightarrow{L_{\infty}}\DM^{\textup{eff},\wedge}(k;R)$ is fully faithful. Suppose $M$ and $N$ are geometric motives in $\DM_{\textup{gm}}^{\textup{eff}}(k;R)$. We are to show that
\[\Map_{\DM_{\textup{gm}}^{\textup{eff}}(k;R)}(M,N)\rightarrow\Map_{\DM^{\textup{eff},\wedge}(k;R)}(L_{\infty}M,L_{\infty}N)\]
is a weak equivalence. However, 
\begin{eqnarray*}\Map_{\DM^{\textup{eff},\wedge}(k;R)}(L_{\infty}M,L_{\infty}N) &\simeq& \lim_{n}\Map_{\DM^{\textup{eff}}(k;R)/\DM^{\textup{eff}}(k;R)(n)}(L_nM,L_nN)\\ &\simeq&\lim_{n}\Map_{\DM^{\textup{eff}}(k;R)}(M,i_nL_nN).\end{eqnarray*}
We claim that since $N$ is geometric, for $n\gg 1$, the morphism $N\xrightarrow{\varphi}i_nL_nN$, induced by the unit of the adjunction $(L_n,i_n)$, is an equivalence, from which the result will follow. The cokernel of $\varphi$ lies in $\DM^{\textup{eff}}(k;R)(n)$. By Lemma~\ref{sufflarge}, for $n\gg 1$, all maps $A(n)\rightarrow N$, $A\in\DM^{\textup{eff}}(k;R)$ are trivial. Consider the distinguished triangle
\[N\rightarrow i_nL_nN\rightarrow\coker\varphi\rightarrow N[1].\]
Since $\coker\varphi\in\DM^{\textup{eff}}(k;R)(n)$, lemma~\ref{sufflarge} implies that $\coker\varphi\rightarrow N[1]$ is $0$ for $n\gg 1$. Consequently, the sequence splits, and so $N\oplus\coker\varphi\simeq i_nL_nN$. Since $\Map(\coker\varphi,i_nL_nN)\simeq\Map(L_n\coker\varphi,L_nN)\simeq 0$, $\coker\varphi\rightarrow i_nL_nN$ is $0$. Therefore, $\coker\varphi$ must be $0$. Consequently, $\varphi$ is an equivalence for $n\gg 1$, from which we conclude that the composition
\[\Map_{\DM_{\textup{gm}}^{\textup{eff}}(k;R)}(M,N)\rightarrow\Map_{\DM^{\textup{eff},\wedge}(k;R)}(L_{\infty}M,L_{\infty}N)\simeq\Map_{\DM^{\textup{eff}}(k;R)}(M,N)\]
is the identity morphism, as required.
\end{proof}
\begin{remark}\label{phantom}\textup{
Note that it is not true that $L_{\infty}:\DM^{\textup{eff}}(k;R)\rightarrow\DM^{\textup{eff},\wedge}(k;R)$ is fully faithful, that is, the restriction to geometric motives is necessary. Indeed, Ayoub (lemma 2.4 of \cite{Ayoub2}) constructs a (phantom) motive $F=\textup{hocolim}_{n\in\mathbb{N}}\mathbb{Q}(n)[n]$ in $\DM^{\textup{eff}}(k;\mathbb{Q})$, $k$ a field of infinite transcendence degree over its prime field, that is not equivalent to $0$. Clearly, this maps to $0$ under $L_{\infty}$. His construction is as follows. Note that $H^n(\Spec k;\mathbb{Q}(n))=K_n^M(k)\otimes\mathbb{Q}$, where $K_n^M(k)$ is the $n$-th Milnor K-group of $k$. Let $(a_n)_{n\in\mathbb{N}}$ be a sequence of elements of $k^{\times}$ that are algebraically independent. Let $a_n:\mathbb{Q}(n)[n]\rightarrow\mathbb{Q}(n+1)[n+1]$ be the map corresponding to $a_n\in k^{\times}$ modulo the isomorphism $\Hom(\mathbb{Q}(n)[n],\mathbb{Q}(n+1)[n+1])\simeq \Hom(\mathbb{Q}(0),\mathbb{Q}(1)[1])=H^1(\Spec k;\mathbb{Q}(1))=k^{\times}\otimes\mathbb{Q}$. This gives an $\mathbb{N}$-inductive system $\{\mathbb{Q}(n)[n]\}_{n\in\mathbb{N}}$, using which we obtain the object $F:=\textup{hocolim}_{n\in\mathbb{N}}\mathbb{Q}(n)[n]$. Ayoub shows that $F$ is nonzero by showing that the natural map $\alpha_{\infty}:\mathbb{Q}(0)\rightarrow F$ is nonzero. Indeed, since $\mathbb{Q}(0)$ is compact, $\alpha_{\infty}$ is zero if and only if for some $n$ the natural map $\alpha_n:\mathbb{Q}(0)\rightarrow \mathbb{Q}(n)[n]$ is zero. But this corresponds, under the identification $\Hom(\mathbb{Q}(0),\mathbb{Q}(n)[n])=K_n^M(k)\otimes\mathbb{Q}$, to the symbol $\{a_0,\hdots,a_n\}\in K_n^M(k)\otimes\mathbb{Q}$. This is nonzero as a result of the assumption that the $a_i$ are algebraically independent. As a side remark, it is called a \textit{phantom} motive because its Betti realization is $0$. This example also shows that the conservativity conjecture is false if we do not restrict to geometric motives.}
\end{remark}
Aside from completed motives, we also need \textit{convergent motives} that we now define.
\begin{definition}
An effective \textup{convergent} motive $X$ is an object of $\DM^{\textup{eff}}(k;R)$ with the property that for each $N\gg 1$, $i_NL_NX$ is $i_NL_N$ applied to a geometric effective motive. Let $\DM_{\textup{conv}}^{\textup{eff}}(k;R)$ be the full subcategory of the stable $\infty$-category $\DM^{\textup{eff}}(k;R)$ consisting of effective convergent motives. We call this stable $\infty$-category the $\infty$-category of effective \textit{convergent} motives.
\end{definition}
An example of an effective convergent motive is $M(B\mathbb{G}_m)=\bigoplus_{i=0}^{\infty}\mathbf{1}_k(i)[2i]$. We remark that there is a natural functor $\DM_{\textup{conv}}^{\textup{eff}}(k;R)\rightarrow\DM^{\textup{eff},\wedge}(k;R)$ with essential image in $\DM_{\textup{gm}}^{\textup{eff},\wedge}(k;R)$. Note that the phantom motive $F$ in remark~\ref{phantom} is convergent and effective. Furthermore, its image under $\DM_{\textup{conv}}^{\textup{eff}}(k;R)\rightarrow\DM^{\textup{eff},\wedge}(k;R)$ is $0$, and so this functor is not fully faithful. On the other hand, $\oplus_{\mathbb{N}}\mathbf{1}_k$ is effective but not convergent.
\begin{remark}\label{badslice}\textup{
Though over large base fields, say $\mathbb{C}$, the functors $i_nL_n$ do not preserve geometricity, it is not known if they do preserve geometricity if we work over finite fields or number fields \cite{Ayoubslice}. Conjecturally, this is the case at least if we work over finite fields and with rational coefficients because such mixed motives are expected to be mixed Tate motives.}
\end{remark}
Another definition that will be useful for us is the following.
\begin{definition}[Slice motives] Let $\DM_{\textup{sl}}^{\textup{eff}}(k;R)$ be the smallest stable $\infty$-subcategory of $\DM^{\textup{eff}}(k;R)$ containing all motives of the form $i_nL_nX$ for $n\geq 0$ and geometric motives $X$. We call such motives (effective) \textup{slice} motives because they are the slices of geometric effective motives.
\end{definition}
We define here the notion of virtual dimension; it will recur throughout this paper.
\begin{definition}
The virtual dimension of an object $M$ of $\DM(k;R)$, denoted by $\vdim\ M$, is the largest $n\in\mathbb{Z}\cup\{\pm\infty\}$ such that $M\in\DM^{\textup{eff}}(k;R)(n)$.
\end{definition}
\begin{remark} \textup{Note that $\vdim\ 0=+\infty$. Also, for $F$ the phantom motive in remark~\ref{phantom}, we have $\vdim\ F=\infty$, while $F\not\simeq 0$. The value $-\infty$ can also be attained: $\vdim\bigoplus_{n=-\infty}^{\infty}\mathbf{1}_k(n)[2n]=-\infty$.}
\end{remark}
\subsection{$\cal{M}(k;R)$ and injectivity from $K_0(\DM_{\textup{gm}}^{\textup{eff}}(k;R))$}
In this subsection, we define the abelian group $\cal{M}(k;R)$ in which our integrals will land. We also show that the image of the map
\[K_0(\DM_{\textup{gm}}^{\textup{eff}}(k;R))\rightarrow K_0(\DM_{\textup{sl}}^{\textup{eff}}(k;R))\]
induced by inclusion naturally injects into this group. We denote this image by $\imK$. First, let us fix some notation.\\
\\
Suppose $\cal{A}$ is a stable $\infty$-category. Denote by $F(\cal{A})$ the free abelian group on equivalence classes $[X]$ of objects $X$ of $\cal{A}$. Denote by $E(\cal{A})$ the subgroup of $F(\cal{A})$ generated by elements of the form $[X\oplus Y]-[X]-[Y]$. Let $T(\cal{A})$ be the larger subgroup of $F(\cal{A})$ generated by elements of the form $[A]-[B]+[C]$, where
\[A\rightarrow B\rightarrow C\rightarrow A[1]\]
is a cofiber sequence in $\cal{A}$. Define $G(\cal{A}):=F(\cal{A})/E(\cal{A})$ and $K_0(\cal{A}):=F(\cal{A})/T(\cal{A})$. Clearly, there is a surjection $G(\cal{A})\rightarrow K_0(\cal{A})$. Let $I(\cal{A})$ be its kernel.\\
\\
Denote by $T_{\infty}(\DM_{\textup{conv}}^{\textup{eff}}(k;R))$ the image in $F(\DM^{\textup{eff},\wedge}(k;R))$ of the smallest subgroup of $F(\DM_{\textup{conv}}^{\textup{eff}}(k;R))$ containing $T(\DM_{\textup{conv}}^{\textup{eff}}(k;R))$ and, for $W$ effective convergent, elements of the form
\[[W]\]
if $[i_nL_nW]=0$ in $K_0(\DM_{\textup{sl}}^{\textup{eff}}(k;R))$ for $n\gg 1$. We call this latter kind of relations \textit{infinitesimal} relations because as far as geometry is concerned, such terms should make no contributions. Define the abelian groups
\[\cal{M}(k;R):=F(\DM^{\textup{eff},\wedge}(k;R))/T_{\infty}(\DM_{\textup{conv}}^{\textup{eff}}(k;R))\]
and
\[\cal{M}_0(k;R):=F(\DM^{\textup{eff},\wedge}(k;R))/E(\DM_{\textup{conv}}^{\textup{eff}}(k;R)),\]
where we are abusing notation and viewing $E(\DM_{\textup{conv}}^{\textup{eff}}(k;R))$ and $T_{\infty}(\DM_{\textup{conv}}^{\textup{eff}}(k;R))$ as its image in $F(\DM^{\textup{eff},\wedge}(k;R))$. There is a natural group homomorphism
\[c_R:K_0(\DM_{\textup{sl}}^{\textup{eff}}(k;R))\rightarrow\cal{M}(k;R),\]
as well as a natural group homomorphism
\[G(\DM_{\textup{sl}}^{\textup{eff}}(k;R))\rightarrow\cal{M}_0(k;R).\]
Our integrals will take values in $\cal{M}(k;R)$. Consequently, if we want to prove results about classes of geometric motives in $K_0(\DM_{\textup{gm}}^{\textup{eff}}(k;R))$ using integration, it will be useful to at least know the following injectivity result.
\begin{proposition}\label{injectivity}
The natural homomorphism 
\[c_R|_{\imK}:\imK\rightarrow \cal{M}(k;R)\]
is injective.
\end{proposition}
\begin{proof}
Prior to showing that $c_R|_{\imK}$ is injective, we show that 
\[\imG\hookrightarrow G(\DM_{\textup{sl}}^{\textup{eff}}(k;R))\rightarrow \cal{M}_0(k;R),\]
where $\imG$ is the image of
\[G(\DM_{\textup{gm}}^{\textup{eff}}(k;R))\rightarrow G(\DM_{\textup{sl}}^{\textup{eff}}(k;R))\]
is injective. Every object of $\imG$ is of the form $[X]$ for some $X\in\DM_{\textup{gm}}^{\textup{eff}}(k;R)$. Suppose $[X]=0$ in $\cal{M}_0(k;R)$. Then there is an effective convergent motive $Z$ such that $X\oplus Z\simeq Z$ in $\DM^{\textup{eff},\wedge}(k;R)$. Since $X$ is geometric, the proof of Lemma~\ref{fullfaithfulness} implies that for $N\gg 1$, $X\rightarrow i_NL_NX$ is an equivalence. Furthermore, since $Z$ is convergent, we may choose $N$ so that it additionally satisfies $i_NL_NZ\in\DM_{\textup{sl}}^{\textup{eff}}(k;R)$. Applying the exact functor $i_NL_N$ to both sides, we obtain $X\oplus i_NL_NZ\simeq i_NL_NZ$ as objects in $\DM_{\textup{sl}}^{\textup{eff}}(k;R)$. This implies that $[X]=0$ in $G(\DM_{\textup{sl}}^{\textup{eff}}(k;R))$. This argument establishes the injectivity of
\[\imG\rightarrow \cal{M}_0(k;R).\]
We now show that $\imK\rightarrow \cal{M}(k;R)$ is injective. Note that there is a natural quotient map $\cal{M}_0(k;R)\rightarrow\cal{M}(k;R)$. Denote its kernel by $J$. The injectivity of $\imK\rightarrow \cal{M}(k;R)$ is equivalent to showing that $\imG\cap J=\imG\cap I(\DM_{\textup{sl}}^{\textup{eff}}(k;R))$. The inclusion $\imG\cap I(\DM_{\textup{sl}}^{\textup{eff}}(k;R))\subseteq \imG\cap J$ is clear. Now suppose $[X]$ is in $\imG\cap J$. Then $[X]=[A]+[C]-[B]+[W]$ in $\cal{M}_0(k;R)$ for $A\rightarrow B\rightarrow C\rightarrow A[1]$ a cofiber sequence with $A,B,C$ (in the image of) effective convergent motives in $\DM^{\textup{eff},\wedge}(k;R)$, and $W$ a convergent motive such that $[W]$ is an infinitesimal relation. Therefore, there is an effective convergent motive $U$ such that 
\[X\oplus B\oplus U\oplus\simeq A\oplus C\oplus U\oplus W.\] 
As in the previous argument, we may apply $i_NL_N$ for $N\gg 1$ such that $X\rightarrow i_NL_NX$ is an equivalence and $i_NL_NA$, $i_NL_NB$, $i_NL_NC$, $i_NL_NU$, $i_NL_NW$ are slice motives. The exactness of $i_NL_N$ gives us the equivalence
\begingroup
    \fontsize{9.5pt}{11pt}\selectfont 
\[X\oplus i_NL_NB\oplus i_NL_NU\simeq i_NL_NA\oplus i_NL_NC\oplus i_NL_NU\oplus i_NL_NW\]
\endgroup
in $\DM_{\textup{sl}}^{\textup{eff}}(k;R)$. Since $i_NL_NA\rightarrow i_NL_NB\rightarrow i_NL_NC\rightarrow i_NL_NA[1]$ is a cofiber sequence in $\DM_{\textup{sl}}^{\textup{eff}}(k;R)$, it follows that 
\[[X]=[i_NL_NA]+[i_NL_NC]-[i_NL_NB]+[i_NL_NW],\]
is an element of $\im_G\cap I(\DM_{\textup{sl}}^{\textup{eff}}(k;R))$, as required.
\end{proof}
We end this subsection by making the following important conjecture.
\begin{conjecture}[Injectivity conjecture $I(k;R)$] The natural group homomorphism
\[K_0(\DM_{\textup{gm}}^{\textup{eff}}(k;R))\rightarrow K_0(\DM_{\textup{sl}}^{\textup{eff}}(k;R))\]
is injective.
\end{conjecture}
We remark that this conjecture (in combination with Proposition~\ref{injectivity}) would imply that the natural group homomorphism
\[K_0(\DM_{\textup{gm}}^{\textup{eff}}(k;R))\rightarrow \cal{M}(k;R)\]
is injective. Hence, we would have that whenever we prove the equality of two classes of geometric effective motives in $\cal{M}(k;R)$ using integration, the classes will also be equal in $K_0(\DM_{\textup{gm}}^{\textup{eff}}(k;R))$. Clearly, the morphism in the injectivity conjecture is an isomorphism whenever we are in a setting in which the slice functors $i_nL_n$ preserve geometric objects. If the conjectural description of mixed motives over $k=\mathbb{F}_q$ and with $R=\mathbb{Q}$ is true, then the conjecture $I(\mathbb{F}_q;\mathbb{Q})$ is true. It is also believed that the slice filtration preserves geometricity if $k$ is a number field and at least if $R=\mathbb{Q}$. In any case, even if the preservation of geometricity may be false for $k$ and $R$, as in the case $k=\mathbb{C}$ and $R=\mathbb{Q}$, it still makes sense to ask if the morphism on Grothendieck groups is \textit{injective}.\\
\\
The injectivity of the analogous result for virtual motives in unknown and is a source of complications in classical motivic integration. Indeed, in classical motivic integration, integrals take values in a completion $\widehat{\cal{M}}_k$ of $K_0(\Var_k)[\mathbb{L}^{-1}]$, and it is not known if the completion map $K_0(\Var_k)[\mathbb{L}^{-1}]\rightarrow\widehat{\cal{M}}_k$ is injective. Therefore, classical motivic integration has its complications when we try to show equality of classes in $K_0(\Var_k)[\mathbb{L}^{-1}]$. In our case, integrals take values in $\cal{M}(k;R)$, and so the above proposition will help us deduce results about classes of geometric motives viewed inside $K_0(\DM_{\textup{sl}}^{\textup{eff}}(k;R))$. If we also have the injectivity conjectured above, then we would have the analogue of this conjecture for virtual motives. Let me point out that I believe that at least for $R=\mathbb{Q}$, the natural morphism 
\[K_0(\DM_{\textup{gm}}^{\textup{eff}}(k;R))\rightarrow K_0(\DM_{\textup{gm}}^{\textup{eff},\wedge}(k;R))\]
is injective. In fact, I believe the stronger result that the composition
\[K_0(\DM_{\textup{gm}}^{\textup{eff}}(k;R))\rightarrow K_0(\DM_{\textup{gm}}^{\textup{eff},\wedge}(k;R))\rightarrow\varprojlim_n K_0(\DM_{\textup{gm}}^{\textup{eff}}(k;R)/\DM_{\textup{gm}}^{\textup{eff}}(k;R)(n))\]
is injective. One main reason for developing these various forms of integration is to either have an injective analogue of such a group homomorphism or to connect such an injectivity to existing conjectures on motives. For example, as we will see, in the case of rational Voevodsky motives, we can construct a theory of integration with such an injective homomorphism if we assume the existence of a motivic $t$-structure. In the $\ell$-adic and Hodge settings, such complications do not arise and we have the optimal kind of injectivities desired. This is because the analogues of the slice functors $i_nL_n$ preserve constructibility. Therefore, if we want to prove arithmetic results or results about mixed Hodge structures (with rational coefficients), it may suffice to use our mixed $\ell$-adic or mixed Hodge integrations.
\subsection{Motivic measure and measurable subsets}
In this section, we define the motivic measure on the Jet scheme $X_{\infty}$ of a smooth $k$-variety $X$. We first define \textit{stable} subschemes of the Jet scheme and define a measure on such subschemes. We will then define \textit{good} and \textit{measurable} subsets of the Jet scheme and extend our measure to such subsets.
\begin{definition}
Let $X$ be a smooth $k$-scheme. A subscheme $A\subseteq X_{\infty}$ is said be \textup{stable} if there is an $m\in\mathbb{N}$ such that $A_m:=\pi_m(A)$ is a locally closed subscheme of $\cal{J}_m(X)$ and $A=\pi_m^{-1}(A_m)$. We shall say that such an $A$ is \textup{stable at least at the $m$th level}.
\end{definition}
We first define our motivic measure on stable subschemes. In order to do so, we need the following lemma.
\begin{lemma}\label{stabilizationlemma}
Suppose $X$ is a smooth $k$-variety of dimension $d$. If $A$ is a subscheme of $X_{\infty}$ that is stable at least at the $N$th level, then for every $m\geq N$,
\[\pi^{A_{m+1}}_!\mathbf{1}_{A_{m+1}}((m+1)d)[2(m+1)d]\simeq \pi^{A_m}_!\mathbf{1}_{A_m}(md)[2md]\]
in $\DM_{\textup{gm}}(k;R)$.
\end{lemma}

\begin{proof}
First note that both objects are constructible motives in $\DM(k;R)$ by Corollary 4.2.12 of \cite{CisDeg1}, and so they are geometric. Note that by the smoothness of $X$, for every $m\geq N$
\[A_{m+1}\xrightarrow{\pi^{m+1}_m}A_m\]
is an $\mathbb{A}^d$-bundle. By purity, we know that $\pi^{m+1}_{m,!}\mathbf{1}_{A_{m+1}}\simeq \pi^{m+1}_{m,\#}\mathbf{1}_{A_{m+1}}(-d)[-2d]$. Therefore, we obtain
\begin{eqnarray*}\pi^{A_{m+1}}_!\mathbf{1}_{A_{m+1}}((m+1)d)[2(m+1)d] &\simeq & \pi^{A_m}_!\pi^{m+1}_{m,!}\mathbf{1}_{A_{m+1}}((m+1)d)[2(m+1)d]\\ &\simeq& \pi^{A_m}_!\pi^{m+1}_{m,\#}\mathbf{1}_{A_{m+1}}(md)[2md]\\ &\simeq& \pi^{A_m}_!\mathbf{1}_{A_m}(md)[2md],
\end{eqnarray*}
where the last equivalence follows from $\mathbb{A}^1$-homotopy invariance and the fact that $\pi^{m+1}_m$ is an $\mathbb{A}^d$-bundle implying that $\pi^{m+1}_{m,\#}\mathbf{1}_{A_{m+1}}\simeq\mathbf{1}_{A_m}$. The conclusion follows.
\end{proof}
Therefore, on such stable subschemes, we can define a measure $\mu_X$ as follows.
\begin{definition}\label{volumemeasure}
For $A$ a stable subscheme of $X_{\infty}$, we define its \textup{volume} by
\[\mu_X(A):= \pi^{A_m}_!\mathbf{1}_{A_m}((m+1)d)[2(m+1)d]\in\DM_{\textup{gm}}^{\textup{eff}}(k;R)\]
for sufficiently large $m$.
\end{definition}
By Lemma~\ref{stabilizationlemma}, this is independent of $m$ for sufficiently large $m$. Furthermore, it is also geometric. For effectivity, we used the fact that $\pi^Y_!\mathbf{1}_Y\in\DM^{\textup{eff}}(k;R)(-d_Y)$, where $d_Y:=\dim Y$. We show this in the following lemma.
\begin{lemma}\label{virtdim}
Suppose $X$ is a $k$-variety of dimension $d$. Then $\vdim\ \pi^X_!\mathbf{1}_X=-d$.
\end{lemma}
\begin{proof}First assume that $X$ is a smooth $k$-variety. By purity, $\pi^X_!\mathbf{1}_X\simeq M(X)(-d)[-2d]$, and so $\vdim\ \pi^X_!\mathbf{1}_X\geq -d$. If $\vdim\pi^X_!\mathbf{1}_X>-d$, then $\vdim\ M(X)\geq 1$, that is $M(X)\in\DM_{\textup{gm}}^{\textup{eff}}(k;R)(1)$. From this we obtain $M(X_L)\in\DM_{\textup{gm}}^{\textup{eff}}(L;R)(1)$ for any finite extension $L|k$. Take a point $\Spec L\rightarrow X_L$ for some large enough $L$. From this, we see that $\mathbf{1}_L$ splits off $M(X_L)$ as a direct summand. As a result, $M(X_L)$ has virtual dimension $0$. Consequently, $\vdim\ \pi^X_!\mathbf{1}_X=-d$.\\
\\
We now prove the lemma when $X$ need not be smooth. We do so by inducting on the dimension of $X$. The lemma is true if $X$ is of dimension $0$. Suppose the lemma is true for dimensions $<d$. Let $\textup{Sing}X$ be the singular locus of $X$. Then $\dim\textup{Sing}X<d$, and so by the inductive hypothesis, $\vdim\ \pi^{\textup{Sing}X}_!\mathbf{1}_{\textup{Sing}X}=-\dim\textup{Sing}X>-d$. By the smoothness of $X\setminus\textup{Sing}X$, we have $\vdim\ \pi^{X\setminus\textup{Sing}X}_!\mathbf{1}_{X\setminus\textup{Sing}X}=-d$. Consider the localization cofiber sequence
\[\pi^{X\setminus\textup{Sing}X}_!\mathbf{1}_{X\setminus\textup{Sing}X}\rightarrow\pi^X_!\mathbf{1}_X\rightarrow\pi^{\textup{Sing}X}_!\mathbf{1}_{\textup{Sing}X}\rightarrow \pi^{X\setminus\textup{Sing}X}_!\mathbf{1}_{X\setminus\textup{Sing}X}[1]\]
in $\underline{\DM}_{cdh}(k;R)$. Twisting it $d$ times, we obtain the cofiber sequence
\[\pi^{X\setminus\textup{Sing}X}_!\mathbf{1}_{X\setminus\textup{Sing}X}(d)\rightarrow\pi^X_!\mathbf{1}_X(d)\rightarrow\pi^{\textup{Sing}X}_!\mathbf{1}_{\textup{Sing}X}(d)\rightarrow \pi^{X\setminus\textup{Sing}X}_!\mathbf{1}_{X\setminus\textup{Sing}X}(d)[1]\]
in $\underline{\DM}_{cdh}(k;R)$. In the above cofiber sequence, $\pi^X_!\mathbf{1}_X(d)$ is an extension of two effective motives, and so is itself an effective cdh-motive. Consequently, $\vdim\ \pi^X_!\mathbf{1}_X\geq -d$. Assume to the contrary that $\vdim\ \pi^X_!\mathbf{1}_X>-d$. By the inductive hypothesis, $\pi^{\textup{Sing}X}_!\mathbf{1}_X(d)\in\DM^{\textup{eff}}_{cdh}(k;R)(1)$. Then we have $\vdim\pi^{X\setminus\textup{Sing}X}_!\mathbf{1}_{X\setminus\textup{Sing}X}>-d$ from the localization sequence above, a contradiction. As a result, $\vdim\ \pi^X_!\mathbf{1}_X=-d$, as required. Note that we are using Corollary 4.0.14 of the published version of Kelly's thesis \cite{Kelly} saying that (the right adjoint of) the canonical functor
\[\DM^{\textup{eff}}(k;R)\rightarrow\underline{\DM}^{\textup{eff}}_{\textup{cdh}}(k;R)\]
is an equivalence of categories for any perfect field $k$ of characteristic exponent $p\in R^{\times}$.
\end{proof}
Later, when we want to extend $\mu_X$ to a larger collection of subsets of $X_{\infty}$ called measurable subsets, the following corollary will be used in proving that our measure on measurable subsets is well-defined (up to equivalence).
\begin{corollary}\label{virtdimbound}
Suppose $A\subseteq \bigcup_{i=1}^NA_i\subseteq X_{\infty}$, where $A$ and the $A_i$ are stable subschemes. Then
\[\vdim\mu_X(A)\geq\min_{1\leq i\leq N}\vdim\mu_X(A_i).\]
\end{corollary}
\begin{proof}
This is an easy consequence of Lemma~\ref{virtdim}, and the fact that if $X\subseteq\bigcup_{i=1}^NX_i$ is a covering of a subscheme by a finite collection of subschemes, then $\dim X\leq\max_{1\leq i\leq N}\dim X_i$.
\end{proof}
As a sanity check, we have additivity of $\mu_X$ in the following weak sense.
\begin{lemma}\label{djunion}
$\mu_X$ is additive on finite disjoint unions of stable subschemes of $X_{\infty}$. More precisely, if the stable subscheme $A=\sqcup_{i=1}^kA_i$ is a finite disjoint union of stable subschemes of $X_{\infty}$, then
\[\mu_X(A)\simeq \bigoplus_{i=1}^k\mu_X(A_i).\]
\end{lemma}
\begin{proof}
By the stability of the $A_i$, there is an $N$ such that every $A_i$ stabilizes from the $N$th level onward. Then it is easy to see that $A_N=\sqcup_{i=1}^kA_{i,N}$, and so
\begin{eqnarray*}\mu_X(A) &\simeq& \pi^{A_N}_!\mathbf{1}_{A_N}((N+1)d)[2(N+1)d]\\ &\simeq& \bigoplus_{i=1}^k\pi^{A_{i,N}}_!\mathbf{1}_{A_{i,N}}((N+1)d)[2(N+1)d]\\ &\simeq&
\bigoplus_{i=1}^k\mu_X(A_i).
\end{eqnarray*}
\end{proof}
Another lemma that will be important in the proof of the well-definedness of the measure $\mu_X$ once we extend it to measurable subsets, to be defined later, is the following \textit{compactness} result.
\begin{lemma}[Compactness lemma]\label{compactness}
Suppose $D$ and $D_n,\ n\in\mathbb{N}$, are stable subschemes of $X_{\infty}$ such that $D\subseteq\bigcup_{n\in\mathbb{N}}D_n$ and the virtual dimensions of $\mu_X(D_n)$ go to infinity. Then $D$ is contained in the union of finitely many of the $D_n$.
\end{lemma}
\begin{proof}
The proof is the same as that of Lemma 2.3 of \cite{Looijenga} but made scheme-theoretic. Suppose $D=\pi_n^{-1}\pi_n(D)$, where $\pi_n(D)$ is a locally closed $k$-subscheme of $\cal{J}_n(X)$. Such an $n$ exists because $D$ is assumed to be a stable subscheme. Assume to the contrary that $D$ cannot be covered by a finite subcollection. Since $\lim_{i\rightarrow\infty}\vdim\mu_X(D_i)=\infty$, there is a $k\in\mathbb{N}$ such that $\vdim \mu_X(D_i)>(n+2)d$ for every $i>k$. By assumption, $D\setminus\cup_{i\leq k}D_i\neq\emptyset$, and so we may choose $x_{n+1}\in\pi_{n+1}\left(D\setminus\cup_{i\leq k}D_i\right)$. Then the (scheme-theoretic) fiber $\pi_{n+1}^{-1}(x_{n+1})\subset D$. This set is not covered by finitely many of the $D_i$. Indeed, $\pi_{n+1}^{-1}(x_{n+1})$ is not covered by the $D_i$ when $i\leq k$. Furthermore, $\pi_{n+1}^{-1}(x_{n+1})\cap D_i$ has positive codimension in $D_i$ for $i>k$.\\
\\
Consequently, we can inductively construct a sequence $(x_m)_{m>n}$ such that for every $m>n$ $x_m\in\cal{J}_m(X)$, $x_{m+1}$ is above $x_m$, and $\pi_m^{-1}(x_m)$ is not coverable by finitely many of the $D_i$. This determines an element $x\in\cal{J}_{\infty}(X)$ such that $\pi_n(x)\in \pi_n(D)$. Consequently, $x\in D$ because $D$ is stable at least at level $n$. Since $D$ is covered by the $D_i$, there is a $j$ such that $x\in D_j$. $D_j$ is stable, and so is stable at least at some level $m>n$. This implies that $\pi_m^{-1}(x_m)\subseteq D_j$, contradicting the fact that $\pi_m^{-1}(x_m)$ is not finitely coverable by the $D_i$.
\end{proof}
The subsets of $X_{\infty}$ that will show up in the integration of Voevodsky motives will not necessarily be stable subschemes; they will come from subsets of $X_{\infty}$, called measurable subsets, that can be approximated by stable subschemes in the following sense.
\begin{definition}
A subset $C\subseteq X_{\infty}$ is said to be \textup{good} if there is a monotonic sequence (the inclusions are locally closed embeddings of $k$-schemes) of stable subschemes
\[C_0\supseteq\hdots\supseteq C_n\supseteq C_{n+1}\supseteq\hdots\ (\text{or }C_0\subseteq\hdots\subseteq C_n\subseteq C_{n+1}\subseteq\hdots)\]
of $X_{\infty}$ containing (resp. contained in) $C$, and stable $C_{n,i}$, $i,n\in\mathbb{N}$ such that for every $n$
\[C_n\setminus C\subseteq\bigcup_{i\in\mathbb{N}}C_{n,i}\ \left(\text{resp. }C\setminus C_n\subseteq\bigcup_{i\in\mathbb{N}}C_{n,i}\right)\]
such that for every $n,i$, 
\[n\leq\vdim\mu_X(C_{n,i})\]
and
\[\vdim\mu_X(C_{n,i})\xrightarrow{i\rightarrow\infty}\infty.\]
We then define the \textup{volume} of $C$ as the object
\[\mu_X(C):=(\mu_X(C_n))_n\in\DM_{\textup{gm}}^{\textup{eff},\wedge}(k;R).\]
We call a pair $(C,\cal{S})$ consisting of a subset $C\subseteq X_{\infty}$ and a \textup{finite} decomposition $\cal{S}$ ($C=\sqcup_{i\in\cal{S}}C_i$, $C_i\subseteq X_{\infty}$) \textup{measurable} if each $C_i$ is a good subset. We then let
\[\mu_X(C,\cal{S}):=\bigoplus_{i\in\cal{S}}\mu_X(C_i).\]
We view a good subset $C$ without a prescribed decomposition as a measurable subset with the trivial decomposition $\cal{S}=\{C\}$.
\end{definition}
Just as in classical measure theory, the question arises as to whether this measure is well-defined (up to equivalence).
\begin{proposition}\label{welldefinedmeasure}
This measure $\mu_X$ is well-defined up to equivalence, that is, for a fixed measurable subset $(C,\cal{S})$, any two sets of data in the above definition give rise to volumes that are equivalent in $\DM_{\textup{gm}}^{\textup{eff},\wedge}(k;R)$.
\end{proposition}
\begin{proof}
Since the decomposition is fixed, we may assume without loss of generality that our measurable subset $(C,\cal{S})$ is a good subset with the trivial decomposition. Suppose $(C_n)_n$ and $(D_n)_n$ are two monotonic sequences as in the definition above. We are to show that
\[(\mu_X(C_n))_n\simeq(\mu_X(D_n))_n\]
in $\DM_{\textup{gm}}^{\textup{eff},\wedge}(k;R)$.\\
\\
Suppose first that both are decreasing sequences. Note that $(C_n\cap D_n)_n$ is also a decreasing sequence satisfying the properties in the definition above. Indeed, $(C_n\cap D_n)\setminus C\subseteq \bigcup_{i\in\mathbb{N}}C_{n,i}$. Therefore, we may assume without loss of generality that for each $n$, $C_n\subseteq D_n$. If both sequences are increasing, then a similar argument shows that we may assume without loss of generality that $C_n\subseteq D_n$. If $(C_n)_n$ is increasing and $(D_n)_n$ is decreasing, it is then clear that $C_n\subseteq C\subseteq D_n$. Therefore, in all three cases, we may assume without loss of generality that $C_n\subseteq D_n$.\\
\\
It suffices to show that there are equivalences $\mu_X(C_n)\simeq\mu_X(D_n)$ in $\DM_{\textup{gm}}^{\textup{eff}}(k;R)/\DM_{\textup{gm}}^{\textup{eff}}(k;R)(n)$ that are compatible with each other as $n$ varies. This can be done by first writing $C_n\hookrightarrow D_n$ as the composition of an open embedding $C_n\hookrightarrow K_n$ followed by a closed embedding $K_n\hookrightarrow D_n$, where $K_n$ is the inverse image under $\pi_m$ of a locally closed $k$-subvariety of $\cal{J}_m(X)$ for some $m$. Then, we may use the localization sequences. Indeed, in $\DM_{\textup{gm}}^{\textup{eff}}(k;R)$ we have cofiber sequences
\[\mu_X(C_n)\rightarrow\mu_X(K_n)\rightarrow\mu_X(K_n\setminus C_n)\rightarrow\mu_X(C_n)[1]\]
\[\mu_X(D_n\setminus K_n)\rightarrow\mu_X(D_n)\rightarrow\mu_X(K_n)\rightarrow\mu_X(D_n\setminus K_n)[1]\]
from localization sequences. Here, we are abusing notation and identifying the stable subsets as locally closed subschemes of $\cal{J}_N(X)$ for some $N\gg 1$. Here, we can assume that they are cofiber sequences in $\DM(k;R)$ because we have the cofiber sequences in $\underline{\DM}_{\textup{cdh}}(k;R)\simeq\DM(k;R)$ (see subsection~\ref{mixedmotives}). Furthermore, $\DM_{\textup{gm}}^{\textup{eff}}(k;R)$ is a triangulated full subcategory of $\DM(k;R)$ and the objects in the above cofiber sequences are in $\DM_{\textup{gm}}^{\textup{eff}}(k;R)$ by the discussion immedaitely after Definition~\ref{volumemeasure}.\\
\\
Note that $K_n\setminus C_n\subseteq D_n\setminus C_n$. If $(C_n)$ and $(D_n)$ are both increasing, then $D_n\setminus C_n\subseteq C\setminus C_n\subseteq\bigcup_{i\in\mathbb{N}}C_{n,i}$. $K_n\setminus C_n$ is a stable subscheme of $X_{\infty}$, and so Lemma~\ref{compactness} implies that there are finitely many $C_{n,i}$ such that
\[K_n\setminus C_n\subseteq\bigcup_{i=1}^NC_{n,i}.\]
Corollary~\ref{virtdimbound} implies that the virtual dimension of $\mu_X(K_n\setminus C_n)$ is at least the minimum of the virtual dimensions of the $\mu_X(C_{n,i})$ which is at least $n$. Therefore, $\mu_X(K_n\setminus C_n)\in\DM_{\textup{gm}}^{\textup{eff}}(k;R)(n)$. As a result,
\[\mu_X(C_n)\rightarrow\mu_X(K_n)\]
is an equivalence in $\DM_{\textup{gm}}^{\textup{eff}}(k;R)/\DM_{\textup{gm}}^{\textup{eff}}(k;R)(n)$. Similarly, using $D_n\setminus K_n\subseteq C\setminus K_n\subseteq C\setminus C_n$, we deduce that
\[\mu_X(D_n)\rightarrow\mu_X(K_n)\]
is also an equivalence in $\DM_{\textup{gm}}^{\textup{eff}}(k;R)/\DM_{\textup{gm}}^{\textup{eff}}(k;R)(n)$. Consequently, we have equivalences
\[\xymatrix{\mu_X(C_n)\ar@{->}[d]^[@!-90]{\sim}\\ \mu_X(K_n)\\ \mu_X(D_n)\ar@{->}[u]_[@!-90]{\sim}\\
}\]
in $\DM_{\textup{gm}}^{\textup{eff}}(k;R)/\DM_{\textup{gm}}^{\textup{eff}}(k;R)(n)$. We can similarly construct morphisms
\[\mu_X(C_n)\rightarrow\mu_X(C_{n+1})\]
that are equivalences in $\DM_{\textup{gm}}^{\textup{eff}}(k;R)/\DM_{\textup{gm}}^{\textup{eff}}(k;R)(n)$. As are result, we have a diagram
\[\xymatrix{\mu_X(C_n)\ar@{->}[d]^[@!-90]{\sim} \ar@{->}[r]^{\sim} & \mu_X(C_{n+1})\ar@{->}[d]^[@!-90]{\sim}\\ \mu_X(K_n) & \mu_X(K_{n+1})\\ \mu_X(D_n)\ar@{->}[u]_[@!-90]{\sim} \ar@{.>}[r] & \mu_X(D_{n+1}) \ar@{->}[u]_[@!-90]{\sim}
}\]
in $\DM_{\textup{gm}}^{\textup{eff}}(k;R)/\DM_{\textup{gm}}^{\textup{eff}}(k;R)(n)$. We can fill in the dotted arrows (with necessarily an equivalence in $\DM_{\textup{gm}}^{\textup{eff}}(k;R)/\DM_{\textup{gm}}^{\textup{eff}}(k;R)(n)$) so that the diagram commutes. We have shown that if $(C_n)$ and $(D_n)$ are both increasing, then we can find equivalences $\mu_X(C_n)\xrightarrow{\sim}\mu_X(D_n)$ in $\DM_{\textup{gm}}^{\textup{eff}}(k;R)/\DM_{\textup{gm}}^{\textup{eff}}(k;R)(n)$ that are compatible as $n$ varies. Equivalently, we have established the existence of an equivalence $(\mu_X(C_n))_n\simeq(\mu_X(D_n))_n$ in $\DM_{\textup{gm}}^{\textup{eff},\wedge}(k;R)$, as required. The other two cases can be easily dealt with via variants of this proof.
\end{proof}
Some examples of measurable subsets of $X_{\infty}$ are stable subschemes, subschemes of the form $\cal{J}_{\infty}(Y)$ for any $k$-subvariety $Y\subseteq X$, and pairs of the form $(\pi_n^{-1}Z,\pi_n^{-1}\cal{S})$, where $(Z,\cal{S})$ is a constructable subset of $\cal{J}_n(X)$ for some $n$ with a prescribed finite decomposition $\cal{S}$ into locally closed subschemes. Here, if $\cal{S}=\{S_1,\hdots,S_k\}$, then $\pi_n^{-1}\cal{S}:=\{\pi_n^{-1}S_1,\hdots,\pi_n^{-1}S_k\}$. Another example of a measurable subset of $X_{\infty}$ is the \textit{scheme-theoretic} disjoint union of countably many stable subschemes whose virtual dimensions go to $\infty$ and such that the union of the the first $N$ objects, $N$ any integer, is also a stable subscheme; see lemma~\ref{measdisj}. In the following lemma, we compute the measure of $\cal{J}_{\infty}(Y)$ when $Y$ is a locally closed $k$-subvariety of $X$ of dimension less than that of $X$.

\begin{lemma}\label{measurezero}
Suppose $Y$ is a subvariety of a smooth $k$-scheme $X$ of strictly positive codimension. Then $\cal{J}_{\infty}(Y)$ is good (and so measurable) with $\mu_X(\cal{J}_{\infty}(Y))\simeq 0$ in $\DM_{\textup{gm}}^{\textup{eff},\wedge}(k;R)$.
\end{lemma}

\begin{proof}
Note that 
\[\cal{J}_{\infty}(Y)=\bigcap_{n=0}^{\infty}\pi_n^{-1}\cal{J}_n(Y).\]
which is the intersection of the following decreasing sequence
\[\pi_0^{-1}\cal{J}_0(Y)\supseteq\pi_1^{-1}\cal{J}_1(Y)\supseteq\hdots\]
of stable subschemes of $\cal{J}_{\infty}(X)$ containing $\cal{J}_{\infty}(Y)$.\\
\\
By the main result of \cite{Green3}, there is a positive integer $e$ such that for all $m$ sufficiently large,
\[\pi_{\lfloor m/e\rfloor}\cal{J}_{\infty}(Y)=\pi^m_{\lfloor m/e\rfloor}\cal{J}_m(Y).\]
Choose the decreasing sequence
\[\pi_0^{-1}\cal{J}_0(Y)\supseteq\pi_{e}^{-1}\cal{J}_e(Y)\supseteq\pi_{2e}^{-1}\cal{J}_{2e}(Y)\supseteq\pi_{3e}^{-1}\cal{J}_{3e}(Y)\supseteq\hdots\]
We show that for each $n$, $\mu_X(\pi_{ne}^{-1}\cal{J}_{ne}(Y))\in\DM_{\textup{gm}}^{\textup{eff}}(k;R)(nc)$, where $c$ is the codimension of $Y$ in $X$. We follow a well-known argument in classical motivic integration. By lemma 4.3 of \cite{DenLoe}, $\dim \pi_n(\cal{J}_{\infty}(Y))\leq (n+1)\dim Y$.
As a result, for $m\gg 0$
\begin{eqnarray*}
\dim\cal{J}_m(Y)&\leq&\dim\pi^m_{\lfloor m/e\rfloor}\cal{J}_m(Y)+(m-\lfloor m/e\rfloor)\dim X\\&=& \dim\pi_{\lfloor m/e\rfloor}\cal{J}_{\infty}(Y)+(m-\lfloor m/e\rfloor)\dim X\\ &\leq& (\lfloor m/e\rfloor+1)\dim Y+(m-\lfloor m/e\rfloor)\dim X\\&=&(m+1)\dim X-(\lfloor m/e\rfloor+1)(\dim X-\dim Y)
\end{eqnarray*}
Using this inequality, for $n\gg 0$
\begin{eqnarray*}
\vdim\mu_X(\pi_{ne}^{-1}\cal{J}_{ne}(Y))&=&\vdim\pi^{\cal{J}_{ne}(Y)}_!\mathbf{1}_{\cal{J}_{ne}(Y)}((ne+1)\dim X)[2(ne+1)\dim X]\\ &=& (ne+1)\dim X-\dim\cal{J}_{ne}(Y)\\&\geq& (ne+1)\dim X-(ne+1)\dim X+(n+1)c\\&=&(n+1)c.
\end{eqnarray*} 
Consequently, $\mu_X(\pi_n^{-1}\cal{J}_n(Y))\in\DM_{\textup{gm}}^{\textup{eff}}(k;R)(nc)$. Since by assumption $c\geq 1$, $\mu_X(\cal{J}_{\infty}(Y))\simeq 0$ in $\DM_{\textup{gm}}^{\textup{eff},\wedge}(k;R)$.
\end{proof}

\begin{lemma}\label{measdisj}
Suppose $C=\bigsqcup_{n=0}^{\infty}C_n$ as subschemes of $X_{\infty}$, where $C_n$ are stable subschemes of $X_{\infty}$ with $\vdim\mu_X(C_n)\rightarrow\infty$. Then $C$ is measurable and
\[\mu_X(C)\simeq\bigoplus_{n=0}^{\infty}\mu_X(C_n).\]
\end{lemma}
\begin{proof}
Choose the sequence $C_0\subseteq C_0\sqcup C_1\subseteq\hdots\subseteq C_0\sqcup C_1\sqcup\hdots\sqcup C_n\subseteq\hdots\subseteq C$
of stable subschemes of $X_{\infty}$, and use lemma~\ref{djunion}.
\end{proof}
We end this subsection by making the observation that if $C$ is a stable subscheme with a finite (locally closed) decomposition $\cal{S}$ into stable subschemes $C_i$ given by $C=\sqcup_{i\in\cal{S}}C_i$, then $\mu_X(C)$ is given by a sequence of extensions of $\mu_X(C_i)$, while $\mu_X(C,\cal{S})=\bigoplus_i\mu_X(C_i)$ is given by the trivial extensions. Therefore, \textit{up to a sequence of extensions}, $\mu_X(C)$ and $\mu_X(C,\cal{S})$ are the same.
\subsection{Measurable functions, integrals, and computations}
After the previous preparatory sections, we give here the definition of measurable functions and define our integrals. We also do a simple computation that will be important to us in our applications.
\begin{definition}\label{measurablefunction}
A function $F:X_{\infty}\rightarrow\mathbb{N}_{\geq 0}\cup\left\{\infty\right\}$ is \textup{measurable} if for each $s\in\mathbb{N}_{\geq 0}$, $F^{-1}(s)$ is measurable, and $F^{-1}(\infty)$ has measure $0$.
\end{definition}
An important example of a measurable function for us is $\ord_Y$, where $i:Y\hookrightarrow X$ is a closed $k$-subvariety of $X$ of positive codimension. This function is defined as follows. The associated ideal sheaf $\cal{I}_{Y/X}$ fits into a short exact sequence
\[0\rightarrow\cal{I}_{Y/X}\rightarrow\cal{O}_X\rightarrow i_*\cal{O}_Y\rightarrow 0.\]
Given $\gamma\in\cal{J}_{\infty}(X)$, we may view it as a morphism $\gamma:\Spec k(\gamma)[[t]]\rightarrow X$ or $\gamma^*:\cal{O}_X\rightarrow k(\gamma)[[t]]$, where $k(\gamma)$ is the residue field of $\gamma\in\cal{J}_{\infty}(X)$. Then $\ord_Y(\gamma)$ is defined as the largest number $e\in\mathbb{N}_{\geq 0}\cup\{\infty\}$ such that the composition
\[\cal{O}_X\xrightarrow{\gamma^*}k(\gamma)[[t]]\rightarrow k(\gamma)[[t]]/(t^{e})\]
sends $\cal{I}_{Y/X}$ to zero. Note that by definition, the order function depends only on the isomorphism class of the ideal sheaf $\cal{I}_{Y/X}$. The fact that $\ord_Y^{-1}(\infty)$ has measure zero when $Y$ has strictly positive codimension in $X$ immediately follows from Lemma~\ref{measurezero}.
\begin{definition}Given a measurable function $F:\cal{J}_{\infty}(X)\rightarrow\mathbb{N}_{\geq 0}\cup\left\{\infty\right\}$, define the motivic integral
\[\int_{X_{\infty}}\mathbf{1}_k(F)[2F]d\mu_X:=\left[\bigoplus_{s=0}^{\infty}\mu_X(F^{-1}(s))(s)[2s]\right]\]
as an element of $\cal{M}(k;R)$.
\end{definition}
The reason it makes sense for (the image of the convergent motive) $\bigoplus_{s=0}^{\infty}\mu_X(F^{-1}(s))(s)[2s]$ to be viewed as an object of $\DM_{\textup{gm}}^{\textup{eff},\wedge}(k;R)$ is that $\mu_X(F^{-1}(s))(s)[2s]\in\DM_{\textup{gm}}^{\textup{eff},\wedge}(k;R)(s)$. Note that $\DM_{\textup{gm}}^{\textup{eff},\wedge}(k;R)$ is not closed under arbitrary small colimits just as $\DM_{\textup{gm}}^{\textup{eff}}(k;R)$ is not. Also, note that $\mathbf{1}_k(F)[2F]$ is the analogue of functions of the form $\mathbb{L}^{-F}$ in classical motivic integration, and that the classical motivic integral of $\mathbb{L}^{-F}$ is defined by
\[\int_{X_{\infty}}\mathbb{L}^{-F}d\mu_X:=\sum_{s=0}^{\infty}\mu_X(F^{-1}(s))\mathbb{L}^{-s}\]
as an element of $\widehat{\cal{M}}_k$, where $\mu_X$ here is the measure in classical motivic integration. In classical motivic integration, one reason we must complete $K_0(\Var_k)[\mathbb{L}^{-1}]$ is so that we can talk about infinite sums as above. In our case, we do not complete, but work categorically so that we have a notion of infinite direct sums. This infinite direct sum cannot be taken in a presentable category like $\DM^{\textup{eff}}(k;R)$ because passing to $K_0$ would give us a value group that is zero; this is the reason we consider completed Voevodsky motives. Finally, note that both definitions of integration above are completely analogous to the way Lebesgue integration is defined in real analysis.\\
\\
Of great importance to us in this paper is a formula that will allow us to understand how integrals change via resolution of singularities. In this section, we shall prove that if $f:X\rightarrow Y$ is a proper birational morphism of smooth $k$-varieties with $K_{X/Y}$ its relative canonical divisor, and $D\subseteq Y$ is an effective divisor, then we have an analogue of the transformation rule in classical motivic integration 
\[\int_{Y_{\infty}}\mathbf{1}_k(\ord_D)[2\ord_D]d\mu_Y=\int_{X_{\infty}}\mathbf{1}_k(\ord_{f^{-1}D+K_{X/Y}})[2\ord_{f^{-1}D+K_{X/Y}}]d\mu_X.\]
This is the analogue of the transformation rule in classical motivic integration:
\[\int_{Y_{\infty}}\mathbb{L}^{-\ord_D}d\mu_Y=\int_{X_{\infty}}\mathbb{L}^{-\ord_{f^{-1}D+K_{X/Y}}}d\mu_X.\]
Prior to proving the transformation rule in our setting in the next subsection, we do a computation and compare it to its analogue in classical motivic integration. This computation, though simple, will be used in our applications.
\begin{example}\label{emptyex}
Set $X$ to be a smooth $k$-variety and $F=\ord_{\emptyset}$. Then
\[\int_{X_{\infty}}\mathbf{1}_k(\ord_{\emptyset})[2\ord_{\emptyset}]d\mu_X=[M(X)].\]
Indeed, $\ord_{\emptyset}^{-1}(s)=0$ for $s>0$ and $\ord_{\emptyset}^{-1}(0)=X_{\infty}$. Therefore,
\[\int_{X_{\infty}}\mathbf{1}_k(\ord_{\emptyset})[2\ord_{\emptyset}]d\mu_X=\left[\bigoplus_{s=0}^{\infty}\mu_X(\ord_{\emptyset}^{-1}(s))(s)[2s]\right]=[\mu_X(X_{\infty})]=[M(X)]\]
in $\cal{M}(k;R)$. This is analogous to the calculation in classical motivic integration that
\[\int_{X_{\infty}}\mathbb{L}^{-\ord_{\emptyset}}d\mu_X=[X]\]
in $\widehat{\cal{M}}_k$.
\end{example}
\subsection{The transformation rule}
The aim of this subsection is to prove the promised transformation rule. Recall that we are assuming that $k$ has characteristic exponent invertible in the coefficient ring $R$.
\begin{theorem}\label{changeofvar}
Suppose $f:X\rightarrow Y$ is a proper birational morphism of smooth $k$-varieties with $K_{X/Y}$ its relative canonical divisor, and let $D\subset Y$ be an effective divisor. Then
\[\int_{Y_{\infty}}\mathbf{1}_k(\ord_D)[2\ord_D]d\mu_Y=\int_{X_{\infty}}\mathbf{1}_k(\ord_{f^{-1}D+K_{X/Y}})[2\ord_{f^{-1}D+K_{X/Y}}]d\mu_X.\]
\end{theorem}
As we will shortly see, this is an analogue of the change of variables formula in calculus. In the next section, this formula will be crucial to our applications regarding K-equivalent varieties.\\
\\
Prior to proving the transformation rule above, we recall the first fundamental exact sequence for K\"ahler differentials and related topics.\\
\\
Given a morphism of smooth $k$-schemes $f:X'\rightarrow X$, we have an $\cal{O}_{X'}$-linear morphism $f^*\Omega^1_X\xrightarrow{df}\Omega^1_{X'}$ (of locally free sheaves on $X'$) that fits into the exact sequence
\[f^*\Omega^1_X\xrightarrow{df}\Omega^1_{X'}\rightarrow\Omega^1_{X'/X}\rightarrow 0\]
of locally free sheaves. Suppose now that $f$ is \textit{birational} and let $d=\dim X=\dim X'$ as before. Since $f$ is birational, $X$ and $Y$ have the same function fields. The stalk of $\Omega^1_{X'/X}$ at the generic point is isomorphic to $\Omega^1_{K(X')/K(X)}=0$. Therefore, taking stalks at the generic point for $f^*\Omega^1_X\xrightarrow{df}\Omega^1_{X'}$ gives a surjection of finite-dimensional vector spaces of the same dimension ($X$ and $X'$ are birational), and so an isomorphism. Therefore, when $f$ is birational, we obtain the short exact sequence
\[0\rightarrow f^*\Omega^1_X\xrightarrow{df}\Omega^1_{X'}\rightarrow\Omega^1_{X'/X}\rightarrow 0.\]
Taking the $d$th exterior power of $df$, we obtain the exact sequence
\[0\rightarrow f^*\omega_X\rightarrow\omega_{X'}\]
of line bundles, where $\omega_X=\Omega_X^d$ and $\omega_{X'}=\Omega_{X'}^d$ denote the canonical bundles of $X$ and $X'$ (over $k$). Tensoring by $\omega_{X'}^{-1}$, we obtain the exact sequence
\[0\rightarrow f^*\omega_X\otimes\omega_{X'}^{-1}\rightarrow \cal{O}_{X'}\]
of line bundles. Therefore, we may view $f^*\omega_X\otimes\omega_{X'}^{-1}$ as a locally principal ideal of $\cal{O}_{X'}$, which we denote by $J_{X'/X}$. Let $K_{X'/X}$ be the Cartier divisor given locally by the vanishing of $J_{X'/X}$. This is called the relative canonical divisor of $f$. Note that $J_{X'/X}$ is locally given by the vanishing of $\det df$.\\
\\
Suppose now that $L|k$ is a field extension. Let $\gamma:\Spec L[[t]]\rightarrow X'$ be an $L$-point of $\cal{J}_{\infty}(X')$ with $\ord_{K_{X'/X}}(\gamma)=e$. This is equivalent to $\gamma^*J_{X'/X}=(t^e)\subseteq L[[t]]$. Since $J_{X'/X}$ is locally defined by the vanishing of $\det df$, this is equivalent to $\gamma^*(\det df)=(t^e)$, or equivalently, $\det (\gamma^*(df))=(t^e)$.\\
\\
The pillar on which the proof of the transformation rule rests is the following proposition due to Denef and Loeser. From context, it will be clear what $\pi_m$ and $\pi^n_m$ mean; we abuse notation in the following lemma.
\begin{proposition}\label{impprop}(Lemma 3.4 of \cite{DenLoe})
Suppose $f:X'\rightarrow X$ is a proper birational morphism of smooth $k$-varieties. Let $C_e':=\ord_{K_{X'/X}}^{-1}(e)$, where $K_{X'/X}$ is the relative canonical divisor of $f$. Let $C_e:=f_{\infty}C_e'$. Let $L|k$ be a field extension. Let $\gamma\in C_e'(L)$ be an $L$-point, that is, a map $\gamma^*:\cal{O}_{X'}\rightarrow L[[t]]$ such that $\gamma^*(J_{X'/X})=(t^e)$. Then for $m\geq 2e$ one has:
\begin{enumerate}[(a)]
	\item For all $\xi\in\cal{J}_{\infty}(X)$ such that $\pi_m(\xi)=f_m(\pi_m(\gamma))$, there is $\gamma'\in\cal{J}_{\infty}(X')$ such that $f_{\infty}(\gamma')=\xi$ and $\pi_{m-e}(\gamma')=\pi_{m-e}(\gamma)$. In particular, the fiber of $f_m$ over $f_m(\gamma_m)$ lies in the fiber of $\pi^m_{m-e}$ over $\pi_{m-e}(\gamma)$.
	\item $\pi_m(C_e')$ is a union of fibers of $f_m$.
	\item The map $f_m:\pi_m(C_e')\rightarrow\pi_m(C_e)$ is a piecewise trivial $\mathbb{A}^e$-fibration.
\end{enumerate}
\end{proposition}
For the proof of the transformation rule, we also need the following two lemmas.
\begin{lemma}\label{bij}
If $f:X'\rightarrow X$ is a proper birational morphism of smooth $k$-varieties, then away from a measure zero closed subscheme, $f_{\infty}:\cal{J}_{\infty}(X')\rightarrow\cal{J}_{\infty}(X)$ is bijective.
\end{lemma}
\begin{proof}
The proof is more or less that given in classical motivic integration. Let $Z\subset X'$ be a proper closed $k$-subvariety on whose complement $f$ is an isomorphism. Given $\gamma\in\cal{J}_{\infty}(X)$, we denote by $k(\gamma)$ its residue field. Then $\gamma$ can be viewed as a morphism $\gamma:\Spec k(\gamma)[[t]]\rightarrow X$. We show that any $\gamma$ which does not lie entirely in $f(Z)$ uniquely lifts to an arc in $X'$. We apply the valuative criterion for properness to show this. Consider the following diagram.
\[\xymatrix{\Spec(k(\gamma)[[t]]_{(0)}) \ar@{.>}[r] \ar@{->}[d] \ar@{->}[rrd] & X' \ar@{->}[d] & X'\setminus Z \ar@{_{(}->}[l] \ar@{->}[d]^{\simeq} \\ \Spec(k(\gamma)[[t]]) \ar@{->}[r]^{\gamma} \ar@{.>}[ur] & X & X\setminus f(Z). \ar@{_{(}->}[l]}\]
By assumption, the generic point $\Spec(k(\gamma)[[t]]_{(0)})$ is in $X\setminus f(Z)$. Since $f$ is an isomorphism from $X'\setminus Z\rightarrow X\setminus f(Z)$, this generic point lifts to $X'$ uniquely (the upper left horizontal dashed arrow). $f$ is proper, and so the valuative criterion for properness yields the unique existence of the dotted arrow $\Spec(k(\gamma)[[t]])\rightarrow X'$. Consequently, the morphism
\[f_{\infty}:\cal{J}_{\infty}(X')\setminus\cal{J}_{\infty}(Z)\rightarrow \cal{J}_{\infty}(X)\setminus\cal{J}_{\infty}(f(Z)),\]
of $k$-schemes is bijective.
\end{proof}
\begin{lemma}
If $f:X'\rightarrow X$ is a proper birational morphism of smooth $k$-varieties, then for every $m$, $f_m:\cal{J}_m(X')\rightarrow\cal{J}_m(X)$ is surjective. 
\end{lemma}
\begin{proof}
Suppose $Z\subset X$ is a proper closed $k$-subvariety such that $f:X'\setminus f^{-1}(Z)\rightarrow X\setminus Z$ is an isomorphism of schemes. From the previous Lemma~\ref{bij}, we know that if $\gamma\notin\cal{J}_{\infty}(Z)$, then it lies in the image of $f_{\infty}$. If $\gamma\in\cal{J}_m(X)$, then $\pi^{-1}_m(\gamma)$ cannot be contained entirely in $\cal{J}_{\infty}(Z)$ because the latter has measure zero. As a result, there is an element in $\pi^{-1}_m(\gamma)$ that is in the image of $f_{\infty}$. In particular, there is an element in $\cal{J}_m(X')$ that maps to $\gamma\in\cal{J}_m(X)$, and so $f_m$ is surjective.
\end{proof}

\begin{proof}[Proof of the transformation rule] Using the above proposition and lemmas, we are now ready to deduce the transformation rule. Let $C_{\leq e}'=\ord_{K_{X'/X}}^{-1}([0,e])$ and let $C_e'=\ord_{K_{X'/X}}^{-1}(e)$. Note that $\cal{J}_{\infty}(K_{X'/X})$ has measure $0$ by Lemma~\ref{measurezero}. Up to (removing) a measure zero closed subscheme, therefore,
\[C_{\leq 0}'\subseteq C_{\leq 1}'\subseteq\hdots\subseteq\cal{J}_{\infty}(X')\] is a filtration of $\cal{J}_{\infty}(X')$ by locally closed subschemes. We can refine the filtration into even smaller locally closed subschemes according to the order of contact along $f^{-1}D$: set
\[C_{\leq e,\geq k}':=C_{\leq e}'\cap\ord_{f^{-1}D}^{-1}([k,\infty])\text{ and }C_{\leq e,\geq k}:=f_{\infty}(C_{\leq e,\geq k}').\]
We can similarly define $C'_{e,\geq k}$, $C_{e,\geq k}$, $C'_{e,k}$, and $C_{e,k}$. We have the filtration (up to measure zero) of $\ord_D^{-1}([k,\infty])$:
\[(C_{\leq 0,\geq k},\cal{S}_{\leq 0,\geq k})\subseteq (C_{\leq  1,\geq k},\cal{S}_{\leq 1,\geq k})\subseteq\hdots\subseteq (C_{\leq n,\geq k},\cal{S}_{\leq n,\geq k})\subseteq\hdots\subseteq\ord_D^{-1}([k,\infty]),\]
where the $\cal{S}_{\leq e,\geq k}$, as $e\geq 0$ varies, form an increasing sequence of finite decompositions of the constructable subsets $C_{\leq e,\geq k}$ into stable subschemes. By the \textit{infinitesimal relation}, we can write
\[\left[\bigoplus_{k=0}^{\infty}\mu_X(\ord_D^{-1}([k,\infty]))(k)[2k]\right]=\left[\colim_e\bigoplus_k\mu_X(C_{\leq e,\geq k},\cal{S}_{\leq e,\geq k})(k)[2k]\right].\]
Using Proposition~\ref{impprop}, we know that $C_{e,k}$ are constructable subsets of $\cal{J}_{\infty}(X)$. Note that $f_{\infty}:C_{e,k}'\rightarrow C_{e,k}$ is a piecewise trivial $\mathbb{A}^e$-bundle by Proposition~\ref{impprop}. Since $C_{e,k}$ is constructable, choose a finite decomposition into stable subschemes, say $\cal{S}=\{C_i\}_i$, such that atop each $C_i$, $f_{\infty}:C_{e,k}'\rightarrow C_{e,k}$ is a trivial $\mathbb{A}^e$-bundle. Furthermore, $\cal{S}':=\{f_{\infty}^{-1}(C_i)\}_i$ is a decomposition of $C_{e,k}'$ into stable subschemes. We may assume without loss of generality that above each object of $\cal{S}_{e,k}$, $f_{\infty}$ is an $\mathbb{A}^e$-bundle.\\
\\
In $\cal{M}(k;R)$, we have
\begin{eqnarray*}\left[\bigoplus_k\mu_X(\ord_D^{-1}(k))(k)[2k]\right]&=&\left[\colim_e\bigoplus_k\mu_X(C_{\leq e,\geq k},\cal{S}_{\leq e,\geq k})(k)[2k]\right]\\&-&\left[\colim_e\bigoplus_k\mu_X(C_{\leq e,\geq k+1},\cal{S}_{\leq e,\geq k+1})(k)[2k]\right]\\&=&\left[\bigoplus_k\bigoplus_e\mu_X(C_{e,k},\cal{S}_{e,k})(k)[2k]\right].
\end{eqnarray*}
By assumption, atop each object of $\cal{S}_{e,k}$, $f_{\infty}$ is an $\mathbb{A}^e$-bundle. Consequently,
\[\mu_{X'}(C_{e,k}',f_{\infty}^{-1}\cal{S}_{e,k})(e)[2e]=\bigoplus_{C\in\cal{S}_{e,k}}\mu_{X'}(f_{\infty}^{-1}(C))(e)[2e]\simeq \bigoplus_{C\in\cal{S}_{e,k}}\mu_X(C)=\mu_X(C_{e,k},\cal{S}_{e,k}).\] 
From this, we obtain
\begin{eqnarray*}
\int_{X_{\infty}}\mathbf{1}_k(\ord_D)[2\ord_D]d\mu_X &=& \left[\bigoplus_k\mu_X(\ord_D^{-1}(k))(k)[2k]\right]\\ &=& \left[\bigoplus_k\left(\bigoplus_e\mu_X(C_{e,k},\cal{S}_{e,k})\right)(k)[2k]\right]\\ &=& \left[\bigoplus_{k,e}\mu_{X'}(C_{e,k}',f_{\infty}^{-1}(\cal{S}_{e,k}))(k+e)[2(k+e)]\right]\\ &=& \left[\bigoplus_t\left(\bigoplus_{k+e=t}\mu_{X'}(C_{e,k}',f_{\infty}^{-1}(\cal{S}_{e,k}))\right)(t)[2t]\right]\\ &=& \left[\bigoplus_t\mu_{X'}(\ord_{f^{-1}D+K_{X'/X}}^{-1}(t))(t)[2t]\right]\\ &=& \int_{X'_{\infty}}\mathbf{1}_k(\ord_{f^{-1}D+K_{X'/X}})[2\ord_{f^{-1}D+K_{X'/X}}]d\mu_{X'},
\end{eqnarray*}
as required.
\end{proof}
\section{Mixed $\ell$-adic integration}
Since we want to prove unconditional arithmetic theorems, we construct in this section the mixed $\ell$-adic variant of the integration of Voevodsky motives. In order to do so, we describe the essential changes that should be made to the construction of the integration of Voevodsky motives. In particular, we describe here the replacements of completed and convergent motives and indicate where in the previous construction we obtain better results upon this specialization. Throughout this section, $k=\mathbb{F}_q$.
\subsection{Completed and convergent mixed $\ell$-adic complexes, $\cal{M}^{\ell}(k)$, and injectivity}
In order to define mixed $\ell$-adic integration, let us take for each separated finite type $k$-scheme $X$, the full subcategory 
\[\D_m(X;\overline{\mathbb{Q}}_{\ell}):=\{\cal{F}^{\bullet}\in\D^b_c(X;\overline{\mathbb{Q}}_{\ell})|\cal{F}^{\bullet}\textup{ is $\iota$-mixed}\},\]
where $\iota:\overline{\mathbb{Q}}_{\ell}\simeq\mathbb{C}$ is a (non-canonical) isomorphism. When $X$ is the spectrum of $k$ itself, we denote this category of $\iota$-mixed complexes simply by $\D_m$. $\D_{m,\geq n}$ is the smallest stable $\infty$-subcategory of $\D_m$ containing $\iota$-mixed complexes $\cal{F}^{\bullet}$ such that for each $i\in\mathbb{Z}$, $\cal{H}^i(\cal{F}^{\bullet})$ has weight $\leq -n$. 
Note that $\D_{m,\geq n}$ is closed under tensoring with $\overline{\mathbb{Q}}_{\ell}(1)$, and that $\D_{m,\geq n}$ is a $\otimes$-ideal of $\D_{m,\geq 0}$. Consider now completed mixed $\ell$-adic complexes defined as follows.
\begin{definition}[Completed mixed $\ell$-adic complexes]Define \textup{completed mixed $\ell$-adic complexes} (or \textup{completed mixed complexes} for brevity) as the limit
\[\D^{\wedge}_m:=\lim_n\D_{m,\geq 0}/D_{m,\geq n}\]
of symmetric monoidal stable $\infty$-categories, where the transition functors 
\[\D_{m,\geq 0}/\D_{m,\geq n+1}\rightarrow \D_{m,\geq 0}/\D_{m,\geq n}\]
are defined in the natural way.
\end{definition}
This is the analogue of completed effective Voevodsky motives. As in the case of motivic integration, we have functors $L^{\ell}_n:\D_{m,\geq 0}\rightarrow \D_{m,\geq 0}/\D_{m,\geq n}$ for each $n$. By the universal property of limits, we obtain a functor
\[L^{\ell}_{\infty}:\D_{m,\geq 0}\rightarrow \D^{\wedge}_m.\]
The added advantage of working with $\ell$-adic sheaves instead of working with Voevodsky motives is that the functors $L^{\ell}_n$ have fully faithful right adjoints
\[i^{\ell}_n:\D_{m,\geq 0}/\D_{m,\geq n}\hookrightarrow \D_{m,\geq 0}.\]
In particular, the analogue of slice motives is precisely just $\D_{m,\geq 0}$ itself and nothing more. We also have an analogue of convergent motives.
\begin{definition}[Convergent mixed $\ell$-adic complexes]
Let $\D_{\conv,\geq 0}$ be the stable $\infty$-category of $\ell$-adic complexes $\cal{F}^{\bullet}\in\D(k;\overline{\mathbb{Q}}_{\ell})$ such that each $\cal{H}^i(\cal{F}^{\bullet})$ is an $\iota$-mixed $\ell$-adic sheaf with the additional property that for each integer $n\geq 0$, $i_n^{\ell}L_n^{\ell}\cal{F}^{\bullet}\in\D_{m,\geq 0}$.
\end{definition}
We can then make the following definition.
\begin{definition}Define for each prime $\ell\neq\textup{char } k$ the group $\cal{M}^{\ell}(k)$ to be $F(\D^{\wedge}_m)$ modulo the subgroup $T^{\ell}_{\infty}(k)$ generated by relations of the form $[A]-[B]+[C]$ whenever
\[A\rightarrow B\rightarrow C\rightarrow A[1]\]
is a distinguished triangle of constructable complexes in $\D_{m,\geq 0}$ as well as relations of the form $[W]$ whenever $W$ is a convergent complex such that $[i^{\ell}_nL^{\ell}_nW]=0$ in $K_0(D_{m,\geq 0})$ for each $n$.
\end{definition}
This is the analogue of our definition of $\cal{M}(k;R)$ in the case of integration of Voevodsky motives via the slice filtration. The analogue of slice motives is just $\D_{m,\geq 0}$ itself as all $i^{\ell}_nL^{\ell}_n$ of $\iota$-mixed $\ell$-adic complexes with cohomology sheaves of weights at most $0$ are again $\iota$-mixed $\ell$-adic complexes with cohomology sheaves of weight at most $0$. Consequently, we have a natural group homomorphism
\[K_0(D_{m,\geq 0})\rightarrow \cal{M}^{\ell}(k)\]
that can be shown to be \textit{injective} in a way similar to the way the injectivity result for slice motives (Proposition~\ref{injectivity}) was proved. Indeed, we have the following.
\begin{proposition}\label{ellinjectivity}
The natural group homomorphism
\[K_0(D_{m,\geq 0})\rightarrow \cal{M}^{\ell}(k)\]
is injective.
\end{proposition}
\begin{proof}A result that went into the analogous Proposition~\ref{injectivity} for Voevodsky motives was that for geometric motives $X$, $X\rightarrow i_nL_n X$ is an equivalence for $n\gg 1$. The analogous result for bounded constructable $\ell$-adic complexes and the functors $i_n^{\ell}$ and $L_n^{\ell}$ is clear. The other delicate thing about
Proposition~\ref{injectivity} was that we showed that the map was injective on $\imK$ which was the image of $K_0(\DM_{\gm}^{\eff}(k;R))$ in the Grothendieck group of slice motives. In the case of mixed $\ell$-adic complexes, the slices $i_n^{\ell}L_n^{\ell}$ do not take us out of mixed $\ell$-adic complexes. Consequently, this kind of subtlety does not occur here. Therefore, we may proceed as in the proof of Proposition~\ref{injectivity} to prove this proposition.  
\end{proof}
\begin{remark}
\textup{Note that in contrast to the case of the integration of Voevodsky motives for which we had the injectivity conjecture, we have the above Proposition~\ref{ellinjectivity}. After defining $\ell$-adic integration and proving its transformation rule for proper birational morphisms of smooth schemes, we will be able to use Proposition~\ref{ellinjectivity} to unconditionally deduce equality of classes in $K_0(D_{m,\geq 0})$.}
\end{remark}
\subsection{Mixed $\ell$-adic measure and measurable subsets}
Note that mixed $\ell$-adic sheaves also have a six functor formalism with all the good properties such as localization sequences and duality. Consequently, we may define the analogue $\mu_X^{\ell}$ of the measure $\mu_X$. Throughout this section, our six functors will be in in the derived category of $\ell$-adic sheaves or its full stable $\infty$-subcategories of bounded constructible or mixed $\ell$-adic sheaves.
\begin{lemma}\label{ladicstabilizationlemma}Suppose $X$ is a smooth $k$-variety of dimension $d$, and let $A$ be a subscheme of the Jet scheme $X_{\infty}$ that is stable at least at the $N$th level. Then for every $m\geq N$, we have
\[\pi^{A_{m+1}}_!\mathbf{1}_{A_{m+1}}((m+1)d)[2(m+1)d]\simeq \pi^{A_m}_!\mathbf{1}_{A_m}(md)[2md]\]
in $\D_{m,\geq 0}(k;\overline{\mathbb{Q}}_{\ell})$.
\end{lemma} 
\begin{proof}The proof is exactly as that of Lemma~\ref{stabilizationlemma} for Voevodsky motives. Constructability follows because the proper pushforward functors preserve constructability. The fact that for each $m\geq N$ and each $i\in\mathbb{Z}$, $\cal{H}^i(\pi^{A_m}_!\mathbf{1}_{A_m}(md))$ has weight $\leq 0$ follows from the Weil conjectures: $\cal{H}^i(\pi^{A_m}_!\mathbf{1}_{A_m})$ has weight at most $\min\{i,2md\}\leq 2md$.  
\end{proof}
Using Lemma~\ref{ladicstabilizationlemma}, we define our Mixed $\ell$-adic measure on stable subschemes as follows.
\begin{definition}[$\ell$-adically measurable functions]For $A$ a stable subscheme of $X_{\infty}$, we define its \textup{mixed $\ell$-adic volume} by
\[\mu_X^{\ell}(A):=\pi^{A_m}_!\mathbf{1}_{A_m}((m+1)d)[2(m+1)d]\in\D_{m,\geq 0}.\]
\end{definition}
In the sequel, it will be important to have the following notion, the $\ell$-adic analogue of the virtual dimension of motives.
\begin{definition}[$\ell$-adic virtual dimension]\label{ladicvirtdimdef} The \textup{$\ell$-adic virtual dimension} of a complex $\cal{F}^{\bullet}$ of mixed $\ell$-adic sheaves, not necessarily bounded, is defined as
\[\vdim^{\ell}\cal{F}^{\bullet}:=-\max\left\{\frac{1}{2}w(\cal{H}^i(\cal{F}^{\bullet}))|i\in\mathbb{Z}\right\},\]
where $w(-)$ denotes the weight function for mixed $\ell$-adic sheaves.
\end{definition}
As in Lemma~\ref{virtdim}, we can show the following lemma.
\begin{lemma}\label{ladicvirtdim}
Suppose $X$ is a $k$-variety of dimension $d$. Then $\vdim^{\ell}\pi^X_!\mathbf{1}_X=-d$.
\end{lemma}
\begin{proof}
By definition,
\[\vdim^{\ell}\pi^X_!\mathbf{1}_X=-\max\left\{\frac{1}{2}w(H^i(\pi^X_!\mathbf{1}_X))|i\in\mathbb{Z}\right\}.\]
Note that $H^i(\pi^X_!\mathbf{1}_X)=H^i_{c,\et}(X_{\overline{k}};\overline{\mathbb{Q}}_{\ell})$ as a $G_k$-module. By the Weil conjectures, the latter has weight at most $i$ and so $H^{2d}_{c,\et}(X_{\overline{k}};\overline{\mathbb{Q}}_{\ell})$ has weight at most $2d$. For $i>2d$, the cohomologies vanish. If $X$ is smooth, then Poincar\'e duality implies that it has weight exactly $2d$, from which the conclusion for smooth $k$-varieties $X$ follows. In order to deal with non-smooth $k$-varieties $X$, we do induction on the dimension of $X$ using the localization sequence and the decomposition of $X$ into $X\setminus\textup{Sing }X$ and $\textup{Sing }X$. This induction is done analogously to the way the induction in Lemma~\ref{virtdim} was done, and so we do not repeat the argument here.
\end{proof}
\begin{corollary}
Suppose $A\subseteq\bigcup_{i=1}^NA_i\subseteq X_{\infty}$, where $A$ and the $A_i$ are stable subschemes of $X_{\infty}$. Then 
\[\vdim^{\ell}\mu^{\ell}_X(A)\geq\min_{1\leq i\leq N}\vdim^{\ell}\mu^{\ell}_X(A_i).\]
\end{corollary}
\begin{proof}
This is the analogue of Corollary~\ref{virtdimbound}, and its proof is exactly the same once we have Lemma~\ref{ladicvirtdim}, the analogue of Lemma~\ref{virtdim}.
\end{proof}
We also have the following analogue of the motivic compactness lemma~\ref{compactness}.
\begin{lemma}[$\ell$-adic compactness lemma]\label{ladiccompactness}
Suppose $D$ and $D_n,\ n\in\mathbb{N}$, are stable subschemes of $X_{\infty}$ such that $D\subseteq\bigcup_{n\in\mathbb{N}}D_n$ and $\vdim^{\ell}\mu^{\ell}_X(D_n)\rightarrow\infty$ as $n\rightarrow\infty$. Then $D$ is contained in the union of finitely many of the $D_n$.
\end{lemma}
\begin{proof}
The proof is exactly as in that of Lemma~\ref{compactness} for Voevodsky motives.
\end{proof}
As in the integration of Voevodsky motives, the subsets that will show up in mixed $\ell$-adic integration will not necessarily be stable subschemes. However, they will be approximable by stable subschemes in the following sense.
\begin{definition}\label{ladicmeasurable}
A subset $C\subseteq X_{\infty}$ is said to be \textup{$\ell$-adically good} if there is a monotonic sequence (the inclusions are locally closed embeddings of $k$-schemes) of stable subschemes
\[C_0\supseteq\hdots\supseteq C_n\supseteq C_{n+1}\supseteq\hdots\ (\text{or }C_0\subseteq\hdots\subseteq C_n\subseteq C_{n+1}\subseteq\hdots)\]
of $X_{\infty}$ containing (resp. contained in) $C$, and stable $C_{n,i}$, $i,n\in\mathbb{N}$ such that for every $n$
\[C_n\setminus C\subseteq\bigcup_{i\in\mathbb{N}}C_{n,i}\ \left(\text{resp. }C\setminus C_n\subseteq\bigcup_{i\in\mathbb{N}}C_{n,i}\right)\]
such that for every $n,i$, 
\[n\leq\vdim^{\ell}\mu^{\ell}_X(C_{n,i})\]
and
\[\vdim^{\ell}\mu^{\ell}_X(C_{n,i})\xrightarrow{i\rightarrow\infty}\infty.\]
We then define the \textup{mixed $\ell$-adic volume} of $C$ as the object
\[\mu^{\ell}_X(C):=(\mu^{\ell}_X(C_n))_n\in\D^{\wedge}_m.\]
We call a pair $(C,\cal{S})$ consisting of a subset $C\subseteq X_{\infty}$ and a \textup{finite} decomposition $\cal{S}$ ($C=\sqcup_{i\in\cal{S}}C_i$, $C_i\subseteq X_{\infty}$) \textup{$\ell$-adically measurable} if each $C_i$ is an $\ell$-adically good subset. We then let
\[\mu^{\ell}_X(C,\cal{S}):=\bigoplus_{i\in\cal{S}}\mu^{\ell}_X(C_i).\]
We view an $\ell$-adically good subset $C$ without a prescribed decomposition as an $\ell$-adically measurable subset with the trivial decomposition $\cal{S}=\{C\}$.
\end{definition}
As in Proposition~\ref{welldefinedmeasure} for the motivic measure (of Voevodsky motives), we can show that the $\ell$-adic measure defined in Definition~\ref{ladicmeasurable} is well-defined. We do not give the proof here as all the steps are the same.
\subsection{$\ell$-adically measurable functions and mixed $\ell$-adic integration}
We now work toward constructing a theory of integration for mixed $\ell$-adic sheaves that takes a smooth $k$-scheme $X$ along with an effective divisor $D\subset X$ to an element of $\cal{M}^{\ell}(k)$. For this, we make the following definition which is the analogue of Definition~\ref{measurablefunction}.
\begin{definition}[$\ell$-adically measurable functions]\label{ladicmeasurablefunction}
A function $F:X_{\infty}\rightarrow\mathbb{N}_{\geq 0}\cup\left\{\infty\right\}$ is \textup{$\ell$-adically measurable} if for each $s\in\mathbb{N}_{\geq 0}$, $F^{-1}(s)$ is $\ell$-adically measurable, and $F^{-1}(\infty)$ has $\ell$-adic measure $0$.
\end{definition}
\begin{definition}[Mixed $\ell$-adic integration]
For $\ell$-adically measurable functions $F:X_{\infty}\rightarrow\mathbb{N}_{\geq 0}\cup\{\infty\}$, we define its mixed $\ell$-adic integration 
\[\int_{X_{\infty}}\overline{\mathbb{Q}}_{\ell}(F)[2F]d\mu^{\ell}_X:=\left[\oplus_{k=0}^{\infty}\mu^{\ell}_X(F^{-1}(k))(k)[2k]\right]\in\cal{M}^{\ell}(k).\]
\end{definition}
As before, we can prove the following transformation rule for proper birational morphisms of smooth $k$-schemes.
\begin{theorem}[Mixed $\ell$-adic transformation rule]\label{ladicchangeofvar}
Suppose $f:X\rightarrow Y$ is a proper birational morphism of smooth $k$-varieties with $K_{X/Y}$ its relative canonical divisor, and let $D\subset Y$ be an effective divisor. Then
\[\int_{Y_{\infty}}\overline{\mathbb{Q}}_{\ell}(\ord_D)[2\ord_D]d\mu^{\ell}_Y=\int_{X_{\infty}}\overline{\mathbb{Q}}_{\ell}(\ord_{f^{-1}D+K_{X/Y}})[2\ord_{f^{-1}D+K_{X/Y}}]d\mu^{\ell}_X.\]
\end{theorem}
\begin{proof}
The proof is exactly as that of Theorem~\ref{changeofvar} once we have the above definitions and results.
\end{proof}
\section{Integration via the motivic $t$-structure}\label{rationalintegration}
In this section, we construct a theory of integration of rational Voevodsky motives assuming the existence of a motivic $t$-structure. In contrast to the theory of integration constructed in the previous section, this theory is conditional on the existence of a motivic $t$-structure and it works rationally but not integrally. However, it has the feature that the analogue of the injectivity conjecture $I(k;\mathbb{Q})$ needs no additional assumptions. Consequently, it allows us to deduce concrete consequences of the existence of a motivic $t$-structure.\\
\\
We first mention what we mean by a motivic $t$-structure, and then use such a hypothetical structure on rational motives to construct analogues of slices using the positive part of the motivic $t$-structure. Using these structures, we define motivically completed rational motives using which we construct an abelian group which will eventually act as the value group of the integration constructed in this section. After proving analogues of results in the previous section, we define the integration and prove its corresponding transformation rule with respect to proper birational morphisms of smooth schemes. For this section (to fix the notation, at least) as well as for our future applications, we recall here some useful realization functors.
\begin{enumerate}
	\item (Betti and $\ell$-adic realization) Let $\Vect_{\mathbb{Q}_{\ell}}$ and $\Vect_{\mathbb{Q}}$ be the category of finite dimensional $\mathbb{Q}_{\ell}$- and $\mathbb{Q}$-vector spaces, respectively. For each prime $\ell$ coprime to the exponential characteristic of the field $k$, there is an $\ell$-adic realization functor $r_{\mathbb{Q}_{\ell}}:\DM_{\textup{gm}}(k;\mathbb{Q})\rightarrow D^b(\Vect_{\mathbb{Q}_{\ell}})$ induced by sending $M(X)$, a $k$-variety $X$, to $H^*_{\textup{\'et}}(X_{\overline{k}};\mathbb{Q}_{\ell})^*$. If $k$ is a field of characteristic zero with an embedding $\sigma:k\hookrightarrow\mathbb{C}$, we have the Betti realization functor $r_{\sigma}:\DM_{\textup{gm}}(k;\mathbb{Q})\rightarrow D^b(\Vect_{\mathbb{Q}})$ induced by sending  $M(X)$, $X$ any $k$-variety, to $H_*(X^{an};\mathbb{Q})$. There is also an integral contravariant version of the Betti realization
	\[B_{\sigma}:\DM_{\textup{gm}}^{\textup{eff}}(k;\mathbb{Z})^{op}\rightarrow D^b(\mathbb{Z})\]
	given by sending $M(X)$ to $H^*(X^{an};\mathbb{Z})$. See, for example, Lecomte's work \cite{Lecomte}.
	\item(Integral $\ell$-adic realization) Voevodsky's triangulated or stable $\infty$-category of geometric effective motives $\DM_{\textup{gm}}^{\textup{eff}}(k;\mathbb{Z})$ is based on the Nisnevich topology, a topology coarser than the \'etale topology. For each $n$, reducing modulo $\ell^n$ and sheafifying with respect to the \'etale topology gives us functors
\[\DM_{\textup{gm}}^{\textup{eff}}(k;\mathbb{Z})\rightarrow \DM_{\textup{gm,\'et}}^{\textup{eff}}(k;\mathbb{Z}/\ell^n)\hookrightarrow\DM_{\textup{\'et}}^{\textup{eff}}(k;\mathbb{Z}/\ell^n).\]
By rigidity (see Cisinski-D\'eglise's \cite{CDetale} for the case that $k$ does not have finite cohomological dimension), the latter category is equivalent to the derived category of sheaves of $\mathbb{Z}/\ell^n$-modules on the small \'etale site of $k$, that is $\textup{D}(k_{\textup{\'et}},\mathbb{Z}/\ell^n)$. These functors assemble into the $\ell$-adic realization functor
\[\DM_{\textup{gm}}^{\textup{eff}}(k;\mathbb{Z})\rightarrow \widehat{\textup{D}}(k_{\textup{\'et}};\mathbb{Z}_{\ell}),\]
where $\widehat{\textup{D}}(k_{\textup{\'et}};\mathbb{Z}_{\ell})$ is the derived category of $\ell$-adic sheaves. This latter category is equivalent to the derived category of continuous $\ell$-adic Galois representations $\textup{D}(\textup{Rep}_{\textup{cnt}}(G_k;\mathbb{Z}_{\ell}))$. We know that the above functor factors through the bounded derived category of constructable $\ell$-adic sheaves. We can also rationalize everything (by inverting $\ell$) if we want to work over $\mathbb{Q}_{\ell}$. 
\item(Hodge realization) Recall that an integral mixed Hodge structure consists of a finitely generated abelian group $V_{\mathbb{Z}}$ together with a finite increasing filtration $W_i$ on $V_{\mathbb{Q}}:=V_{\mathbb{Z}}\otimes\mathbb{Q}$, called a weight filtration, and a finite decreasing filtration $F^p$ on $V_{\mathbb{C}}:=V_{\mathbb{Z}}\otimes\mathbb{C}$, called a Hodge filtration, with some compatibility conditions. This is an abstraction of the algebraic structure one obtains from Hodge theory on integral cohomology groups of every complex algebraic variety \cite{GilSoul}. There is an abelian category of integral mixed Hodge structures $\textup{MHS}_{\mathbb{Z}}$. Consider the category $\DM_{\textup{gm}}(\mathbb{C};\mathbb{Z})$ of geometric Voevodsky motives over $\Spec\mathbb{C}$. Lecomte and Wach \cite{LecomteWach} have constructed an integral Hodge realization functor
\[\DM_{\textup{gm}}(\mathbb{C};\mathbb{Z})^{op}\rightarrow\textup{D}^b(\textup{MHS}_{\mathbb{Z}})\]
to the bounded derived category of integral mixed Hodge structures. It sends $M(X)$ to $H^*(X^{an};\mathbb{Z})$ with its canonical integral mixed Hodge structure. We could also rationalize to pass to (polarizable) rational mixed Hodge structures.
\end{enumerate}
\subsection{Completed and convergent rational motives, $\cal{M}^{\mot}(k)$, and injectivity}
Let $\mu$ be a $t$-structure on $\DM_{\gm}(k;\mathbb{Q})$ with $\DM_{\gm}(k;\mathbb{Q})^{\geq 0}$ and $\DM_{\gm}(k;\mathbb{Q})^{\leq 0}$ denoting its positive and negative parts, respectively. We denote by $\DM_{\gm}(k;\mathbb{Q})^{\heartsuit}:=\DM_{\gm}(k;\mathbb{Q})^{\geq 0}\cap\DM_{\gm}(k;\mathbb{Q})^{\leq 0}$ the heart of $\mu$. Furthermore, let $^{\mu}H^{\bullet}:\DM_{\gm}(k;\mathbb{Q})\rightarrow\DM_{\gm}(k;\mathbb{Q})^{\heartsuit}$ be the cohomology functors with respect to the $t$-structure $\mu$. We say that the $t$-structure $\mu$ is \textit{non-degenerate} if the functors $\{^{\mu}H^a\}_a$ form a conservative family of functors.
\begin{definition}$\mu$ is said to be a \textup{motivic} $t$-structure if it is non-degenerate and compatible with $\otimes$ and $r$, that is, $\otimes$ and $r$ are $t$-exact. Here $r$ is either $r_{\mathbb{Q}_{\ell}}$ or $r_{\sigma}$ as described before.\end{definition}
A notoriously difficult conjecture in the theory of motives posits the following.
\begin{conjecture}[Motivic $t$-structure conjecture]\label{tstructureconjecture} There is a motivic $t$-structure on the stable $\infty$-category $\DM_{\gm}(k;\mathbb{Q})$ whose heart $\DM_{\gm}(k;\mathbb{Q})^{\heartsuit}$ has semisimple part the abelian category $\Num(k;\mathbb{Q})$ of rational numerical motives over $k$, and such that every motive has a filtration by rational numerical motives. Additionally, for each smooth projective $k$-variety $X$, each $^{\mu}H^qM(X)_{\mathbb{Q}}$ is a semisimple object of the heart. Furthermore, there is a unique increasing (weight) filtration $W_{\bullet}M$ for each $M\in\DM_{\gm}(k;\mathbb{Q})$ such that for each irreducible $P$, $W_mP=P$ and $W_{m-1}P=0$ if $P$ occurs in some $^{\mu}H^iM(X)_{\mathbb{Q}}(a)$ with $m=i-2a$ and $X$ any smooth projective $k$-variety.
\end{conjecture}
\begin{remark}
\textup{In characteristic $0$, the last two conditions follow from the existence of a motivic $t$-structure (Propositions 1.5 and 1.7 of Beilinson's \cite{Beilinson}).}
\end{remark}
Assuming the existence of a motivic $t$-structure $\mu$ on $\DM_{\gm}(k;\mathbb{Q})$, define $\DM_{\gm,\geq n}$ to be the smallest stable $\infty$-subcategory of $\DM_{\gm}(k;\mathbb{Q})$ containing objects $M$ such that for each $i\in\mathbb{Z}$, $^{\mu}H^iM$ has weight at most $-n$. Note that $\DM_{\gm,\geq n}$ is a $\otimes$-ideal of $\DM_{\gm,\geq 0}$ closed under tensoring with $\mathbb{Q}(1)$.
\begin{definition}[Motivically completed rational motives] Define the stable $\infty$-category of \textup{motivically completed rational motives} as the limit
\[\DM_{\gm}^{\mot,\wedge}:=\lim_n\DM_{\gm,\geq 0}/\DM_{\gm,\geq n}\]
of symmetric monoidal stable $\infty$-categories, where the transition functors
\[\DM_{\gm,\geq 0}/\DM_{\gm,\geq n+1}\rightarrow\DM_{\gm,\geq 0}/\DM_{\gm,\geq n}\]
are defined as the natural localization functors.
\end{definition}
As before, we have localization functors
\[L_n^{\mot}:\DM_{\gm,\geq 0}\rightarrow\DM_{\gm,\geq 0}/\DM_{\gm,\geq n}.\]
By the universal property of limits, we have the natural functor
\[L_{\infty}^{\mot}:\DM_{\gm,\geq 0}\rightarrow\DM_{\gm}^{\mot,\wedge}.\]
By assumption, the motivic $t$-structures behaves well with respect to the weight structure. Consequently, there is a fully faithful right adjoint
\[i_n^{\mot}:\DM_{\gm,\geq 0}/\DM_{\gm,\geq n}\hookrightarrow\DM_{\gm,\geq 0}\]
to the localization functor $L_n^{\mot}$. In particular, in this case we do not have more objects than those of $\DM_{\gm,\geq 0}$ itself when we consider the smallest stable $\infty$-category generated by the objects $i_n^{\mot}L_n^{\mot}X$, $X\in\DM_{\gm,\geq 0}$. This is the main reason that using the motivic $t$-structure in the construction of our integration of rational Voevodsky motives allows us to have the analogue of the injectivity conjecture $I(k;\mathbb{Q})$. In the following definition, we are considering the natural extension of the categories $\DM_{\gm,\geq n}$ to the larger presentable stable $\infty$-categories $\DM_{\geq n}$. We are also using the natural extensions of the functors $L_n^{\mot}$. We abuse notation, but this abuse of notation should not cause confusion.
\begin{definition}[Convergent rational motives] A \textup{convergent rational motive} is a motive $X\in\DM_{\geq 0}$ such that for each $N\geq 0$, $i_N^{\mot}L_N^{\mot}X$ is geometric (in $\DM_{\gm}(k;\mathbb{Q})$). The stable $\infty$-category of convergent rational motives is denoted by $\DM_{\gm,\conv}$.
\end{definition}

\begin{definition}
Define the abelian group $\cal{M}^{\mot}(k)$ to be the free abelian group $F(\DM_{\gm}^{\mot,\wedge})$ modulo the subgroup $T_{\infty}^{\mot}(k)$ generated by relations of the form
\[[A]-[B]+[C]\]
where $A\rightarrow B\rightarrow C$
is a distinguished triangle in $\DM_{\gm,\geq 0}$, in addition to relations of the form $[W]$ whenever $W$ is a convergent rational motive such that $[i_n^{\mot}L_n^{\mot}W]=0$ in $K_0(\DM_{\gm,\geq 0})$ for any $n$.
\end{definition}
Note that there is a natural group homomorphism
\[K_0(\DM_{\gm,\geq 0})\rightarrow\cal{M}^{\mot}(k).\]
\begin{proposition}\label{motinjectivity}The group homomorphism
\[K_0(\DM_{\gm,\geq 0})\rightarrow\cal{M}^{\mot}(k)\]
is injective.
\end{proposition}
\begin{proof}
The proof is exactly as in the those of Propositions~\ref{injectivity} and~\ref{ellinjectivity}. The difference is that now $i_n^{\mot}L_n^{\mot}X\in\DM_{\gm,\geq 0}$ for every $X\in\DM_{\gm,\geq 0}$; in the theory of integration based on the slice filtration, the functors $L_n$ do not in general have fully faithful right adjoints $i_n$ on the level of \textit{geometric} motives. The proof is as in the rest of that of Proposition~\ref{injectivity}, as was that of Proposition~\ref{ellinjectivity}.
\end{proof}
\subsection{Rational motivic measure and measurable subsets}
Using the availability of the six functor formalism for rational motives, we may define the analogue $\mu_X^{\mot}$ of the measure $\mu_X$. Throughout this section, our six functors will be those of rational Voevodsky motives.
\begin{lemma}\label{motstabilizationlemma}Suppose $X$ is a smooth $k$-variety of dimension $d$, and let $A$ be a subscheme of the Jet scheme $X_{\infty}$ that is stable at least at the $N$th level. Then for every $m\geq N$, we have
\[\pi^{A_{m+1}}_!\mathbf{1}_{A_{m+1}}((m+1)d)[2(m+1)d]\simeq \pi^{A_m}_!\mathbf{1}_{A_m}(md)[2md]\]
in $\DM_{\gm,\geq 0}$.
\end{lemma} 
\begin{proof}The proof is exactly as that of Lemma~\ref{stabilizationlemma} for Voevodsky motives. Constructability follows because the proper pushforward functors preserve constructability. The fact that for each $m\geq N$ and each $i\in\mathbb{Z}$, $^{\mu}H^i(\pi^{A_m}_!\mathbf{1}_{A_m}(md))$ has weight $\leq 0$ follows from the fact that $^{\mu}H^i(\pi^{A_m}_!\mathbf{1}_{A_m})$ has weight $\leq\min\{i,2md\}\leq 2md$. This is a consequence of the conservativity of the realization functors once we have a motivic $t$-structure. This conservativity follows from the fact that the existence of a motivic $t$-structure implies that the heart, that is the $\mathbb{Q}$-linear abelian $\otimes$-category of mixed motives, is a Tannakian category. See Section 1.3 of Beilinson's \cite{Beilinson}.
\end{proof}
Using Lemma~\ref{motstabilizationlemma}, we define our motivic measure on stable subschemes as follows.
\begin{definition}[Motivic measurable functions]For $A$ a stable subscheme of $X_{\infty}$, we define its \textup{motivic volume} by
\[\mu_X^{\mot}(A):=\pi^{A_m}_!\mathbf{1}_{A_m}((m+1)d)[2(m+1)d]\in\DM_{\gm,\geq 0}.\]
\end{definition}
In the sequel, it will be important to have the following notion, the analogue of the (slice) virtual dimension of motives.
\begin{definition}[Motivic virtual dimension]\label{motvirtdimdef} The \textup{motivic virtual dimension} of a rational motive $M\in\DM(k;\mathbb{Q})$ is defined as
\[\vdim^{\mot}M:=-\max\left\{\frac{1}{2}w(^{\mu}H^iM)|i\in\mathbb{Z}\right\},\]
where $w(-)$ denotes the weight function for motives.
\end{definition}
\begin{lemma}\label{motvirtdim}
Suppose $X$ is a $k$-variety of dimension $d$. Then $\vdim^{\mot}\pi^X_!\mathbf{1}_X=-d$.
\end{lemma}
\begin{proof}
The proof is as that of Lemma~\ref{ladicvirtdim}. For rational motives too we have Poincar\'e-Verdier duality, which was used in the proof of Lemma~\ref{virtdim} to settle the smooth case.
\end{proof}
\begin{corollary}\label{motvirtdimbound}
Suppose $A\subseteq\bigcup_{i=1}^NA_i\subseteq X_{\infty}$, where $A$ and the $A_i$ are stable subschemes of $X_{\infty}$. Then
\[\vdim^{\mot}\mu^{\mot}_X(A_i)\geq\min_{1\leq i\leq N}\vdim^{\mot}\mu^{\mot}_X(A).\]
\end{corollary}
\begin{proof}
This is the analogue of Corollary~\ref{virtdimbound}, and its proof relies on the fact that for a $k$-variety $X$ of dimension $d$, $2d$ is the largest weight of the $^{\mu}H^{i}M(X)_{\mathbb{Q}}$ as $i$ varies.
\end{proof}
We also have the following analogue of the compactness Lemma~\ref{compactness}.
\begin{lemma}[Motivic compactness lemma]\label{motcompactness}
Suppose $D$ and $D_n,\ n\in\mathbb{N}$, are stable subschemes of $X_{\infty}$ such that $D\subseteq\bigcup_{n\in\mathbb{N}}D_n$ and the $\vdim^{\mot}\mu^{\ell}_X(D_n)\rightarrow\infty$ as $n\rightarrow\infty$. Then $D$ is contained in the union of finitely many of the $D_n$.
\end{lemma}
\begin{proof}
The proof is exactly as in that of Lemma~\ref{compactness} for Voevodsky motives and that of Lemma~\ref{ladiccompactness} for mixed $\ell$-adic sheaves.
\end{proof}
As in the integration of Voevodsky motives, the subsets that will show up in the integration of rational motives will not necessarily be stable subschemes. However, they will be approximable by stable subschemes in the following sense.
\begin{definition}\label{motmeasurable}
A subset $C\subseteq X_{\infty}$ is said to be \textup{motivically good} if there is a monotonic sequence (the inclusions are locally closed embeddings of $k$-schemes) of stable subschemes
\[C_0\supseteq\hdots\supseteq C_n\supseteq C_{n+1}\supseteq\hdots\ (\text{or }C_0\subseteq\hdots\subseteq C_n\subseteq C_{n+1}\subseteq\hdots)\]
of $X_{\infty}$ containing (resp. contained in) $C$, and stable $C_{n,i}$, $i,n\in\mathbb{N}$ such that for every $n$
\[C_n\setminus C\subseteq\bigcup_{i\in\mathbb{N}}C_{n,i}\ \left(\text{resp. }C\setminus C_n\subseteq\bigcup_{i\in\mathbb{N}}C_{n,i}\right)\]
such that for every $n,i$, 
\[n\leq\vdim^{\mot}\mu^{\mot}_X(C_{n,i})\]
and
\[\vdim^{\mot}\mu^{\mot}_X(C_{n,i})\xrightarrow{i\rightarrow\infty}\infty.\]
We then define the \textup{motivic volume} of $C$ as the object
\[\mu^{\mot}_X(C):=(\mu^{\mot}_X(C_n))_n\in\DM_{\gm}^{\mot,\wedge}.\]
We call a pair $(C,\cal{S})$ consisting of a subset $C\subseteq X_{\infty}$ and a \textup{finite} decomposition $\cal{S}$ ($C=\sqcup_{i\in\cal{S}}C_i$, $C_i\subseteq X_{\infty}$) \textup{motivically measurable} if each $C_i$ is a motivically good subset. We then let
\[\mu^{\mot}_X(C,\cal{S}):=\bigoplus_{i\in\cal{S}}\mu^{\mot}_X(C_i).\]
We view a motivically good subset $C$ without a prescribed decomposition as a motivically measurable subset with the trivial decomposition $\cal{S}=\{C\}$.
\end{definition}
As in Proposition~\ref{welldefinedmeasure} for the motivic measure (of Voevodsky motives), we can show that the motivic measure defined in Definition~\ref{motmeasurable} is well-defined. We do not give the proof here as the proof is very close to that of Proposition~\ref{welldefinedmeasure}.
\subsection{Motivically measurable functions and motivic integration}
We now work toward constructing a theory of integration for rational motives that takes a smooth $k$-scheme $X$ along with an effective divisor $D\subset X$ to an element of $\cal{M}^{\mot}(k)$. For this, we make the following definition which is the analogue of Definition~\ref{measurablefunction}.
\begin{definition}[Motivically measurable functions]\label{motmeasurablefunction}
A function $F:X_{\infty}\rightarrow\mathbb{N}_{\geq 0}\cup\left\{\infty\right\}$ is \textup{motivically measurable} if for each $s\in\mathbb{N}_{\geq 0}$, $F^{-1}(s)$ is motivically measurable, and $F^{-1}(\infty)$ has motivic measure $0$.
\end{definition}
\begin{definition}[Rational motivic integration]
For motivically measurable functions $F:X_{\infty}\rightarrow\mathbb{N}_{\geq 0}\cup\{\infty\}$, we define its motivic integration (via the motivic $t$-structure) as
\[\int_{X_{\infty}}\mathbb{Q}(F)[2F]d\mu^{\mot}_X:=\left[\oplus_{k=0}^{\infty}\mu^{\mot}_X(F^{-1}(k))(k)[2k]\right]\in\cal{M}^{\mot}(k).\]
\end{definition}
As before, we prove the following transformation rule for proper birational morphisms of smooth $k$-schemes.
\begin{theorem}[Rational motivic transformation rule]\label{motchangeofvar}
Suppose $f:X\rightarrow Y$ is a proper birational morphism of smooth $k$-varieties with $K_{X/Y}$ its relative canonical divisor, and let $D\subset Y$ be an effective divisor. Then
\[\int_{Y_{\infty}}\mathbb{Q}(\ord_D)[2\ord_D]d\mu^{\mot}_Y=\int_{X_{\infty}}\mathbb{Q}(\ord_{f^{-1}D+K_{X/Y}})[2\ord_{f^{-1}D+K_{X/Y}}]d\mu^{\mot}_X.\]
\end{theorem}
\begin{proof}
The proof is exactly as in that of Theorem~\ref{changeofvar}, and we do not repeat the proof here.
\end{proof}

\section{Mixed Hodge integration}\label{Hodge}
Similar to mixed $\ell$-adic integration, integration of Voevodsky motives via the slice filtration as well as via the motivic $t$-structure, we can define mixed Hodge integration, an integration theory for (the bounded derived category of rational) mixed Hodge modules. 
\subsection{Completed and convergent mixed $\mathbb{Q}$-Hodge structures, $\cal{M}^{\Hdg}(k)$, and injectivity}
In the theory of Voevodsky motives, we considered the slice tower. Here, we need a notion of effectivity for polarizable mixed $\mathbb{Q}$-Hodge structures.
\begin{definition}We say that a pure $\mathbb{Q}$-Hodge structure is effective if all its nonzero Hodge numbers are concentrated in the lower half plane.
\end{definition}
\begin{remark}
\textup{Our usage of the word "effective" is idiosyncratic and dual to the usual usage. The reason for this is because we want the Tate-Hodge structure $\mathbb{Q}(1)$ to be defined below to be effective in accordance with the effectivity of $\mathbb{Q}(1)$ in Voevodsky motives.}
\end{remark}
\begin{definition}We say that a polarizable mixed $\mathbb{Q}$-Hodge structure is effective if all its simple subquotients are effective in the sense of the previous definition.
\end{definition}
\begin{definition}
We say that a bounded complex of polarizable mixed $\mathbb{Q}$-Hodge structures $M$ is effective if for each $i$ the polarizable mixed $\mathbb{Q}$-Hodge structure $\cal{H}^i(M)$ is effective.
\end{definition}
\begin{definition}
Denote by $\mathbb{Q}(1)$ the $\mathbb{Q}$-Hodge structure given by $\mathbb{Q}$ in degree $(-1,-1)$, and $0$ elsewhere. This is the dual of the compactly supported complex for $\mathbb{A}^1$. Denote by $\mathbb{Q}(-1)$ its dual with $\mathbb{Q}$ in degree $(1,1)$ and $0$ elsewhere. Define $\mathbb{Q}(n):=\mathbb{Q}(\textup{sgn } n)^{\otimes |n|}$, where $\textup{sgn }n\in\{0,\pm 1\}$ is the sign of $n$ with value $0$ if and only if $n=0$. Given a polarizable $\mathbb{Q}$-Hodge structure $M$, denote by $M(n):=M\otimes\mathbb{Q}(n)$.
\end{definition}
\begin{definition}
Let $MHS_{\mathbb{Q}}$ be the abelian category of polarizable mixed $\mathbb{Q}$-Hodge structures, and by $D^b_{\Hdg,\geq n}$ the full subcategory of $D^b(MHS_{\mathbb{Q}})$ generated by objects of the form $M(n)$, where $M$ varies over effective bounded complexes of polarizable mixed $\mathbb{Q}$-Hodge structures.\end{definition}
\begin{definition}[Completed Hodge structures] Note that $D^b_{\Hdg,\geq n}$ is a tensor ideal of $D^b_{\Hdg,\geq 0}$ closed under tensoring by $\mathbb{Q}(1)$. Define the stable $\infty$-category
\[D^{b,\wedge}_{\Hdg}:=\lim_nD^b_{\Hdg,\geq 0}/D^b_{\Hdg,\geq n}\]
as a limit of symmetric monoidal stable $\infty$-categories. The transition functors are the natural localization functors
\[D^b_{\Hdg,\geq 0}/D^b_{\Hdg,\geq n+1}\rightarrow D^b_{\Hdg,\geq 0}/D^b_{\Hdg,\geq n}.\]
We call this the category of completed Hodge structures.
\end{definition}
As before, we have localization functors
\[L_n^{\Hdg}:D^b_{\Hdg,\geq 0}\rightarrow D^b_{\Hdg,\geq 0}/D^b_{\Hdg,\geq n}.\]
By the universal property of limits, we have the natural functor
\[L_{\infty}^{\Hdg}:D^b_{\Hdg,\geq 0}\rightarrow D^{b,\wedge}_{\Hdg}.\]
There is a fully faithful right adjoint
\[i_n^{\Hdg}:D^b_{\Hdg,\geq 0}/D^b_{\Hdg,\geq n}\hookrightarrow D^b_{\Hdg,\geq 0}\]
to the localization functor $L_n^{\Hdg}$. We also have extensions to the unbounded derived categories. Again, we abuse notation; this should not cause confusion.
\begin{definition}[Convergent $\mathbb{Q}$-Hodge structures]
Let $\D_{\conv,\geq 0}^{\Hdg}$ be the stable $\infty$-category of (\textup{not} necessarily bounded) complexes $\cal{F}^{\bullet}$ of polarizable $\mathbb{Q}$-Hodge structures such that for each integer $n\geq 0$, $i_n^{\Hdg}L_n^{\Hdg}\cal{F}^{\bullet}\in\D_{\Hdg,\geq 0}$.
\end{definition}

\begin{definition}
Define the abelian group $\cal{M}^{\Hdg}(k)$ to be the free abelian group $F(D^{b,\wedge}_{\Hdg})$ modulo the subgroup $T_{\infty}^{\Hdg}(k)$ generated by relations of the form
\[[A]-[B]+[C]\]
where $A\rightarrow B\rightarrow C\rightarrow A[1]$ is a distinguished triangle in $D^b_{\Hdg,\geq 0}$, in addition to relations of the form $[W]$ whenever $W$ is a convergent $\mathbb{Q}$-Hodge structure such that $[i_n^{\Hdg}L_n^{\Hdg}W]=0$ in $K_0(D^b_{\Hdg,\geq 0})$ for any $n$.
\end{definition}
Note that there is a natural group homomorphism
\[K_0(D^b_{\Hdg,\geq 0})\rightarrow\cal{M}^{\Hdg}(k).\]
\begin{proposition}\label{Hdginjectivity}The group homomorphism
\[K_0(D^b_{\Hdg,\geq 0})\rightarrow\cal{M}^{\Hdg}(k)\]
is injective.
\end{proposition}
\begin{proof}
The idea of the proof is exactly as in the proof of Proposition~\ref{injectivity}. The difference is that now the functors $i_n^{\Hdg}L_n^{\Hdg}X\in D^b_{\Hdg,\geq 0}$ for every $X\in D^b_{\Hdg,\geq 0}$; in the theory of integration based on the slice filtration, the functors $L_n$ did not have fully faithful right adjoints $i_n$ on the level of \textup{geometric} motives. 
\end{proof}
\subsection{Hodge measure and Hodge-measurable subsets}
Using the availability of the six functor formalism for (the bounded derived category of) mixed Hodge modules and the fact that the bounded derived category of mixed Hodge modules over $\Spec k$ correspond to the bounded derived category of polarizable mixed $\mathbb{Q}$-Hodge structures \cite{Saito}, we may define the analogue of the Hodge measure $\mu_X^{\Hdg}$. This will be the analogue of the motivic measure $\mu_X$ and the mixed $\ell$-adic measure $\mu_X^{\ell}$. Throughout this section, our six functors will be those of rational mixed Hodge modules.
\begin{lemma}\label{Hdgstabilizationlemma}Suppose $X$ is a smooth $k$-variety of dimension $d$, and let $A$ be a subscheme of the Jet scheme $X_{\infty}$ that is stable at least at the $N$th level. Then for every $m\geq N$, we have
\[\pi^{A_{m+1}}_!\mathbf{1}_{A_{m+1}}((m+1)d)[2(m+1)d]\simeq \pi^{A_m}_!\mathbf{1}_{A_m}(md)[2md]\]
in $D^b_{\Hdg,\geq 0}$.
\end{lemma} 
\begin{proof}The proof is exactly as that of Lemma~\ref{stabilizationlemma} for Voevodsky motives. Constructability follows because the proper pushforward functors preserve constructability. The fact that for each $m\geq N$ and each $i\in\mathbb{Z}$, $H^i(\pi^{A_m}_!\mathbf{1}_{A_m}(md))$ has weight $\leq 0$ follows from the fact that $H^i(\pi^{A_m}_!\mathbf{1}_{A_m})$ has weight $\leq\min\{i,2md\}\leq 2md$.  
\end{proof}
Using Lemma~\ref{Hdgstabilizationlemma}, we define our Hodge measure on stable subschemes as follows.
\begin{definition}[Hodge-measurable functions]For $A$ a stable subscheme of $X_{\infty}$, we define its \textup{Hodge volume} by
\[\mu_X^{\Hdg}(A):=\pi^{A_m}_!\mathbf{1}_{A_m}((m+1)d)[2(m+1)d]\in D^b_{\Hdg,\geq 0}.\]
\end{definition}
The analogue of the virtual dimension of motives in this case is the following.
\begin{definition}[Hodge virtual dimension]\label{Hdgvirtdimdef} The \textup{Hodge virtual dimension} of a rational polarizable $\mathbb{Q}$-Hodge structure $M$ is defined as
\[\vdim^{\Hdg}M:=-\max\left\{\frac{1}{2}w(H^iM)|i\in\mathbb{Z}\right\},\]
where $w(-)$ denotes the weight for rational mixed Hodge structures.
\end{definition}
\begin{lemma}\label{Hdgvirtdim}
Suppose $X$ is a $k$-variety of dimension $d$. Then $\vdim^{\Hdg}\pi^X_!\mathbf{1}_X=-d$.
\end{lemma}
\begin{proof}
The proof is similar to that of Lemma~\ref{ladicvirtdim}.
\end{proof}
\begin{corollary}\label{Hdgvirtdimbound}
Suppose $A\subseteq\bigcup_{i=1}^NA_i\subseteq X_{\infty}$, where $A$ and the $A_i$ are stable subschemes of $X_{\infty}$. Then
\[\vdim^{\Hdg}\mu^{\Hdg}_X(A_i)\geq\min_{1\leq i\leq N}\vdim^{\Hdg}\mu^{\Hdg}_X(A).\]
\end{corollary}
\begin{proof}
This is the analogue of Corollary~\ref{virtdimbound}, and its proof relies on the fact that for a $k$-variety $X$ of dimension $d$, $2d$ is the largest weight of the $H^{i}(X;\mathbb{Q})$ as $i$ varies.
\end{proof}
We also have the following analogue of the compactness Lemma~\ref{compactness}.
\begin{lemma}[Hodge compactness lemma]\label{Hdgcompactness}
Suppose $D$ and $D_n,\ n\in\mathbb{N}$, are stable subschemes of $X_{\infty}$ such that $D\subseteq\bigcup_{n\in\mathbb{N}}D_n$ and the $\vdim^{\Hdg}\mu^{\Hdg}_X(D_n)\rightarrow\infty$ as $n\rightarrow\infty$. Then $D$ is contained in the union of finitely many of the $D_n$.
\end{lemma}
\begin{proof}
The proof is exactly as in that of Lemma~\ref{compactness} for Voevodsky motives and that of Lemma~\ref{ladiccompactness} for mixed $\ell$-adic sheaves.
\end{proof}
As in the other theories of integration, the subsets that will appear will be approximable by stable subschemes in the following sense but they will not necessarily be stable.
\begin{definition}\label{Hdgmeasurable}
A subset $C\subseteq X_{\infty}$ is said to be \textup{Hodge-good} if there is a monotonic sequence (the inclusions are locally closed embeddings of $k$-schemes) of stable subschemes
\[C_0\supseteq\hdots\supseteq C_n\supseteq C_{n+1}\supseteq\hdots\ (\text{or }C_0\subseteq\hdots\subseteq C_n\subseteq C_{n+1}\subseteq\hdots)\]
of $X_{\infty}$ containing (resp. contained in) $C$, and stable $C_{n,i}$, $i,n\in\mathbb{N}$ such that for every $n$
\[C_n\setminus C\subseteq\bigcup_{i\in\mathbb{N}}C_{n,i}\ \left(\text{resp. }C\setminus C_n\subseteq\bigcup_{i\in\mathbb{N}}C_{n,i}\right)\]
such that for every $n,i$, 
\[n\leq\vdim^{\Hdg}\mu^{\Hdg}_X(C_{n,i})\]
and
\[\vdim^{\Hdg}\mu^{\Hdg}_X(C_{n,i})\xrightarrow{i\rightarrow\infty}\infty.\]
We then define the \textup{Hodge volume} of $C$ as the object
\[\mu^{\Hdg}_X(C):=(\mu^{\Hdg}_X(C_n))_n\in D^{b,\wedge}_{\Hdg}.\]
We call a pair $(C,\cal{S})$ consisting of a subset $C\subseteq X_{\infty}$ and a \textup{finite} decomposition $\cal{S}$ ($C=\sqcup_{i\in\cal{S}}C_i$, $C_i\subseteq X_{\infty}$) \textup{Hodge measurable} if each $C_i$ is a Hodge-good subset. We then let
\[\mu^{\Hdg}_X(C,\cal{S}):=\bigoplus_{i\in\cal{S}}\mu^{\Hdg}_X(C_i).\]
We view a Hodge-good subset $C$ without a prescribed decomposition as a Hodge measurable subset with the trivial decomposition $\cal{S}=\{C\}$.
\end{definition}
As in Proposition~\ref{welldefinedmeasure} for the motivic measure (of Voevodsky motives), we can show that the Hodge measure defined in Definition~\ref{Hdgmeasurable} is well-defined. We do not give the proof here as all the steps are the same.
\subsection{Hodge measurable functions and mixed Hodge integration}
We now work toward constructing a theory of integration for rational motives that takes a smooth $k$-scheme $X$ along with an effective divisor $D\subset X$ to an element of $\cal{M}^{\Hdg}(k)$. For this, we make the following definition which is the analogue of Definition~\ref{measurablefunction}.
\begin{definition}[Hodge measurable functions]\label{Hdgmeasurablefunction}
A function $F:X_{\infty}\rightarrow\mathbb{N}_{\geq 0}\cup\left\{\infty\right\}$ is \textup{Hodge measurable} if for each $s\in\mathbb{N}_{\geq 0}$, $F^{-1}(s)$ is Hodge measurable, and $F^{-1}(\infty)$ has Hodge measure $0$.
\end{definition}
\begin{definition}[Mixed Hodge integration]
For Hodge measurable functions $F:X_{\infty}\rightarrow\mathbb{N}_{\geq 0}\cup\{\infty\}$, we define its Hodge integral as
\[\int_{X_{\infty}}\mathbb{Q}(F)[2F]d\mu^{\Hdg}_X:=\left[\oplus_{k=0}^{\infty}\mu^{\Hdg}_X(F^{-1}(k))(k)[2k]\right]\in\cal{M}^{\Hdg}(k).\]
\end{definition}
As before, we prove the following transformation rule for proper birational morphisms of smooth $k$-schemes.
\begin{theorem}[Hodge transformation rule]\label{Hdgchangeofvar}
Suppose $f:X\rightarrow Y$ is a proper birational morphism of smooth $k$-varieties with $K_{X/Y}$ its relative canonical divisor, and let $D\subset Y$ be an effective divisor. Then
\[\int_{Y_{\infty}}\mathbb{Q}(\ord_D)[2\ord_D]d\mu^{\Hdg}_Y=\int_{X_{\infty}}\mathbb{Q}(\ord_{f^{-1}D+K_{X/Y}})[2\ord_{f^{-1}D+K_{X/Y}}]d\mu^{\Hdg}_X.\]
\end{theorem}
\begin{proof}
The proof is exactly as in that of Theorem~\ref{changeofvar}, and we do not repeat the proof here.
\end{proof}
\section{Applications}\label{apps}
Now that we have constructed our theories of integration, we apply them to obtain concrete results regarding K-equivalent varieties. We answer a few questions related to the number theory and geometry of K-equivalent varieties that were inaccessible to classical motivic integration. Using the theory of integration based on the existence of a motivic $t$-structure, we can connect the existence of a motivic $t$-structure to some other conjectures related to $D$- and K-equivalent varieties.
\subsection{Consequences of integration via the slice filtration}
In this subsection, we prove the following theorem using our theory of motivic integration via the slice filtration.
\begin{theorem}\label{mainint}
Suppose $X$ and $Y$ K-equivalent smooth $k$-varieties. Then in $K_0(\DM_{\textup{sl}}^{\textup{eff}}(k;\mathbb{Z}[1/p]))$, $[M(X)]=[M(Y)]$. If the injectivity conjecture hold for the $\mathbb{Z}[1/p]$-algebra $R$, $[M(X)]=[M(Y)]$ in $K_0(\DM_{\gm}^{\textup{eff}}(k;R))$.
\end{theorem}
\begin{proof}
Since $X$ and $Y$ are K-equivalent, there is a smooth $k$-variety $Z$ as well as proper birational morphisms $f:Z\rightarrow X$ and $g:Z\rightarrow Y$ such that $f^*\omega_X\simeq g^*\omega_Y$:
\[\xymatrix{& Z\ar@{->}_{f}[dl] \ar@{->}^{g}[dr] & \\ X \ar@{.>}[rr] && Y.}\]
Applying the transformation rule to each of $f$ and $g$, we obtain
\[[M(X)]=\int_{X_{\infty}}\mathbf{1}_k(\ord_{\emptyset})[2\ord_{\emptyset}]d\mu_X=\int_{Z_{\infty}}\mathbf{1}_k(\ord_{K_{Z/X}})[2\ord_{K_{Z/X}}]d\mu_Z\]
and
\[[M(Y)]=\int_{Y_{\infty}}\mathbf{1}_k(\ord_{\emptyset})[2\ord_{\emptyset}]d\mu_Y=\int_{Z_{\infty}}\mathbf{1}_k(\ord_{K_{Z/Y}})[2\ord_{K_{Z/Y}}]d\mu_Z.\]
Since $X$ and $Y$ are K-equivalent,
\[K_{Z/X}=K_Z-f^*K_X=K_Z-g^*K_Y=K_{Z/Y},\]
and so
\[[M(X)]=\int_{Z_{\infty}}\mathbf{1}_k(\ord_{K_{Z/X}})[2\ord_{K_{Z/X}}]d\mu_Z=\int_{Z_{\infty}}\mathbf{1}_k(\ord_{K_{Z/Y}})[2\ord_{K_{Z/Y}}]d\mu_Z=[M(Y)]\]
in $\cal{M}(k;\mathbb{Z}[1/p])$. Consequently, by Proposition~\ref{injectivity} $[M(X)]=[M(Y)]$ in $K_0(\DM_{\textup{sl}}^{\textup{eff}}(k;\mathbb{Z}[1/p]))$, as required.
\end{proof}
When $X$ and $Y$ are smooth projective $k$-varieties, we have the following stronger result.
\begin{theorem}\label{mainint2}
Suppose $X$ and $Y$ are K-equivalent smooth projective $k$-varieties. If the injectivity conjecture is true for the $\mathbb{Z}[1/p]$-algebra $R$, then there is an effective Chow motive $P\in\Chow^{\textup{eff}}(k;R)$ such that $M(X)\oplus P\simeq M(Y)\oplus P$ in $\Chow^{\textup{eff}}(k;R)^{op}\hookrightarrow\DM_{\textup{gm}}^{\textup{eff}}(k;R)$.
\end{theorem}
\begin{proof}
By Proposition 3.1.3 of Bondarko's paper \cite{Bondarko} (characteristic $p>0$) and Theorem 6.4.2 of Bondarko's \cite{Bondarko2} (characteristic zero), we know that the inclusion 
\[\Chow^{\textup{eff}}(k;\mathbb{Z}[1/p])^{op}\hookrightarrow\DM_{\textup{gm}}^{\textup{eff}}(k;\mathbb{Z}[1/p])\]
induces an isomorphism after taking $K_0$. As a consequence of Theorem~\ref{mainint} and the fact that $X$ and $Y$ are smooth projective, $[M(X)]=[M(Y)]$ in $\Chow^{\textup{eff}}(k;R)$ at least if the injectivity conjecture is true. From this it follows that there is an effective Chow motive $P\in\Chow^{\textup{eff}}(k;\mathbb{Z}[1/p])$ such that $M(X)\oplus P\simeq M(Y)\oplus P$ in $\Chow^{\textup{eff}}(k;R)^{op}\hookrightarrow\DM_{\textup{gm}}^{\textup{eff}}(k;\mathbb{Z}[1/p])$.
\end{proof}
This gives a conditional partial answer to the following conjecture. Wang conjectured it for complex varieties, while we conjecture it in general \cite{Wang2}.
\begin{conjecture}\label{wang2}
If $X$ and $Y$ are K-equivalent smooth projective $k$-varieties, then $M(X)\simeq M(Y)$ (integrally), that is, they have equivalent integral Chow motives.
\end{conjecture}
Birational compact hyperk\"ahler manifolds have isomorphic Hodge structures. In particular, they have the same Hodge and Betti numbers. Furthermore, they have isomorphic integral singular cohomology \textit{rings}. However, the isomorphism between the integral singular cohomology \textit{rings} is not only the result of them being K-equivalent smooth projective complex varieties. It is not in general true that K-equivalent smooth projective complex varieties have isomorphic integral cohomology \textit{rings}. For example, Nam-Hoon Lee and Keiji Oguiso have jointly constructed birational Calabi-Yau complex threefolds with non-isomorphic integral singular cohomology \textit{rings} \cite{LeeOguiso}. It has been conjectured that if we take into account quantum corrections, then the (quantum) cohomology rings are isomorphic, that is, K-equivalent smooth projective complex varieties have isomorphic quantum cohomology rings in the extended K\"ahler moduli space. In the hyperk\"ahler case, the usual cup product coincides with its quantum-corrected cup product which gives the isomorphism of singular cohomology rings. This begs the question of how can one bring in the data of quantum corrections into our theory of motivic integration, if possible? Admittedly, this is a very open-ended question.\\
\\
In three dimensions, Reid and Mori's classification of three dimensional singularties has allowed Koll\'ar and Mori to completely understand the relation between birational minimal models. Via a complete classification of three dimensional flops and flips, it can be shown that three dimensional birational minimal models have isomorphic integral singular (and intersection) cohomology groups and mixed (and intersection pure) Hodge structures. These integral results point to the possibility that the above conjecture may be valid.\\
\\
As a small corollary to Theorem~\ref{mainint2}, let us mention the following which is a shadow of the proof underlying Batyrev's proof of the equality of Betti numbers of birational Calabi-Yau complex varieties.
\begin{corollary}\label{Cornumericalrational}
If $X$ and $Y$ are K-equivalent smooth projective $k$-varieties and the injectivity conjecture is true, then their rational numerical motives are isomorphic.
\end{corollary}
\begin{proof}
The proof is a simple consequence of Theorem~\ref{mainint2} and the fact that rational numerical motives form a semisimple abelian category, a theorem due to Jannsen \cite{Jannsen}.
\end{proof}
\begin{remark}\textup{We remark that if $k$ admits resolution of singularities, then birational smooth projective Calabi-Yau $k$-varieties are K-equivalent. Consequently, conditional on the injectivity conjecture, the above are true for such birational varieties as well. The same corollary is true more generally for birational smooth projective $k$-varieties with nef canonical bundles.}
\end{remark}
An advantage of obtaining motivic results is that through realization functors we can deduce results of interest to geometers and number theorists. See the beginning of Section~\ref{rationalintegration} for a description of some of the realization functors. As another corollary of Theorem~\ref{mainint2}, we have the following result.
\begin{theorem}\label{singulartheorem}
Suppose the injectivity conjecture $I(k;\mathbb{Z})$ is true for $k$ of characteristic zero. If $X$ and $Y$ are K-equivalent smooth projective $k$-varieties with $\sigma:k\hookrightarrow\mathbb{C}$ any embedding, then $H^*(X^{an};\mathbb{Z})\simeq H^*(Y^{an};\mathbb{Z})$ as graded abelian \textup{groups}.
\end{theorem}
\begin{proof}
Indeed, there is the Betti realization functor $B_{\sigma}:\DM_{\textup{gm}}^{\textup{eff}}(k;\mathbb{Z})^{op}\rightarrow D^b(\mathbb{Z})$ into the bounded derived category of finitely generated abelian groups. This is given by sending $M(X)$ to the graded abelian group $\oplus_nH^n(X^{an};\mathbb{Z})[-n]\in D^b(\mathbb{Z})$. Applying Theorem~\ref{mainint2}, we deduce that, conditional on the injectivity conjecture,
\[H^*(X^{an};\mathbb{Z})\oplus H^*(B_{\sigma}(P))\simeq H^*(Y^{an};\mathbb{Z})\oplus H^*(B_{\sigma}(P))\] 
as finitely generated \textit{graded} abelian groups. Since we are landing in finitely generated abelian groups, the classification of finitely generated abelian groups allows us to cancel $H^*(B_{\sigma}(P))$ from both sides to obtain $H^*(X^{an};\mathbb{Z})\simeq H^*(Y^{an};\mathbb{Z})$ as abelian groups. Using weights, we obtain the isomorphism as graded abelian groups.
\end{proof}
\begin{remark}
\textup{In an unpublished note, Chenyang Xu and Qizheng Yin have proved the unconditional version of this theorem for K-equivalent complex smooth projective varieties using classical motivic integration and a motivic measure found in Section 3.3 of Gillet-Soul\'e \cite{GilSoul}. For birational smooth projective Calabi-Yau complex manifolds $X$ and $Y$, Mark McLean has claimed a proof of the unconditional version of the above theorem by identifying the quantum cohomologies of $X$ and $Y$ up to tensoring with a suitable Novikov ring \cite{Mclean}. His proof uses Hamiltonian Floer cohomology.}
\end{remark}
\begin{remark}
\textup{The cohomology rings need not be isomorphic. In fact, Nam-Hoon Lee and Keiji Oguiso have jointly constructed birational Calabi-Yau complex threefolds with non-isomorphic integral singular cohomology \textup{rings} \cite{LeeOguiso}.} 
\end{remark}
From Theorem~\ref{mainint}, if the injectivity conjecture $I(k;\mathbb{Z}[1/p])$ is true, then $X$ and $Y$ that are K-equivalent smooth $k$-varieties satisfy $[M(X)]=[M(Y)]$ in $K_0(\DM_{\textup{gm}}(k;\mathbb{Z}[1/p]))$. In particular, if $k=\mathbb{F}_q$, then taking traces of the Frobenius we would obtain that that they have the same zeta functions: $\zeta_X(t)=\zeta_Y(t)$. In fact, we show the following.
\begin{theorem}\label{ladic}
Suppose the injectivity conjecture $I(k;\mathbb{Z}[1/p])$ is true. If $X$ and $Y$ are K-equivalent smooth projective $k$-varieties, then they have the same $\ell$-adic Galois representations (up to semi-simplification). In the case $k=\mathbb{F}_q$, two K-equivalent smooth $\mathbb{F}_q$-varieties have the same zeta functions.
\end{theorem}
If $\mathbb{F}_q$ admits resolution of singularities, then all this is true for $X$ and $Y$ that are birationally equivalent smooth projective Calabi-Yau $\mathbb{F}_q$-varieties. Note that for the equality of zeta functions, we do not require the varieties to be projective.
\begin{remark}
\textup{Note that over finite fields and number fields, the slice filtration is expected to preserve geometricity, at least for $R=\mathbb{Q}$. Consequently, it is expected that $\DM_{\textup{gm}}^{\textup{eff}}(k;\mathbb{Q})=\DM_{\textup{sl}}^{\textup{eff}}(k;\mathbb{Q})$.}
\end{remark}
\begin{remark}
\textup{In the next section, we give an unconditional argument for the isomorphism of \textit{rational} $\ell$-adic Galois representations (up to semisimplification) of two K-equivalent smooth projective $\mathbb{F}_q$-varieties by using our mixed $\ell$-adic integration. Using mixed $\ell$-adic integration, we also unconditionally prove the equality of zeta functions of K-equivalent smooth $\mathbb{F}_q$-varieties, without the projectivity condition.}
\end{remark}
Using the Hodge realization functor, we deduce the following corollary of Theorem~\ref{mainint}.
\begin{corollary}\label{Hodge} Suppose the injectivity conjecture $I(\mathbb{C};\mathbb{Z})$ is true. If $X$ and $Y$ are K-equivalent complex varieties, then they have the same classes of \textup{integral} mixed Hodge structures in $K_0(\textup{MHS}_{\mathbb{Z}})$. Furthermore, two K-equivalent smooth projective complex varieties would have the same rational Hodge structures.
\end{corollary}
\begin{proof}
The first statement is a direct consequence of Theorem~\ref{mainint}, while the second part is a consequence of the semisimplicity of the abelian category of polarizable rational mixed Hodge structures.
\end{proof}
\begin{remark}
\textup{In Section~\ref{Hodge}, we give an unconditional proof of the last statement about rational Hodge structure using our mixed Hodge integration. This recovers the well-known result (Theorem~\ref{Konttheorem}) of Kontsevich \cite{Kontsevich} mentioned earlier (and that of Batyrev \cite{Batyrev} as a consequence of the decomposition theorem).}
\end{remark}
Finally, recall the Definition~\ref{KScategory} of Krull-Schmidt categories and the following discussion in the introduction. A direct consequence of Theorem~\ref{mainint2} is the following.
\begin{corollary}\label{Kimuramainint2}Suppose $\Chow^{\textup{eff}}(k;R)$ is a Krull-Schmidt category, $k$ and $R$ as in our conventions set at the beginning of this introduction, and that the injectivity conjecture $I(k;R)$ is true. Then for $X$ and $Y$ K-equivalent smooth projective $k$-varieties, the Chow motives of $X$ and $Y$ in $\Chow^{\textup{eff}}(k;R)$ are equivalent.
\end{corollary}
\subsection{Consequences of motivic integration via the motivic $t$-structure}\label{mot}
In this subsection, we show how the existence of a motivic $t$-structure on $\DM_{\textup{gm}}(k;\mathbb{Q})$ with the expected properties can be used along with our motivic integration via the motivic $t$-structure to show the equivalence of rational Voevodsky motives of K-equivalent smooth projective $k$-varieties. First a proposition.
\begin{proposition}\label{propmain}
Suppose the motivic $t$-structure Conjecture~\ref{tstructureconjecture} is true for $\DM_{\textup{gm}}(k;\mathbb{Q})$. Then for $X$ and $Y$ smooth projective $k$-varieties such that $[M(X)_{\mathbb{Q}}]=[M(Y)_{\mathbb{Q}}]$ in $K_0(\DM_{\textup{gm}}(k;\mathbb{Q}))$, $M(X)_{\mathbb{Q}}\simeq M(Y)_{\mathbb{Q}}$.
\end{proposition}
\begin{proof}
Then the inclusions $\textup{Num}(k;\mathbb{Q})\hookrightarrow \DM_{\textup{gm}}(k;\mathbb{Q})^{\heartsuit}$ and $\DM_{\textup{gm}}(k;\mathbb{Q})^{\heartsuit}\hookrightarrow \DM_{\textup{gm}}(k;\mathbb{Q})$ induce isomorphisms
\[K_0(\textup{Num}(k;\mathbb{Q}))\rightarrow K_0(\DM_{\textup{gm}}(k;\mathbb{Q})^{\heartsuit})\] 
and
\[K_0(\DM_{\textup{gm}}(k;\mathbb{Q})^{\heartsuit})\rightarrow K_0(\DM_{\textup{gm}}(k;\mathbb{Q})).\]
The latter isomorphism follows from the theorem of the heart \cite{Barwick}. The second morphism has inverse given by $[M]\mapsto\sum_a(-1)^a[^{\mu}H^aM]$. Since $[M(X)_{\mathbb{Q}}]=[M(Y)_{\mathbb{Q}}]$ in $K_0(\DM_{\textup{gm}}(k;\mathbb{Q}))$, we obtain that
\begingroup
    \fontsize{9.5pt}{11pt}\selectfont

$$\left[\left(\oplus_{a\ \textup{even}}\ ^{\mu}H^aM(X)_{\mathbb{Q}}\right)\oplus\left(\oplus_{a\ \textup{odd}}\ ^{\mu}H^aM(Y)_{\mathbb{Q}}\right)\right]=\left[\left(\oplus_{a\ \textup{odd}}\ ^{\mu}H^aM(X)_{\mathbb{Q}}\right)\oplus\left(\oplus_{a\ \textup{even}}\ ^{\mu}H^aM(Y)_{\mathbb{Q}}\right)\right]$$
\endgroup
in $K_0(\DM_{\textup{gm}}(k;\mathbb{Q})^{\heartsuit})$. By assumption, for each smooth projective $k$-variety $Z$, $^{\mu}H^aM(Z)$ is semisimple, and so a numerical motive by the assumption that the semisimple part of the heart is the category of rational numerical motives. Therefore, the last equality holds in $K_0(\textup{Num}(k;\mathbb{Q}))$. By Jannsen's theorem \cite{Jannsen}, $\textup{Num}(k;\mathbb{Q})$ is a semisimple abelian category, and so
\begingroup
    \fontsize{9.5pt}{11pt}\selectfont
$$\left(\oplus_{a\ \textup{even}}\ ^{\mu}H^aM(X)_{\mathbb{Q}}\right)\oplus\left(\oplus_{a\ \textup{odd}}\ ^{\mu}H^aM(Y)_{\mathbb{Q}}\right)\simeq\left(\oplus_{a\ \textup{odd}}\ ^{\mu}H^aM(X)_{\mathbb{Q}}\right)\oplus\left(\oplus_{a\ \textup{even}}\ ^{\mu}H^aM(Y)_{\mathbb{Q}}\right)$$
\endgroup
as numerical motives, and so also in $\DM_{\textup{gm}}(k;\mathbb{Q})^{\heartsuit}$. By Proposition 1.7 of \cite{Beilinson}, this forces us to have $\ ^{\mu}H^aM(X)_{\mathbb{Q}}\simeq \ ^{\mu}H^aM(Y)_{\mathbb{Q}}$ for every $a$. By Proposition 1.4 of Beilinson's \cite{Beilinson}, for each smooth projective $k$-variety $Z$, $M(Z)_{\mathbb{Q}}\simeq\oplus_a\ ^{\mu}H^aM(Z)_{\mathbb{Q}}[-a]$. Consequently, $M(X)_{\mathbb{Q}}\simeq M(Y)_{\mathbb{Q}}$.
\end{proof}
\begin{remark}
\textup{We do not need to actually identify the semisimple part of the heart with rational numerical motives to make the above proof work. We just need to have that the motivic cohomologies of smooth projective varieties are semisimple. Therefore, there is no need to invoke Jannsen's theorem.}
\end{remark}
Using this proposition and motivic integration via the motivic $t$-structure, we can prove the following theorem.
\begin{theorem}\label{motivicmain}
Suppose the motivic $t$-structure conjecture is true for $\DM_{\textup{gm}}(k;\mathbb{Q})$. Then if $X$ and $Y$ are K-equivalent smooth projective $k$-varieties, $M(X)_{\mathbb{Q}}\simeq M(Y)_{\mathbb{Q}}$.
\end{theorem}
\begin{proof}
Replacing integration with respect to $\mu_X$ in the proof of Theorem~\ref{mainint} with $\mu_X^{\mot}$, we obtain that $[M(X)_{\mathbb{Q}}]=[M(Y)_{\mathbb{Q}}]$ in $\cal{M}^{\mot}(k)$. Using Proposition~\ref{motinjectivity}, we obtain that $[M(X)_{\mathbb{Q}}]=[M(Y)_{\mathbb{Q}}]$ in $K_0(\DM_{\gm,\geq 0})$. The injection $\DM_{\gm,\geq 0}\hookrightarrow\DM_{\gm}(k;\mathbb{Q})$ gives us a map on $K_0$, from which we obtain the equality $[M(X)_{\mathbb{Q}}]=[M(Y)_{\mathbb{Q}}]$ in $K_0(\DM_{\gm}(k;\mathbb{Q}))$. Proposition~\ref{propmain} gives us $M(X)_{\mathbb{Q}}\simeq M(Y)_{\mathbb{Q}}$, as required.
\end{proof}
All of these suggest the following conjecture, which is a weaker version of Conjecture~\ref{wang2} above.
\begin{conjecture}\label{wang}
If $X$ and $Y$ are K-equivalent smooth projective $k$-varieties, then $M(X)_{\mathbb{Q}}\simeq M(Y)_{\mathbb{Q}}$.
\end{conjecture}
For rational coefficients, it is highly unexpected that the motivic $t$-structure conjecture above is false, and so Theorem~\ref{motivicmain} suggests that it is very likely that this last conjecture is true. The goal for the future is to refine the construction of motivic integration in order to unconditionally prove this last conjecture and possibly its more general integral version.\\
\\
The results in this subsection are also related to D-equivalence and Orlov's conjecture. Consider the following two important conjectures in noncommutative geometry due to Bondal-Orlov and Orlov.
\begin{conjecture}[Bondal-Orlov \cite{BO}]\label{BondalOrlov}
If $X$ and $Y$ are birationally equivalent smooth projective Calabi-Yau complex varieties, then $\dcat(X)\simeq\dcat(Y)$.
\end{conjecture}
Though true for birational Calabi-Yau varieties of dimension at most 3 and birational symplectic 4-folds, this conjecture is vastly open. For more detail, see the preprint of J.Wierzba~\cite{Wierzba}.\\
\\
Note that there is a functor
\[\DM_{\textup{gm}}(k;\mathbb{Q})\rightarrow\textup{KMM}(k)_{\mathbb{Q}}\]
sending $M(X)_{\mathbb{Q}}$ to $NM(X)\in\textup{KMM}(k)_{\mathbb{Q}}$, where $\textup{KMM}(k)_{\mathbb{Q}}$ is Kontsevich's category of rational noncommutative motives. See \cite{Tabuada2} for details. We thus obtain the following corollary of Theorem~\ref{motivicmain}.
\begin{corollary}\label{noncommutativemotivicmain}
Suppose the motivic $t$-structure conjecture is true for $\DM_{\textup{gm}}(k;\mathbb{Q})$. Then if $X$ and $Y$ are K-equivalent smooth projective $k$-varieties, $NM(X)\simeq NM(Y)$ in $\textup{KMM(k)}_{\mathbb{Q}}$.
\end{corollary}
We know this to be true unconditionally if we replace K-equivalence with D-equivalence (combine Proposition 1 of \cite{BO} with the functor in \cite{Tabuada} or \cite{Tabuada2}). In light of Conjecture~\ref{BondalOrlov}, it is still open if K-equivalence, at least in the setting of birational smooth projective complex varieties, implies D-equivalence. What this last corollary states is that if the expected motivic $t$-structure on rational geometric Voevodsky motives exists, then K-equivalence implies the equivalence of noncommutative motives in the sense of Kontsevich. Consequently, many of the noncommutative cohomology theories should agree for K-equivalent smooth projective varieties.\\
\\
The second conjecture concerns the relation between the bounded derived category of coherent sheaves on smooth projective complex varieties and their rational motives.
\begin{conjecture}[Orlov \cite{Orlov}]\label{Orlov}
If $X$ and $Y$ are smooth projective complex varieties such that $\dcat(X)\simeq\dcat(Y)$ as $\mathbb{C}$-linear triangulated categories, then $M(X)_{\mathbb{Q}}\simeq M(Y)_{\mathbb{Q}}$ in $\DM_{\textup{gm}}(\mathbb{C},\mathbb{Q})$.
\end{conjecture}
Equivalences of such bounded derived categories are given by Fourier-Mukai transforms, and under certain conditions on the kernel of this transform, the latter conjecture is true \cite{Orlov}. However, the unconditional conjecture is still vastly open.
\begin{remark}\textup{Conjecture~\ref{Orlov} is false if we require the equivalence of \textit{integral} motives in the conclusion. Indeed, there are D-equivalent Calabi-Yau threefolds with non-isomorphic third integral singular cohomology groups \cite{Addington}. Furthermore, the converse is also false. Indeed, take $X$ to be $\mathbb{P}_k^2$ blown-up at a point and $Y$ to be $\mathbb{P}_k^1\times_k\mathbb{P}^1_k$. Both have motives $\mathbf{1}_k\oplus\mathbf{1}_k(1)[2]^{\oplus 2}\oplus\mathbf{1}_k(2)[4]$ and are both Fano. If they were D-equivalent, then they would be isomorphic (use result of Bondal and Orlov \cite{BO02}), which is not true.}
\end{remark}
\begin{remark}
\textup{It is a result of Balmer \cite{Balmer} in $\otimes$-triangular geometry that quasi-compact and quasi-separated schemes $X$ can be recovered from their $\otimes$-triangulated categories of perfect complexes $D_{\textup{perf}}(X)$. This is not true if we forget the $\otimes$-structure, as shown by Mukai that there are abelian varieties $A$ such that $A\not\simeq\widehat{A}:=\textup{Pic}^0(A)$ but are D-equivalent. Therefore, Orlov's conjecture states that if we forget the $\otimes$-structure on $D_{\textup{perf}}(X)\simeq \dcat(X)$, for $X$ any smooth projective complex variety, we can at least recover the rational (Chow) motive of $X$. From another perspective, we have Gabriel's theorem stating that if $\textup{Coh}(X)\simeq\textup{Coh}(Y)$ as abelian categories, then $X$ and $Y$ are isomorphic \cite{Gabriel}. Therefore, another perhaps more enlightening interpretation of Orlov's conjecture is that considering derived equivalence on the left hand side forces us to consider equivalence in Voevodsky's category of rational geometric motives, conjecturally a bounded derived category of mixed motives.}
\end{remark}
In any case, the combination of the above two conjectures suggests that if $X$ and $Y$ are birational smooth projective Calabi-Yau complex varieties, then $M(X)_{\mathbb{Q}}\simeq M(Y)_{\mathbb{Q}}$, giving more evidence for the validity of the rational version of the conjecture of Wang.\\
\\
A priori, there is no obvious connection between the existence of a motivic $t$-structure on $\DM_{\textup{gm}}(\mathbb{C};\mathbb{Q})$ and conjecture~\ref{Orlov} of Orlov. However, a theorem of Kawamata states the following.
\begin{theorem}[Part (2) of Theorem 2.3 of \cite{Kawamata}] Suppose $X$ and $Y$ are D-equivalent smooth projective complex varieties such that $X$ is of general type ($\kappa(X)=\dim X$, i.e. maximal Kodaira dimension) or $\kappa(X,-K_X)=\dim X$. Then they are K-equivalent.\end{theorem}
In combination with Theorem~\ref{mainint2} and Theorem~\ref{motivicmain}, we obtain the following theorems.
\begin{theorem}\label{motivicOrlov1}
Suppose the injectivity conjecture $I(\mathbb{C};R)$ is true, and suppose $X$ and $Y$ are smooth projective complex varieties such that $\kappa(X)=\dim X$ or $\kappa(X,-K_X)=\dim X$. Then $\dcat(X)\simeq\dcat(Y)\implies M(X)\oplus P\simeq M(Y)\oplus P$ for some Chow motive $P$ (with coefficients in $R$).
\end{theorem}
We also obtain the following consequence of the exitence of the motivic $t$-structure.
\begin{theorem}\label{motivicOrlov}
Suppose the motivic $t$-structure conjecture is true for $\DM_{\textup{gm}}(\mathbb{C};\mathbb{Q})$. Then for $X$ and $Y$ smooth projective complex varieties such that $\kappa(X)=\dim X$ or $\kappa(X,-K_X)=\dim X$, $\dcat(X)\simeq\dcat(Y)\implies M(X)_{\mathbb{Q}}\simeq M(Y)_{\mathbb{Q}}$.
\end{theorem}
In other words, if the motivic $t$-structure conjecture is true, then Orlov's Conjecture~\ref{Orlov} is true for smooth projective complex varieties with $\kappa(X)=\dim X$ or $\kappa(X,-K_X)=\dim X$.
\begin{remark}\textup{We remark that when $X$ and $Y$ are smooth projective complex varieties such that $X$ has ample or anti-ample canonical bundle, then a $\mathbb{C}$-linear equivalence of bounded derived categories of coherent sheaves implies an isomorphism of $X$ and $Y$. This is a theorem due to Bondal and Orlov \cite{BO2}.}
\end{remark}
In a good sense, the vast majority of smooth projective complex varieties are of general type. For example, in the moduli space of curves, the space of curves of Kodaira dimension $-\infty$ (genus $0$) is a point, the space of curves of Kodaira dimension $0$ (genus $1$) is $1$-dimensional, while the space of curves of genus $g\geq 2$ (general type) has dimension $3g-3$. Also, a hypersurface of degree $d$ in $\mathbb{P}^n$ is of general type if and only if $d>n+1$, and so \textit{most} hypersurfaces are of general type. Consequently, conditional on the existence of the expected motivic $t$-structure on $\DM_{\textup{gm}}(\mathbb{C};\mathbb{Q})$, we have established Conjecture~\ref{Orlov} for the vast majority of smooth projective complex varieties. Most likely, any possible counterexample to Orlov's conjecture can only be found among Calabi-Yau varieties. If we develop a more refined theory of integration for Voevodsky motives, we may be able to unconditionally prove that K-equivalent smooth projective varieties have equivalent (Chow) motives, in which case we will make unconditional the theorems and corollaries above that are conditional on the existence of a motivic $t$-structure, the validity of the injectivity conjectures, or the Krull-Schmidt property. This technique of passing through K-equivalence to show that D-equivalence implies equivalence of rational motives will not work for all smooth projective varieties since D-equivalence does not in general imply K-equivalence. Uehara has provided an example of two birational smooth projective complex varieties that are D-equivalent but not K-equivalent \cite{Uehara}.
\subsection{Consequences of mixed Hodge and $\ell$-adic integration}\label{appladic}
In this subsection, we apply our theory of mixed $\ell$-adic integration to show that K-equivalent smooth $\mathbb{F}_q$-varieties have isomorphic $\ell$-adic Galois representations up to semisimplification. In particular, they have the same zeta functions. Through this would be a consequence of our previous results if we assume the appropriate injectivity conjecture or the existence of a motivic $t$-structure, mixed $\ell$-adic integration allows us to deduce this result unconditionally.
\begin{theorem}
Suppose $X$ and $Y$ are K-equivalent smooth projective $\mathbb{F}_q$-varieties. Then $H^*_{\et}(X_{\overline{\mathbb{F}}_q};\overline{\mathbb{Q}}_{\ell})^{ss}\simeq H^*_{\et}(Y_{\overline{\mathbb{F}}_q};\overline{\mathbb{Q}}_{\ell})^{ss}$ as \textup{graded} $\ell$-adic Galois representations.
\end{theorem} 
\begin{proof}
Upon replacing integration with respect to $\mu_X$ with mixed $\ell$-adic integration in the proof of Theorem~\ref{mainint}, we can deduce in the same manner that $[H^*_{\et}(X_{\overline{\mathbb{F}}_q};\overline{\mathbb{Q}}_{\ell})]=[H^*_{\et}(Y_{\overline{\mathbb{F}}_q};\overline{\mathbb{Q}}_{\ell}]$ in $\cal{M}^{\ell}(k)$. Using Proposition~\ref{ellinjectivity}, we obtain that $[H^*_{\et}(X_{\overline{\mathbb{F}}_q};\overline{\mathbb{Q}}_{\ell})]=[H^*_{\et}(Y_{\overline{\mathbb{F}}_q};\overline{\mathbb{Q}}_{\ell}]$ in $K_0(D_{m,\geq 0})$, and so also in $K_0(D^b(\mathbb{F}_q;\overline{\mathbb{Q}}_{\ell}))$. We conclude that $H^*_{\et}(X_{\overline{\mathbb{F}}_q};\overline{\mathbb{Q}}_{\ell})^{ss}\simeq H^*_{\et}(Y_{\overline{\mathbb{F}}_q};\overline{\mathbb{Q}}_{\ell})^{ss}$ as $\ell$-adic Galois representations. Using the Weil conjectures, we know that the weights of $H^i_{\et}(X_{\overline{\mathbb{F}}_q};\overline{\mathbb{Q}}_{\ell})$ and $H^i_{\et}(Y_{\overline{\mathbb{F}}_q};\overline{\mathbb{Q}}_{\ell})$ are exactly $i$. Consequently, $H^*_{\et}(X_{\overline{\mathbb{F}}_q};\overline{\mathbb{Q}}_{\ell})^{ss}\simeq H^*_{\et}(Y_{\overline{\mathbb{F}}_q};\overline{\mathbb{Q}}_{\ell})^{ss}$ as \textit{graded} $\ell$-adic Galois representations, as required.
\end{proof}
As a corollary of the equality of classes of the $\ell$-adic Galois representations of K-equivalent smooth $\mathbb{F}_q$-varieties, we obtain the following result.
\begin{corollary}
If $X$ and $Y$ are K-equivalent smooth $\mathbb{F}_q$-varieties, then they have the same zeta functions. Equivalently, for each $n\in\mathbb{N}$, $|X(\mathbb{F}_{q^n})|=|Y(\mathbb{F}_{q^n})|$.
\end{corollary}
Instead of using mixed $\ell$-adic integration, we can use mixed Hodge integration for K-equivalent smooth projective $k$-varieties, $k$ any field of characterstic zero. Doing so gives us the following theorem about $\mathbb{Q}$-Hodge structures.
\begin{theorem}
Suppose $k$ is a characteristic zero field, and $X$ and $Y$ are K-equivalent smooth projective $k$-varieties. Then they have the same $\mathbb{Q}$-Hodge structures.
\end{theorem}
\begin{proof}
Using mixed Hodge integration instead of integration with respect to $\mu_X$ in the proof of Theorem~\ref{mainint}, we obtain that $[H^*(X;\mathbb{Q})]=[H^*(Y;\mathbb{Q})]$ in $\cal{M}^{\Hdg}(k)$. Applying Proposition~\ref{Hdginjectivity}, we obtain the same equality in $K_0(D^b(MHS_{\mathbb{Q}}))$, the Grothendieck group of the bounded derived category of polarizable mixed $\mathbb{Q}$-Hodge structures. Since $MHS_{\mathbb{Q}}$ is a semisimple abelian category, we obtain that $H^*(X;\mathbb{Q})\simeq H^*(Y;\mathbb{Q})$ as polarizable mixed $\mathbb{Q}$-Hodge structures. Note that the fact that the different degree cohomology groups have different weights implies that this isomorphism is even true degreewise. Consequently, $X$ and $Y$ have the same $\mathbb{Q}$-Hodge structures, as required.   
\end{proof}
Note that this recovers Kontsevich's Theorem~\ref{Konttheorem} (as well as Batyrev's theorem \cite{Batyrev}). Again, if we want equality of mixed Hodge numbers, equality of classes in the Grothendieck group of polarizable mixed rational Hodge structures is sufficient. The equality of the classes is even true without projectivity. Using the Deligne-Hodge polynomial, we obtain the following.
\begin{corollary}
Suppose $k$ is a characteristic zero field, and $X$ and $Y$ are K-equivalent smooth $k$-varieties. Then they have the same Deligne-Hodge polynomials.
\end{corollary}

\small{\textsc{Department of Mathematics, Princeton University, Princeton NJ, USA}}\\
\small{\textsc{Department of Mathematics, University of Regensburg, Regensburg, Germany}}\\
\textit{Email address 1:} \texttt{\small{masoud.zargar@ur.de}}\\
\textit{Email address 2:} \texttt{\small{mzargar1225@gmail.com}}
\end{document}